\documentclass[a4paper,10pt,american]{amsart}

\usepackage{amssymb}
\usepackage{dsfont}
\usepackage{stmaryrd}
\usepackage{esint}
\usepackage{changes}

\usepackage{babel}
\usepackage{xcolor}
\usepackage[colorlinks,pdfpagelabels,pdfstartview = FitH,bookmarksopen
= true,bookmarksnumbered = true,linkcolor = blue,plainpages =
false,hypertexnames = false,citecolor = red,pagebackref=false]{hyperref}
\usepackage{enumitem}
\usepackage{changes}
\usepackage{ulem}

\allowdisplaybreaks

\newtheorem{theorem}{Theorem}
\newtheorem{lemma}[theorem]{Lemma}
\newtheorem{proposition}[theorem]{Proposition}
\newtheorem{corollary}[theorem]{Corollary}

\theoremstyle{definition}
\newtheorem{definition}[theorem]{Definition}

\theoremstyle{remark}
\newtheorem{remark}[theorem]{Remark}

\numberwithin{theorem}{section}
\numberwithin{equation}{section}

\newcommand{\N}{\mathds{N}}
\newcommand{\R}{\mathds{R}}

\renewcommand{\L}{\mathcal{L}}

\renewcommand{\b}{\mathfrak{b}}

\renewcommand{\d}{\mathrm{d}}
                  
\newcommand{\dx}{\mathrm{d}x}

\newcommand{\dt}{\mathrm{d}t}
\newcommand{\ds}{\mathrm{d}s}

\newcommand{\eps}{\varepsilon}
\renewcommand{\epsilon}{\varepsilon}
\renewcommand{\rho}{\varrho}

\DeclareMathOperator{\spt}{spt}
\DeclareMathOperator{\diam}{diam}
\DeclareMathOperator{\dist}{dist}

\DeclareMathOperator*{\esssup}{ess\,sup}
\DeclareMathOperator{\Div}{div}
\newcommand{\ca}{\operatorname{cap}}

\newcommand{\wto}{\rightharpoondown}
\newcommand{\wsto}{\overset{\raisebox{-1ex}{\scriptsize $*$}}{\rightharpoondown}}

\def\Xint#1{\mathchoice
    {\XXint\displaystyle\textstyle{#1}}%
    {\XXint\textstyle\scriptstyle{#1}}%
    {\XXint\scriptstyle\scriptscriptstyle{#1}}%
    {\XXint\scriptscriptstyle\scriptscriptstyle{#1}}%
    \!\int}
\def\XXint#1#2#3{\setbox0=\hbox{$#1{#2#3}{\int}$}
    \vcenter{\hbox{$#2#3$}}\kern-0.5\wd0}
\def\bint{\Xint-}
\def\dashint{\Xint{\raise4pt\hbox to7pt{\hrulefill}}}

\def\XXiint#1#2#3{\setbox0=\hbox{$#1{#2#3}{\iint}$}
    \vcenter{\hbox{$#2#3$}}\kern-0.5\wd0}

\newcommand{\power}[2]{\boldsymbol{#1^{\mbox{\unboldmath{\scriptsize$#2$}}}}}


\begin{document}
\title[Doubly nonlinear systems in general noncylindrical domains]{Existence of variational solutions to doubly nonlinear systems in general noncylindrical domains}
\date{\today}

\author[L.~Schätzler]{Leah Schätzler}
\address{Leah Sch\"atzler\\
Department of Mathematics and Systems Analysis, Aalto University\\
P.O.~Box 11100, FI-00076 Aalto, Finland}
\email{ext-leah.schatzler@aalto.fi}

\author[C.~Scheven]{Christoph Scheven}
\address{Christoph Scheven\\
Fakultät für Mathematik, Universität Duisburg-Essen\\
Thea-Leymann-Str.~9, 45127 Essen, Germany}
\email{christoph.scheven@uni-due.de}

\author[J.~Siltakoski]{Jarkko Siltakoski}
\address{Jarkko Siltakoski\\
Department of Mathematics and Statistics, University of Helsinki, P.O.
Box 68 (Gustaf Hallstr\"omin katu 2b), Finland}
\email{jarkko.siltakoski@helsinki.fi}

\author[C.~Stanko]{Calvin Stanko}
\address{Calvin Stanko\\
Fachbereich Mathematik, Paris-Lodron Universit\"at Salzburg\\
Hellbrunner Str.~34, 5020 Salzburg, Austria}
\email{calvin.stanko@plus.ac.at}

\begin{abstract}
We consider the Cauchy--Dirichlet problem to doubly nonlinear systems of the form
\begin{align*}
    \partial_t \big( |u|^{q-1}u \big) - \operatorname{div} \big( D_\xi f(x,u,Du) \big) = - D_u f(x,u,Du)
\end{align*}
with $q \in (0, \infty)$ in a bounded noncylindrical domain $E \subset \mathds{R}^{n+1}$.
Further, we suppose that $x \mapsto f(x,u,\xi)$ is integrable, that $(u,\xi) \mapsto f(x,u,\xi)$ is convex, and that $f$ satisfies a $p$-growth and -coercivity condition for some $p>\max \left\{ 1,\frac{n(q+1)}{n+q+1} \right\}$.
Merely assuming that $\L^{n+1}(\partial E) = 0$, we prove the existence of variational solutions $u \in L^\infty \big( 0,T;L^{q+1}(E,\mathds{R}^N) \big)$.
If $E$ does not shrink too fast, we show that for the solution $u$ constructed in the first step, $\vert u \vert^{q-1}u$ admits a distributional time derivative.
Moreover, under suitable conditions on $E$ and the stricter lower bound $p \geq \frac{(n+1)(q+1)}{n+q+1}$, $u$ is continuous with respect to time. 
\end{abstract}


\maketitle

\section{Introduction}
In the present paper we are concerned with the Cauchy--Dirichlet problem to parabolic systems in bounded noncylindrical domains, i.e.~subsets $E$ of space-time $\R^n \times [0,\infty)$ that do not necessarily have the form $\Omega \times [0,T]$ with a fixed domain $\Omega \subset \R^n$.
Instead, the spatial domain is allowed to evolve over time.
For a precise description of our setting, we refer to the beginning of Section~\ref{Sec:Setting_and_main_results}.
The prototype of the considered doubly nonlinear systems is given by 
\begin{align}\label{prototype equation}
    \partial_t(\vert u \vert^{q-1}u) - \Div(\vert Du \vert^{p-2}Du) = 0
\end{align}
with exponents $q\in(0, \infty)$ and $p \in (1,\infty)$.
Well-known special cases include the heat equation ($q=1$ and $p=2$), the parabolic $p$-Laplace equation ($q=1$ and arbitrary $p \in (1, \infty)$), and the porous medium equation ($p=2$ and arbitrary $q \in (0, \infty)$).
Although at first glance systems of type \eqref{prototype equation} look similar in all parameter regimes, different combinations of exponents $p$ and $q$ describe different physics.
Indeed, in the slow diffusion regime $p>q+1$, solutions can have compact support and perturbations propagate with finite speed, whereas the opposite is true in the fast diffusion case $p \leq q+1$, cf.~\cite{Ivanov}.
Applications related to \eqref{prototype equation} include the flow of a non-Newtonian fluid through a porous medium, cf.~\cite{Kalashnikov, Vazquez-1,BDGLS_23}, dynamics of shallow water flows \cite{Alonso-etal} and glaciers \cite{CDDSV_03,Mahaffy}, and friction-dominated flow in a gas network \cite{Bamberger-etal, Leugering-Mophou}.
Moreover, various types of evolution equations in noncylindrical domains have been used to model phenomena appearing in soil science~\cite{Showalter_Walkington}, filtration processes~\cite{OKZ_58} and fluid dynamics, cf.~\cite{Aronson, Barenblatt1, Barenblatt2, Ladyzenskaja}, in particular the Hele-shaw flow \cite{Shelley} and the drop-splash problem \cite{Krechetnikov}, 
as well as the motion of incompressible fluids \cite{FZ_2010} and the interplay between internal solitary waves \cite{BBFSV08}. Well known examples are the Stefan problem \cite{Stefan}, and the Skorokhod problem \cite{Skorokhod}. Further problems formulated in time-dependent domains also arise in quantum mechanics \cite{BBFSV13, BL82, DCDL_11}, as well as in mathematical biology, where reaction-diffusion equations are utilized to model pattern formation, cf.~\cite{Crampin_Gaffney_Maini, Murray}. For contributions on stochastic parabolic equations in noncylindrical domains, see~\cite{LS_84, Burdzy-Chen-Sylvester-1, Burdzy-Chen-Sylvester-2}.
For an extensive overview, we refer to \cite{Knobloch_Krechetnikov}.

At this stage, a brief overview of previous existence results is in order.
First, in the case of cylindrical domains, the existence of variational or weak solutions to doubly nonlinear equations was shown in \cite{Alt_Luckhaus, Bernis, Grange_Mignot} for bounded and unbounded domains.
In particular, Alt and Luckhaus \cite{Alt_Luckhaus} considered quasilinear systems
\begin{align*}
    \partial_t b(u) - \Div( a(b(u), Du)) = f(b(u)),
\end{align*}
where $b$ is the gradient of a convex $C^{1}$-function satisfying $b(0)=0$, and $a(b(\zeta), \xi)$ is continuous with respect to $(\zeta, \xi)$, and for $1<p<\infty$ it satisfies the ellipticity condition
\begin{align*}
    (a(b(\zeta), \xi) -  a(b(\zeta), \eta))\cdot(\xi - \eta) \ge c \vert \xi - \eta \vert^{p}
\end{align*}
with a constant $c>0$, as well as a standard $(p-1)$-growth condition.
The proof relies on the Galerkin method.
Further, Ivanov, Mkrtychyan, and Jäger \cite{Ivanov,Ivanov_Mkrtychyan, Ivanov_Mkrtychyan_Jaeger} proved existence for doubly nonlinear equations via regularization of the considered equations.
Subsequently, Akagi and Stefanelli \cite{Akagi_Stefanelli} used the Weighted Energy Dissipation Functional approach, also known as elliptic regularization, to deal with the Cauchy--Dirichlet problem to equations of the form
\begin{align*}
    \partial_t b(u) - \Div(a(Du)) \ni f,
\end{align*}
where $b \subset \mathds{R} \times \mathds{R}$ and $a \subset \mathds{R}^{n} \times \mathds{R}^{n}$ are maximal monotone graphs fulfilling polynomial growth conditions.
Moreover, B\"ogelein, Duzaar, Marcellini and Scheven \cite{BDMS18-2} showed the existence of non-negative variational solutions to the Cauchy--Dirichlet problem with time-independent boundary values to general doubly nonlinear equations of the form
\begin{align*}
    \partial_t b(u) - \Div( D_\xi f(x,u,Du) ) = - D_u f(x,u,Du),
\end{align*}
where $b$ is an increasing, piecewise $C^{1}$ function satisfying a certain growth condition. Furthermore, $f$ is convex and $p$-coercive, but does not need to satisfy a growth assumption from above.
The proof relies on a nonlinear variant of the method of minimizing movements.
Assuming that $f$ additionally fulfills a standard $p$-growth condition, time-dependent boundary values and obstacle functions were treated in \cite{Schaetzler,Schaetzler-obstacle-deg,Schaetzler-obstacle-sing}.

Now, we continue with noncylindrical domains.
We point out that in general the case of domains that are allowed to shrink in time is more difficult to handle than the case of nondecreasing domains.
The reason is that the boundary value problem can become overdetermined if the domain decreases too fast.
Nevertheless, the literature regarding the existence theory for linear equations in noncylindrical domains is extensive, cf.~\cite{Acquistapace-Terreni,Bonaccorsi_Guatteri,Brown-Hu-Lieberman,Cannarsa-DaPrato-Zolesio,Gianazza_Savare,Lions,Lumer-Schnaubelt,Savare}.
Moreover, several nonlinear PDEs have recently been considered in time-dependent domains with a smooth boundary, cf.~\cite{Alphonse-et-al, Alphonse_Elliott, Lan_Son_Tang_Thuy, Nakao-1, Nakao-2}.

Much less is known with respect to nonlinear parabolic equations in nonsmooth time-varying domains.
In particular, for the porous medium equation and a related reaction-diffusion system, only non-negative solutions and continuous boundary values have been considered by Abdulla \cite{Abdulla-PME,Abdulla-reaction-diffusion}.
Since the proof relies on approximation by strictly positive data, the idea cannot be extended to the case of sign-changing solutions.

In contrast, the picture for systems of parabolic $p$-Laplace type is more complete.
The first result in this direction is due to Paronetto \cite{Paronetto}, who considered the Cauchy--Dirichlet problem to
\begin{align*}
    \partial_t u - \mathcal{A} u = f
\end{align*}
with a monotone operator $\mathcal{A}$.
His assumptions cover the case that $\mathcal{A}$ is the $p$-Laplace operator in the range $p \geq 2$.
Concerning the domain, he assumes that the time slices have Lipschitz boundary and that they are regular deformations of their neighbors.
Moreover, Calvo, Novaga and Orlandi \cite{Calvo_Novaga_Orlandi} established the existence of weak solutions for parabolic $p$-Laplace type equations in domains that are Lipschitz-regular in space-time, and proved uniqueness under stronger assumptions on the domain.

Finally, Bögelein, Duzaar, Scheven and Singer \cite{BDSS18} were concerned with the Cauchy--Dirichlet problem with zero boundary values to systems of the form
\begin{align*}
    \partial_t u - \Div( D_\xi f(x,u,Du) ) = - D_u f(x,u,Du)
\end{align*}
in a bounded noncylindrical domain $E \subset (\Omega \times (0,T)) \subset \R^{n+1}$.
Here, the Carathéo\-dory function $f$ satisfies a $p$-coercivity condition for some $p>1$, but no specific growth condition from above, and the partial map $(u,\xi) \mapsto f(x,u,\xi)$ is convex.
Merely assuming that $\L^{n+1}(\partial E) = 0$, they established the existence of variational solutions in a subspace of $L^\infty(0,T;L^2(\Omega,\R^N))$.
For nondecreasing domains, the proof relies on the method of minimizing movements, whereas general domains are approximated from outside by a union of cylinders of small height, and the existence result from the first step is used in each of these cylinders.
In this context, note that the idea of variational solutions to parabolic PDEs stems from Lichnewsky and Temam \cite{Lichnewsky_Temam}.
The advantage of this approach is that it allows more flexibility compared to a PDE approach, especially in developing an existence theory, cf.~\cite{BDM13, BDS17, BDS16}.
Subsequently, in \cite{BDSS18} the authors show that under stronger assumptions on $E$ and $f$, variational solutions admit a weak time derivative in the dual of the space that encodes the boundary values, and that they are continuous with respect to time.
In particular, they assume that the speed, at which $E$ is allowed to shrink, is bounded.

In the present paper, we continue our research related to doubly-nonlinear systems of type \eqref{prototype equation} in irregular noncylindrical domains $E \subset \R^{n+1}$, that we started in \cite{SSSS} with the case of nondecreasing domains.
Using the same overall strategy as in \cite{BDSS18}, we first establish the existence of variational solutions in domains that merely satisfy $\L^{n+1}(\partial E) = 0$.
Next, under stronger conditions on $E$ we prove that a power of the constructed solution $u$ admits a weak time derivative, and that $u$ is continuous in time.
For the precise results, we refer to Section \ref{sec:main-results}.
Note that our results generalize the results from \cite{BDSS18} even in the case $q=1$ of $p$-Laplace systems. One of the main improvements compared to \cite{BDSS18} is that we are able to weaken the Lipschitz type condition on the shrinking of the domain from \cite{BDSS18} to a Sobolev type condition, see~\eqref{one sided growth condition} and \eqref{definition_of_r_for_one_sided_growth}. In particular, in the case $p>2$, our results cover examples of Petrovski\u\i{} type domains, cf.~\cite{BBG17}, that were not covered in \cite{BDSS18}.  We refer to Remark~\ref{rem:Petrovskii} for a more detailed discussion. Moreover, we generalize the results from \cite{BDSS18} in a second way, since we deal with general time-independent boundary values, compared to zero boundary values as in \cite{BDSS18}.

\section{Setting and main results}\label{Sec:Setting_and_main_results}
\subsection{Notation and assumptions}
Consider an open and bounded set $\Omega \subset \R^n$, $n \geq 2$, and for $0<T<\infty$ define the space-time cylinder $\Omega_T := \Omega \times [0,T)$.
Further, let $E \subset \Omega_T$ be a relatively open noncylindrical domain satisfying
\begin{equation}
	\L^{n+1}(\partial E) = 0.
	\label{eq:boundary_E_zero_measure}
\end{equation}
For fixed $t \in [0,T)$, the time slice of $E$ is given by
\begin{align*}
	E^t := \{x \in \R^n : (x,t) \in E\} \subset \R^n,
\end{align*}
such that
\begin{align*}
	E = \bigcup_{t \in [0,T)} E^t \times \{t\}.
\end{align*}
Further, we denote the lateral boundary of $E$ by 
\begin{align*}
    \partial_\mathrm{lat} E :=\bigcup_{t \in [0,T)} \partial E^t \times \{t\}.
\end{align*}

We are interested in the Cauchy-Dirichlet problem
\begin{equation}
	\left\{
	\begin{array}{cl}
		\partial_t \big(|u|^{q-1}u \big) - \Div \big( D_\xi f(x,u,Du) \big) = - D_u f(x,u,Du)
		& \text{in } E, \\[5pt]
		u=u_{\ast}
		&\text{on } \partial_\mathrm{lat} E, \\[5pt]
		u(\cdot,0) = u_o
		&\text{in } E^0,
	\end{array}
	\right.
	\label{eq:system}
\end{equation}
where $q \in (0,\infty)$, and for $N \geq 1$ the initial values $u_o \colon \Omega \to \R^N$ and the time-independent boundary datum $u_{\ast} \colon \Omega \to \R^N$ fulfill
\begin{equation}
	\left\{
	\begin{array}{l}
		\mbox{$u_o \in L^{q+1}(\Omega, \mathds{R}^{N}$),} \\[5pt]
		\mbox{$u_{\ast} \in L^{q+1}(\Omega, \mathds{R}^{N}) \cap W^{1,p}(\Omega, \mathds{R}^{N})$.}
	\end{array}
	\right.
	\label{compatibility_condition_for_u_o_and_u_ast}
\end{equation}
Further, suppose that the integrand $f \colon \Omega \times \R^N \times \R^{Nn} \to [0,\infty)$ satisfies
\begin{equation}
	\left\{
	\begin{array}{l}
		\mbox{$x \mapsto f(x,u,\xi) \in L^1(\Omega)$ for any $(u,\xi) \in \R^N \times \R^{Nn}$,} \\[5pt]
		\mbox{$(u,\xi) \mapsto f(x,u,\xi)$ is convex for a.e.~$x \in \Omega$,} \\[5pt]
		\mbox{$\nu |\xi|^p \leq f(x,u,\xi) \le L (\vert \xi \vert^{p} + \vert u \vert^{p} + G(x))$}\\
        \mbox{for a.e.~$x \in \Omega$ and all $(u,\xi) \in \R^N \times \R^{Nn}$,} 
	\end{array}
	\right.
	\label{eq:integrand}
\end{equation}
with constants $0< \nu \le L$, $p \in (1,\infty)$, and $G \in L^1(\Omega_{T})$.
As a consequence, by \cite[Lemma 2.1]{Marcellini} there exists $c=c(p,L)$ such that the local Lipschitz condition
\begin{align}\label{ineq:Lipschitz_condition}
    \vert f(x,u_{1}, \xi_{1} ) - &f(x,u_{2}, \xi_{2} ) \vert \nonumber \\
    &\le c\Big[ (\vert \xi_{1} \vert + \vert \xi_{2} \vert + \vert u_{1} \vert + \vert u_{2} \vert)^{p-1} + \vert G(x) \vert^{\frac{1}{p'}}  \Big](\vert u_{1} - u_{2} \vert + \vert \xi_{1} - \xi_{2} \vert )
\end{align}
holds true for any $(u_{1}, \xi_{1}), (u_{2}, \xi_{2}) \in \mathds{R}^{N} \times
\mathds{R}^{Nn}$ and a.e.~$x \in \Omega$. Throughout the paper, $p':= \frac{p}{p-1}$ denotes the H\"older-conjugate to $p$.

Next, concerning function spaces, on the level of the time slices we consider  
\begin{align*}
	V_t :=
	\big\{ v \in W^{1,p}_{u_{\ast}}(\Omega,\R^N) : v=u_{\ast} \text{ a.e.~in } \Omega \setminus E^t \big\}.
\end{align*}
Further, we define the parabolic function spaces
\begin{align*}
	V^p(E) :=
	\big\{ v \in L^p\big(0,T;W^{1,p}_{u_\ast}(\Omega,\R^N)\big) : v(t) \in V_t \text{ for a.e.~} t \in [0,T) \big\},
\end{align*}
and
\begin{align*}
	V^p_q(E) :=
	V^p(E) \cap L^{q+1}(\Omega_T,\R^N), 
\end{align*}
equipped with the norms
\begin{align*}
    \| v \|_{V^p(E)}
	:=
	\| Dv \|_{L^p(\Omega_T,\R^{Nn})} +\| v \|_{L^p(\Omega_T,\R^{N})},
\end{align*}
and
\begin{align*}
    \| v \|_{V^p_q(E)}
	:=
	\| Dv \|_{L^p(\Omega_T,\R^{Nn})} + \| v \|_{L^{q+1}(\Omega_T,\R^N)},
\end{align*}
respectively. Moreover, we use the notations $V^{p,0}(E)$ and
$V^{p,0}_q(E)$ for the spaces that are defined as above with $u_\ast\equiv0$ and equipped with the norms $\| \cdot \|_{V^p(E)}$ and $\| \cdot \|_{V^p_q(E)}$, respectively. 
For $u \in \R^N$, we use the shorthand notation
\begin{align*}
    \power{u}{q}
	:=
	|u|^{q-1} u,
\end{align*}
and for $u,v \in \R^N$ we define the boundary term by
\begin{align*}
	\b[u,v]
	&:=
	\tfrac{1}{q+1} |v|^{q+1} - \tfrac{1}{q+1} |u|^{q+1} - \power{u}{q} \cdot (v-u) \\
	&=
	\tfrac{1}{q+1} |v|^{q+1} + \tfrac{q}{q+1} |u|^{q+1} - \power{u}{q} \cdot v.
\end{align*}

This allows us to define variational solutions to \eqref{eq:system}.
Note that the variational inequality \eqref{eq:variational_inequality} can be derived from \eqref{eq:system}$_{1}$ analogously as in \cite{BDMS18-1}.
\begin{definition}[Variational solutions]\label{Definition:variational_solution}
Let $E \subset \Omega_T$ be relatively open, and assume that the variational integrand $f \colon \Omega \times \R^N \times \R^{Nn} \to [0,\infty)$ satisfies \eqref{eq:integrand}. Furthermore, suppose that $u_{\ast}, u_o \colon \Omega \to \mathds{R}^{N}$ fulfill the conditions \eqref{compatibility_condition_for_u_o_and_u_ast}.
We call a measurable map
\begin{align}\label{measurable_map_var_ineq}
    u \in L^\infty \big( 0,T;L^{q+1}(\Omega,\R^N) \big) \cap V^p_q(E)
\end{align}
a \emph{variational solution} to \eqref{eq:system} if and only if the variational inequality
\begin{align}
	\iint_{E\cap \Omega_\tau} &f(x,u,Du) \,\dx\dt
	\nonumber \\ &\leq
	\iint_{E\cap \Omega_\tau} f(x,v,Dv) \,\dx\dt
	+\iint_{E\cap \Omega_\tau} \partial_t v \cdot (\power{v}{q} - \power{u}{q}) \,\dx\dt
    \label{eq:variational_inequality} \\
	&\phantom{=}
	-\int_{E^\tau}	\b[u(\tau),v(\tau)] \,\dx
	+\int_{E^0} \b[u_o,v(0)] \,\dx
    \nonumber
\end{align}
holds true for a.e.~$\tau \in (0,T)$ and any comparison map $v \in V^p_q(E)$ with time derivative $\partial_t v \in L^{q+1}(\Omega_T,\R^N)$.
\end{definition}

\subsection{Main results}
\label{sec:main-results}
Now, we state the main result regarding the existence of variational solutions.
Note that the lower bound on $p$ can be omitted if $q=1$, cf. \cite{BDMS18-2}.
However, in the doubly nonlinear case $q\neq1$, we need the restriction on $p$ in order to show almost everywhere convergence of the approximations, which enables us to pass to the limit in the nonlinear time term of the variational inequality.
The condition on $p$ stems from an application of the Gagliardo-Nirenberg inequality in the proof of Theorem \ref{thm:existence_in_general_domains}.
More precisely, we will use the fact that $v \in
L^{\infty} \big(0,T;L^{q+1}(\Omega, \mathds{R}^{N}) \big) \cap
L^{p} \big(0,T;W_{0}^{1,p}(\Omega, \mathds{R}^{N}) \big)$ implies that $v$ is integrable to the power $\tfrac{p(n+q+1)}{n}>q+1$. 
\begin{theorem}
\label{thm:existence_in_general_domains}
Let $E \subset \Omega_{T}$ be a relatively open noncylindrical domain satisfying \eqref{eq:boundary_E_zero_measure}, and let $p > \max \Big\{ 1, \tfrac{n(q+1)}{n+q+1} \Big\}$.
Further, suppose that the variational integrand $f$ satisfies \eqref{eq:integrand}, and that \eqref{compatibility_condition_for_u_o_and_u_ast} holds for the initial and boundary values.
Then, there exists a variational solution to \eqref{eq:system} in the sense of Definition~\ref{Definition:variational_solution}.
\end{theorem}

\begin{remark}
\label{rem:initial_condition_convergence}
In \cite[Lemma~3.14]{SSSS} we have shown that any variational solution admits the initial values $u(0) = u_o$ in the $L_{\text{loc}}^{q+1}$-sense, i.e.,
\begin{align*}
    \lim_{h \to 0} \tfrac{1}{h} \int_{0}^{h} \Vert u(t)-u_o \Vert_{L^{q+1}(K, \mathds{R}^{N})}^{q+1} \,\dt =0
\end{align*}
for any compact set $K \subset E^{0}$.
If $u_o$ and $u_{\ast}$ coincide on $\Omega \setminus E^0$, i.e.~$u_o \equiv u_{\ast}$ on $\Omega \setminus E^0$, $K$ can be replaced by $\Omega$.
\end{remark}

In order to derive basic regularity properties of the solutions, we need to introduce additional assumptions.
First, for any $t \in [0,T)$ we assume that the complement $\Omega \setminus E^{t}$ of the time slice $E^{t}$ satisfies the measure density condition
\begin{align}\label{ineq:lower_measure_bound}
    \L^{n} \big((\mathds{R}^n \setminus E^{t}) \cap B_{r}(x_{0}) \big) \ge \delta \L^{n}(B_{r}(x_{0}))
    \quad \text{for all } x_{0} \in \partial E^{t} \text{ and } r >0
\end{align}
with a constant $\delta >0$.
Moreover, we suppose that the speed at which $E$ is allowed to shrink, is bounded.
To this end, for the complementary excess
\begin{align*}
    \boldsymbol{e}^{c}(E^{s}, E^{t}):= \sup_{x \in \Omega \setminus E^{t}} \dist(x, \Omega \setminus E^{s}),
\end{align*}
we impose the one-sided growth condition
\begin{align}\label{ineq:one_sided_growth_with_modulus}
\boldsymbol{e}^{c}(E^{s},E^{t})\leq\omega(t-s)\quad\text{whenever }0\leq s\leq t<T,
\end{align}
where $\omega:[0,\infty)\rightarrow[0,\infty)$ with $\omega(0)=0$ denotes a modulus of continuity.
Without loss of generality, we may assume that $\omega$ is strictly increasing.
Observe that \eqref{ineq:one_sided_growth_with_modulus} prevents the formation of holes inside the domain $E$.

Now, we are ready to give the following statement regarding the time derivative of $\power{u}{q}$ in the dual space.
As usual, its norm is given by 
\begin{align}
    \Vert L \Vert_{(V^{p,0}(E))^{\prime}}:= \sup_{0 \neq u \in V^{p,0}(E)} \tfrac{\vert \langle L,u \rangle \vert}{\Vert u \Vert_{V^{p}(E)}} \quad \text{ for all } L \in (V^{p,0}(E))^{\prime},
    \label{eq:operator-norm}
\end{align}
where $\langle \cdot, \cdot \rangle$ denotes the duality pairing between $(V^{p,0}(E))^{\prime}$ and $V^{p,0}(E)$.

\begin{theorem}\label{thm:time_derivative_in_the_dual_space_general_domains}
Assume that the noncylindrical domain $E$ fulfills \eqref{eq:boundary_E_zero_measure}, \eqref{ineq:lower_measure_bound} and \eqref{ineq:one_sided_growth_with_modulus}, that the initial and boundary values satisfy \eqref{compatibility_condition_for_u_o_and_u_ast}, and that the integrand $f$ satisfies \eqref{eq:integrand}.
Further, let $u$ be the variational solution constructed in Theorem~\ref{thm:existence_in_general_domains}.
Then, $\power{u}{q}$ possesses a distributional derivative 
    \begin{align*}
        \partial_t \power{u}{q} \in (V^{p,0}(E))^{\prime},
    \end{align*}
    with the estimate
    \begin{align*}
        &\Vert   \partial_t \power{u}{q} \Vert_{(V^{p,0}(E))^{\prime}}
        \\ &\le
        c \Big[ T \big( \Vert u_{\ast} \Vert_{W^{1,p}(\Omega,\R^N)}^p + \| G \|_{L^1(\Omega)} \big)
        + \Vert u_o \Vert_{L^{q+1}(E^{0}, \mathds{R}^{N})}^{q+1} + \Vert u_{\ast} \Vert_{L^{q+1}(\Omega, \mathds{R}^{N})}^{q+1} \Big]^{\frac{1}{p'}},
    \end{align*}
    where $c=c(p,q, \nu, L, \diam(\Omega))$ is a constant.
\end{theorem}

In order to conclude from Theorem \ref{thm:time_derivative_in_the_dual_space_general_domains} that $u$ is continuous in time, we need to replace \eqref{ineq:one_sided_growth_with_modulus} by the stricter condition
\begin{align}\label{one sided growth condition}
    \power{e}{c} \big(E^{s}, E^{t} \big) \le \vert \rho(t)-\rho(s) \vert
    \quad \text{for $0 \le s \le t <T$,}
\end{align}
where the function $\rho:(-1,T) \rightarrow(0,\infty)$ satisfies
$\rho\in W^{1,r}(-1,T)$ for
\begin{align}\label{definition_of_r_for_one_sided_growth}
  \left\{
  \begin{array}{ll}
    r = p', & \text{if }p \geq q+1,\\[1.5ex]
    r = \frac{p(n+q+1)-n(q+1)}{p(n+q+1)-(n+1)(q+1)}, & \text{if } \frac{(n+1)(q+1)}{n+q+1} < p < q+1,\\[1.5ex]
    r = \infty, &\text{if } p = \frac{(n+1)(q+1)}{n+q+1}.
  \end{array}
  \right.
\end{align}

\begin{remark}
Note that the lower bound $\frac{p(n+q+1)-n(q+1)}{p(n+q+1)-(n+1)(q+1)}$ goes to $p'$ when $p \uparrow q+1$, and it blows up in the limit $p \downarrow \frac{(n+1)(q+1)}{n+q+1}$.
Further, by the embedding properties of the one-dimensional Sobolev space, \eqref{one sided growth condition} implies \eqref{ineq:one_sided_growth_with_modulus} with $\omega(t)=\|\rho'\|_{L^r(0,T)}t^{1-\frac{1}{r}}$.
\end{remark}

\begin{theorem}\label{continuity_in_time_for_general_domains}
Assume that $E$ satisfies \eqref{eq:boundary_E_zero_measure}, \eqref{ineq:lower_measure_bound} and \eqref{one sided growth condition}, that
$p \geq \tfrac{(n+1)(q+1)}{n+q+1}$,
that the initial and boundary values fulfill \eqref{compatibility_condition_for_u_o_and_u_ast}, and that the integrand $f$ fulfills \eqref{eq:integrand}.
Then, any variational solution $u$ to \eqref{eq:system} in the sense of Definition \ref{Definition:variational_solution} that possesses a time derivative $\partial_t \power{u}{q} \in (V^{p,0}(E))'$ satisfies
\begin{align*}
    u \in C^{0}([0,T); L^{q+1}(\Omega, \mathds{R}^{N})).
\end{align*}
\end{theorem}

It would be interesting to determine whether the lower bound on $p$ in Theorem \ref{continuity_in_time_for_general_domains} is optimal.

\begin{remark}\label{rem:Petrovskii}
If $E$ shrinks too fast, we know that weak solutions to the Cauchy--Dirichlet problem \eqref{eq:system} with a continuous boundary and initial datum may be discontinuous in a boundary point, cf.~\cite{BBG17, BBGP15, KiLi, Lindqvist}.
In particular, for the parabolic $p$-Laplace equation in
\begin{align*}
    \Theta_\lambda := \big\{ (x,t) \in \mathds{R}^{n+1}: \vert x \vert < K |t|^{\lambda} \text{ and } -1<t<0 \big\}
\end{align*}
with a constant $K>0$, A.~Björn, J.~Björn and Gianazza prove Petrovski\u\i{}-type criteria in \cite[Theorem 1.1]{BBG17}.
For $p>2$, they obtain that the origin $(0,0)$ is regular if and only if $\lambda > \frac{1}{p}$.
Since
\begin{align*}
    \power{e}{c} \big(\Theta_\lambda^{s}, \Theta_\lambda^{t} \big)
    =
    |s|^\lambda - |t|^\lambda
\end{align*}
for $-1 < s \leq t < 0$, and since $\varrho(\cdot) = |\cdot|^\lambda \in W^{1,r}(-1,0)$ with $r=p'$ if and only if $\lambda > \frac{1}{p}$, our results are consistent with \cite[Theorem 1.1]{BBG17} in this case.
For $\frac{(n+1)(q+1)}{n+q+1}\le p< 2$, we have that $r>p'$ ($r=p'$ for $p=2$), and thus that \eqref{one sided growth condition} is stronger than the regularity condition $\lambda > \frac{1}{p}$ \big($\lambda \geq \frac{1}{2}$ for $p=2$\big) for the origin in \cite[Theorem 1.1]{BBG17}.
For the case of the heat equation and the porous medium equation ($p=2$ and $0<q<1$) there are more refined versions of Petrovski\u\i{}'s criterion. We refer to \cite[Section 3.3]{Galaktionov} for a discussion.

Overall, in view of the possible irregularity of boundary points we expect that variational solutions $u$ to \eqref{eq:system} neither satisfy the weaker continuity property $u \in C^0\big([0,T);L^{q+1}(\Omega,\R^N) \big)$ if $E$ decreases too fast.
Again, it would be interesting to find the optimal condition on the evolution of $E$ in this context.
\end{remark}

\subsection{Plan of the paper}
After collecting auxiliary results in Section~\ref{sec:preliminaries}, a separate section is dedicated to each of the three main results of the present paper.
In Section~\ref{sec:existence-proof} we prove Theorem~\ref{thm:existence_in_general_domains}, i.e.~the existence of variational solutions to \eqref{eq:system}.
In contrast to the case of nondecreasing domains treated in \cite{SSSS}, we are not able to use the method of minimizing movements for general noncylindrical domains.
Instead, we adopt the overall proof strategy that was employed for the special case $q=1$ in \cite{BDSS18} for such domains.
This means that we cover $E$ by a finite union of cylinders $Q_{\ell,i}$, $i=1,\ldots,\ell$, of height $0 < \frac{T}{\ell} \ll 1$ in the time direction.
Relying on the results of \cite{SSSS}, we obtain the existence of variational solutions $u_{\ell,i}$ in $Q_{\ell,i}$, where we choose initial values $u_o$ for $i=1$ and the values $u_{\ell,i-1}(ih_\ell)$ of the preceding solution at the last time slice of the preceding cylinder for $i=2,\ldots,\ell$.
Next, we define the map $u_\ell \colon \Omega_T \to \R^N$ by gluing the individual solutions $u_{\ell,i}$.
By means of energy estimates, a subsequence of $(u_\ell)_{\ell \in \N}$ converges weakly to a limit map $u$ in the right function space for a variational solution.
At this stage, we glue the variational inequalities satisfied by the functions $u_{\ell,i}$ and would like to pass to the limit $\ell \to \infty$.
However, due to the term with the integrand $\partial_t v \cdot \big(\power{v}{q} - \power{u_\ell}{q} \big)$, to this end we need to ensure that the limit of $\power{u_\ell}{q}$ with respect to weak convergence is $\power{u}{q}$.
Therefore, for general $q \neq 1$ an extra argument compared to the case $q=1$ in \cite{BDSS18} is required.
In fact, we use the variational inequalities to deduce that $u_\ell$ converges to $u$ a.e.~as $\ell \to \infty$.
In this connection, the lower bound on $p$ is necessary to apply the Gagliardo--Nirenberg inequality.

Next, in Section~\ref{sec:time-derivative-proof} we establish Theorem~\ref{thm:time_derivative_in_the_dual_space_general_domains}, which shows that under stronger conditions on $E$, that ensure that $C^\infty_0(E)$ is dense in $V^{p,0}(E)$, the power $\power{u}{q}$ of the variational solution $u$ constructed in Theorem~\ref{thm:existence_in_general_domains} admits a weak time derivative.
While the method of proof is analogous, we generalize the corresponding result in \cite{BDSS18} for the case $q=1$ and stricter conditions on $E$.
Using the bounds in \cite{SSSS} for the weak time derivatives of the maps $\power{u_{\ell,i}}{q}$ from the proof of Theorem~\ref{thm:existence_in_general_domains}, we deduce that $\partial_t \power{u_\ell}{q} \in (V^{p,0}(E))'$ with a quantitative bound that is independent of $\ell \in \N$.
Hence, passing to the limit $\ell \to \infty$, we conclude the desired property for $\power{u}{q}$.

Finally, in Section~\ref{sec:continuity-proof} we prove Theorem~\ref{continuity_in_time_for_general_domains}, i.e.~the continuity in time of variational solutions under suitable assumptions on $E$.
Since mollifications of a variational solution will not attain the correct boundary values in the general setting, this requires a delicate argument.
First, in Section~\ref{sec:integration-by-parts} we develop an integration by parts formula that might be of independent interest.
In this context, we were able to generalize the corresponding formula for $q=1$ in \cite{BDSS18} in two directions.
Namely, we deal with general $q \neq 1$, and we weaken the assumptions on the speed at which $E$ may shrink.
Moreover, we simplify the arguments of \cite{BDSS18}.
In particular, our proof shows transparently how the one-sided condition on $E$ leads to the inequality in the formula.
Next, by means of this integration by parts, we prove the left-sided continuity of certain variational solutions $u$ in time in Section~\ref{sec:left-sided-continuity}, and we derive a localized version with respect to time of the variational inequality in Section~\ref{sec:localization}.
In turn, this allows us to conclude the right-sided continuity of $u$ in time in Section~\ref{sec:right-sided-continuity}.

\section{Preliminaries}
\label{sec:preliminaries}
\subsection{Technical lemmas}
For the following lemma, we refer to \cite[Lemma 8.3]{Giusti}; see also \cite[Lemma 3.2]{BDKS20}.
\begin{lemma}\label{ineq:power_alpha_estimate_q_greater_0}
For any $\alpha > 0$ and $u, v \in \mathds{R}^{N}$, there exists a
constant $C$ depending only on $\alpha$ such that
\begin{align*}
    \tfrac{1}{C} \vert \power{v}{\alpha} - \power{u}{\alpha} \vert \le  ( \vert u \vert + \vert v \vert )^{\alpha -1} \vert v - u \vert \le C \vert \power{v}{\alpha} - \power{u}{\alpha} \vert.
\end{align*}
\end{lemma}

The preceding result implies the following lemma; see also \cite[Lemma 3.3]{BDKS20}.
\begin{lemma}\label{ineq:power_alpha_estimate}
For any $\alpha > 1$ and $u, v \in \mathds{R}^{N}$, there exists a
constant $C$ depending only on $\alpha$ such that
\begin{align*}
    \vert v - u \vert^{\alpha} \le C \vert \power{v}{\alpha} - \power{u}{\alpha} \vert.
\end{align*}
\end{lemma}

Moreover, we have the following estimate, see  \cite[Lemma 3.3]{SSSS}.
\begin{lemma} \label{lem:technical_lemma_2}
Let $q \in (0,\infty)$. Then, for any $u,v \in \R^N$ we have the estimate 
\begin{align*}
	\vert u-v \vert^{q+1}
	\leq
	c(q) \Big(\vert u \vert^\frac{q+1}{2} + \vert v \vert^\frac{q+1}{2} \Big) \Big\vert \power{u}{\frac{q+1}{2}} - \power{v}{\frac{q+1}{2}} \Big\vert.
\end{align*}
\end{lemma}

The proof of the next estimate for the boundary term $\b[u,v]$ can be found in \cite[Lemma 3.4]{SSSS}.
\begin{lemma}
\label{lem:technical_lemma}
Let $q \in (0,\infty)$ and $u,v \in \R^N$. Then, we have that
\begin{align*}
     \Big| \power{u}{\frac{q+1}{2}} - \power{v}{\frac{q+1}{2}} \Big|^2
	\leq
	c(q) \b[u,v]
	\leq
	c(q) \big( \power{v}{q} - \power{u}{q} \big) \cdot (v-u).
\end{align*}
\end{lemma}

Next, we are concerned with the $p$-Hardy inequality.
It holds under the following $p$-fatness condition on the complement of the time slices, which is weaker than the measure density condition~\eqref{ineq:lower_measure_bound}, see e.g. \cite[Example 6.18]{KiLeVa}.
For the definition of the variational $p$-capacity $\ca_p$, we refer to  \cite[Definition 5.32]{KiLeVa}. 

\begin{definition} \label{def:p-fat}
A set $A \subset \R^n$ is called uniformly $p$-fat if there exists a constant $\alpha > 0$ such that
\begin{align}\label{ineq:p-fatness_condition}
\ca_p\big(A \cap \overline{B_\rho(x)}, B_{2\rho}(x)\big) \geq \alpha \ca_p\big(\overline{B_\rho(x)}, B_{2\rho}(x)\big) 
\end{align}
holds true for every $x \in A$ and $\rho > 0$.
\end{definition}

Then, the $p$-Hardy inequality reads as follows, see e.g.~\cite[Theorem 6.25]{KiLeVa}.
\begin{lemma}
\label{lem:p-Hardy}
If $A \subset \R^n$ is an open set whose complement is $p$-fat with fatness constant $\alpha$, then there is a constant $c=c(n,p,\alpha)$ such that the $p$-Hardy inequality
\begin{equation*}
	\int_A \bigg| \frac{u(x)}{\dist(x,\partial\Omega)} \bigg|^p \,\dx
    \leq
	c \int_A |\nabla u(x)|^p \,\dx
\end{equation*}
holds for any $u \in W^{1,p}_0(\Omega)$.
\end{lemma}

\subsection{Mollification in time}
For our analysis it is convenient to work with the regularization with respect to time that was introduced by Landes in \cite{Landes}.
Let $X$ be a separable Banach space, and consider functions $v_o \in X$ and $v \in L^r(0,T;X)$ for some $1 \leq r \leq \infty$.
Later on, we will mostly choose $X=L^{q+1}(\Omega,\R^N)$ and $X=W^{1,p}(\Omega,\R^N)$.
For $h>0$ we construct the mollification $[v]_h$ as a solution to the ordinary differential equation
\begin{equation}
	\partial_t [v]_h
	=
	-\tfrac1h \big( [v]_h - v \big)
	\label{eq:ODE_mollification}
\end{equation}
with initial condition $[v]_h(0) = v_o$, which is given by 
\begin{equation}
	[v]_h(t) :=
	e^{-\frac{t}{h}} v_o + \tfrac1h \int_0^t e^\frac{s-t}{h} v(s) \,\ds
	\label{eq:time_mollification}
\end{equation}
for any $t \in [0,T]$.
In the following lemma, the proofs of which can be found in \cite[Appendix B]{BDM13} and \cite{Naumann}, we state some properties of this mollification.
\begin{lemma}\label{lem:mollification:estimate_and_continuous_convergence}
Suppose that $X$ is a separable Banach space and $v_o \in X$.
If $v \in L^r(0,T;X)$ for some $1 \leq r \leq \infty$, the mollification $[v]_h$ defined by \eqref{eq:time_mollification} satisfies $[v]_h \in L^r(0,T;X)$ and for any $t_o \in (0,T]$ we have the bound
\begin{align*}
	\big\| [v]_h \big\|_{L^r(0,t_o;X)}
	\leq
	\| v \|_{L^r(0,t_o;X)}
	+ \Big[ \tfrac{h}{r} \Big( 1 - e^{-\frac{t_o r}{h}} \Big) \Big]^\frac{1}{r} \| v_o\|_X,
\end{align*}
where the bracket $[\ldots]^\frac{1}{r}$ is interpreted as $1$ in the case $r=\infty$.
Moreover, if $r<\infty$ we have that $[v]_h \to v$ in $L^r(0,T;X)$ in the limit $h \downarrow 0$.
Furthermore, if $v \in C^0([0,T];X)$ and $v_o = v(0)$, then $[v]_h \in C^0([0,T];X)$ with $[v]_h(0) = v(0)$ and moreover, we have that $[v]_h \to v$ in $C^0([0,T];X)$ as $h \downarrow 0$.
\end{lemma}

For the next convergence result for mollifications see \cite[Lemma 2.8]{Schaetzler}.
\begin{lemma}\label{lem:mollification_weak_convergence}
    Let $r\ge 1$ and $v_o \in L^{r}(\Omega_{T}, \mathds{R}^{N})$.
    Suppose that $(v_{\ell})_{\ell \in \mathds{N}}$ is a sequence in $L^{r}(\Omega_{T}, \mathds{R}^{N})$, such that $v_{\ell} \wto v$ weakly in $L^{r}(\Omega_{T}, \mathds{R}^{N})$ for $\ell \to \infty$, where $v \in L^{r}(\Omega_{T}, \mathds{R}^{N})$.
    Then, for the mollifications $[v_{\ell}]_{h}, [v]_{h}$ defined as in \eqref{eq:time_mollification} for a fixed $h>0$ and initial value $v_o$ it follows that $[v_{\ell}]_{h} \wto [v]_{h}$ weakly in $L^{r}(\Omega_{T}, \mathds{R}^{N})$ as $\ell \to \infty$.
\end{lemma}

\subsection{Density of smooth functions}
In this section, we give sufficient conditions under which $C_{0}^{\infty}(E, \mathds{R}^{N})$ is dense in $V_{q}^{p,0}(E)$.
For the proof of the following statements, we refer to \cite[Section 3.5]{SSSS}; see also \cite[Section 5.3.1]{BDSS18} for the case $q=1$.

First, for $t\in[0,T)$ and $\sigma >0$ we denote the inner parallel set of $E^{t}$ by
\begin{equation}\label{def:parallel-set}
    E^{t,\sigma} := \big\{ x \in E^{t} : \dist \big( x, \partial E^{t} \big) > \sigma \big\},
\end{equation}
and define the parabolic function space
\begin{align*}
    V_\mathrm{cpt}^{p,q}(E):= \Big\{ u \in V_{q}^{p,0}(E) : \exists \sigma > 0 \text{ with }\spt(u(t)) \subset \overline{E^{t,\sigma}} \text{ for a.e.~$t\in[0,T)$} \Big\}.
\end{align*}
Assuming that the speed at which $E$ may shrink, is controlled by \eqref{ineq:one_sided_growth_with_modulus}, but without further requirements on the regularity of the time slices, we obtain the following intermediate result.  
\begin{lemma}\label{intermediate_density_argument}
    Suppose that \eqref{ineq:one_sided_growth_with_modulus} holds with a modulus of continuity $\omega$.
    Then  $C_{0}^{\infty}(E, \mathds{R}^{N})$ is dense in $V_\mathrm{cpt}^{p,q}(E)$ with respect to the norm $\Vert \cdot \Vert_{V_{q}^{p}(E)}$. Furthermore, if $u \in V_\mathrm{cpt}^{p,q}(E)$, then there exist $\sigma > 0$ and an approximating sequence $\phi_{k} \in C_{0}^{\infty}(E, \mathds{R}^{N})$ with $\phi_{k} \to u$ in the norm topology of $V_{q}^{p,0}(E)$ such that $\spt(\phi_{k}(t)) \subset \overline{E^{t,\sigma}}$ for all $k \in \mathds{N}$ and $t \in [0,T)$.
\end{lemma}

In the remaining part of the section, we assume that
\begin{equation}
  \label{p-fat}
  \mbox{for every $t\in(0,T)$, $\R^n\setminus E^t$ is uniformly $p$-fat with a parameter $\alpha>0$}
\end{equation}
in the sense of Definition~\ref{def:p-fat}.
Note that the measure density condition \eqref{ineq:lower_measure_bound} implies~\eqref{p-fat} with some $\alpha=\alpha(n,p,\delta)$.
Moreover, we define the subspaces
\begin{equation*}
    \mathcal{V}^{p,0}(E) :=
    \Big\{u\in V^{p,0}(E) :
    u(t) \in W^{1,p}_0(E^t,\R^N) \mbox{ for a.e.~$t\in[0,T)$} \Big\}
\end{equation*}
and
\begin{equation*}
    \mathcal{V}^{p,0}_q(E) :=
    \Big\{u\in V^{p,0}_q(E) : u(t) \in W^{1,p}_0(E^t,\R^N) \mbox{ for a.e.~$t\in[0,T)$} \Big\}.
\end{equation*}

\begin{remark}\label{rem:V=V}
If the measure density condition~\eqref{ineq:lower_measure_bound} holds, we have that $\mathcal{V}^{p,0}(E)=V^{p,0}(E)$, and thus that $\mathcal{V}^{p,0}_q(E)=V^{p,0}_q(E)$, cf.~\cite[Lemma~3.1]{BDSS18}. 
\end{remark}

For a cut-off function $\widetilde{\eta} \in C^{0,1}(\mathds{R})$ such that $\widetilde{\eta} \equiv 0$ in $(-\infty,1]$, $\widetilde{\eta}(r):=r-1$ for $r \in (1,2)$ and $\widetilde{\eta} \equiv 1$ in $[2, \infty)$, and $\sigma >0$ we define 
\begin{align}\label{definition_eta_sigma}
    \eta_{\sigma}(x,t):= \widetilde{\eta}\bigg( \frac{\dist(x,\Omega \setminus E^t)}{\sigma} \bigg) \quad \text{for }(x,t) \in \Omega_{T}.
\end{align}
In this setting, we have the following convergence result.
\begin{lemma}\label{weak_convergence_curoff_density_part}
    Let $E$ satisfy the $p$-fatness condition~\eqref{p-fat} with a parameter $\alpha>0$, consider $u \in \mathcal{V}^{p,0}_q(E)$, and let $\eta_\sigma$ denote the cut-off function in \eqref{definition_eta_sigma}.
    Then, we have that
    \begin{align}
        \eta_{\sigma} u \wto u \text{ weakly in } \mathcal{V}_{q}^{p,0}(E)\text{ as }\sigma \downarrow 0.
    \end{align}
\end{lemma}

Furthermore, the following density result holds.
\begin{proposition}\label{prop:C_0_inf_is_dense_in_curly_V_q_p,0}
    Let $E$ satisfy~\eqref{p-fat} and the one-sided growth condition \eqref{ineq:one_sided_growth_with_modulus}.
    Then $C_{0}^{\infty}(E,\mathds{R}^{N})$ is dense in $\mathcal{V}^{p,0}_q(E)\subset V^{p,0}_q(E)$ with respect to the norm topology in $V^{p,0}_q(E)$.
\end{proposition}

Under the stronger measure density condition \eqref{ineq:lower_measure_bound}, Proposition \ref{prop:C_0_inf_is_dense_in_curly_V_q_p,0} in combination with Remark \ref{rem:V=V} implies the following corollary. 
\begin{corollary}\label{cor:C_0_inf_is_dense_in_V_q_p,0}
    Suppose that $E$ fulfills the measure density condition \eqref{ineq:lower_measure_bound} and the one-sided growth condition \eqref{ineq:one_sided_growth_with_modulus}.
    Then $C_{0}^{\infty}(E,\mathds{R}^{N})$ is dense in $V^{p,0}_q(E)$ with respect to the norm topology in $V^{p,0}_q(E)$.  
\end{corollary}

\begin{remark}\label{rem:C_0_inf_is_dense_in_V_p,0}
    Moreover, by applying  Corollary~\ref{cor:C_0_inf_is_dense_in_V_q_p,0} with $q=p-1$ it follows that $C_{0}^{\infty}(E,\mathds{R}^{N})$ is dense in $V^{p,0}(E)$ with respect to the norm topology in $V^{p,0}(E)$.
\end{remark}

\section{Proof of Theorem \ref{thm:existence_in_general_domains}}
\label{sec:existence-proof}
In this section, we prove Theorem \ref{thm:existence_in_general_domains}. The proof is divided into several steps.

\subsection{Constructing the sequence of approximating solutions}
The construction of the sequence of approximating solutions follows
the same strategy as in \cite[Section 5.1.1]{BDSS18}. Let $\ell \in
\mathds{N}$ and define $h_{\ell}:= \tfrac{T}{\ell}$, as well as
$t_{\ell, i}:=ih_\ell$ for $i\in\{0,\ldots,\ell\}$ and $I_{\ell, i} := [t_{\ell, i-1}, t_{\ell,i})$, for all $i \in \{1,\ldots,\ell\}$. We introduce open sets
\begin{align*}
    E_{\ell, i}:= \bigcup_{t \in I_{\ell, i}} E^t \subset \mathds{R}^{n}
\end{align*}
and the corresponding cylinders
\begin{align*}
    Q_{\ell, i}:=E_{\ell, i} \times I_{\ell, i}.
\end{align*}
This construction yields the covering
\begin{align}\label{cylinders_superset_general_domain}
    E^{(\ell)}
    :=
    \bigcup_{i=1}^{\ell} Q_{\ell,i}\supset E. 
\end{align}
We now want to construct a sequence of variational solutions on the individual cylindrical domains. This is done by the following procedure:
For $i \in \{1,\ldots,\ell\}$, by~\cite[Theorem~2.3]{SSSS} there exists a variational solution
\begin{align*}
    u_{\ell,i} \in C^{0} \big(\overline{I_{\ell,i}}; L^{q+1}(E_{\ell,i}, \mathds{R}^{N}) \big) \cap L^{p} \big(I_{\ell, i}; W_{u_{\ast}}^{1,p}(E_{\ell,i}, \mathds{R}^{N}) \big)
\end{align*}
in the cylindrical domain $Q_{\ell,i}$, where the initial data are chosen as $u_o$ for $i=1$ and 
\begin{align}\label{initial_condition_u_l,i_generarl_domain}
    u_{\ell,i}^{(0)} := u_{\ell, i-1}(t_{\ell,i-1}) \chi_{E_{\ell,i}}
    + u_\ast \chi_{\Omega \setminus E_{\ell,i}}
    \in L^{q+1}(\Omega, \mathds{R}^{N})
\end{align}
for $i \in \{2,...,\ell\}$.
Here, $u_{\ell,i}$, $i \in \{1,...,\ell\}$, have been extended to $(\Omega \setminus E_{\ell,i}) \times \overline{I_{\ell,i}}$ by $u_{\ast}$.
Next, we merge these functions to a single map $u_{\ell} \colon \Omega_{T} \to \mathds{R}^{N}$ defined by
\begin{align}\label{definition_of_approximate_solution(general_domain)}
    u_{\ell}(x,t):=u_{\ell,i}(x,t) \quad \text{ for }(x,t) \in \Omega \times I_{\ell,i} \text{ with } i \in \{1,...,\ell\}.
\end{align}

\subsection{Energy estimates and weak convergence}
Since the functions $u_{\ell, i}$, $i=1,\ldots,\ell$, are variational solutions, the inequalities
\begin{align}\label{ineq:var_ineq_seq_var_sol_u_li}
    \iint_{\Omega \times [t_{\ell,i-1}, \tau_i)}& f(x,u_{\ell,i}, Du_{\ell,i}) \,\dx\dt \nonumber 
    \\ \le& 
    \iint_{\Omega \times [t_{\ell,i-1}, \tau_i)} f(x,v, Dv) \,\dx\dt 
    + \iint_{\Omega \times [t_{\ell,i-1}, \tau_i)} \partial_t v \cdot \big(\power{v}{q}-\power{u_{\ell,i}}{q} \big) \,\dx\dt \nonumber \\
    &-\int_{\Omega } \b[u_{\ell,i}(\tau_i),v(\tau_i)] \,\dx  
    + \int_{\Omega } \b\Big[u_{\ell,i}^{(0)},v(t_{\ell,i-1}) \Big] \,\dx 
\end{align}
hold for any $\tau_i \in \overline{I_{\ell,i}}$ and any comparison map $v \in L^{p}(I_{\ell, i}; W^{1,p}(\Omega, \mathds{R}^{N}))$ with $\partial_t v \in L^{q+1}(Q_{\ell,i}, \mathds{R}^{N})$, $v(t_{\ell,i-1}) \in L^{q+1}(E_{\ell,i}, \R^{N})$ (we extend $v(t_{\ell,i-1})$ to $\Omega \setminus E_{\ell,i}$ by $u_\ast$) and $v \equiv u_{\ast}$ a.e.~in $(\Omega \setminus E_{\ell,i}) \times I_{\ell,i}$.
Here, we particularly stress that we are allowed to write any instead of a.e.~$\tau_i \in I_{\ell,i}$, since the maps $u_{\ell, i}$ are continuous in time by~\cite[Proposition~3.16]{SSSS} and we may assume that $u \equiv u_\ast$ a.e.~in $\Omega \setminus E^t$ for any time $t \in (0,T)$.
For the special choice of $v \equiv u_{\ast}$ as the comparison map in \eqref{ineq:var_ineq_seq_var_sol_u_li} we deduce the inequality
\begin{align*}
    &\iint_{\Omega \times [t_{\ell,i-1}, \tau_i)} f(x,u_{\ell,i}, Du_{\ell,i}) \,\dx\dt     
    \\&\le
    \iint_{\Omega \times [t_{\ell,i-1}, \tau_i)} f(x,u_{\ast}, Du_{\ast}) \,\dx\dt
   -\int_{\Omega } \b[u_{\ell,i}(\tau_i),u_{\ast}] \,\dx  
    + \int_{\Omega } \b \Big[ u_{\ell,i}^{(0)},u_{\ast} \Big] \,\dx.
\end{align*}
Note that the term including the time derivative vanished, because $u_\ast$ is independent of time. 
Next, we establish an analogous inequality for $u_\ell$.
To this end, for $\tau \in [0,T]$ let $m \leq \ell$ such that $\tau \in I_{\ell,m}$, and add up the inequalities for $u_{\ell,i}$, $i=1, \ldots,m$, where $\tau_i = t_{\ell,i}$ for $i=1,\ldots,m-1$ and $\tau_m=\tau$.
This leads us to
\begin{align}
    &\iint_{\Omega_{\tau}} f(x,u_{\ell}, Du_{\ell}) \,\dx\dt
    \nonumber \\&\le
    \iint_{\Omega_{\tau}} f(x,u_{\ast}, Du_{\ast}) \,\dx\dt
    + \int_{\Omega} \sum_{i=1}^m \Big( \b \Big[ u_{\ell,i}^{(0)},u_\ast \Big] - \b[u_{\ell,i}(\tau_{i}),u_\ast]\Big) \,\dx.
    \label{ineq:var_ineq_substituted_u_ast_and_summed_up}
\end{align}
The sum on the right-hand side can be rewritten as
\begin{align}\label{eq:b_sum_split_up}
  \sum_{i=1}^m &\Big( \b \Big[ u_{\ell,i}^{(0)},u_\ast \Big] - \b[u_{\ell,i}(\tau_{i}),u_\ast]\Big) \nonumber \\
  &=
    \b \Big[ u_{\ell,1}^{(0)},u_\ast \Big] - \b[u_{\ell,m}(\tau),u_\ast]
    + \sum_{i=1}^{m-1} \Big( \b \Big[ u_{\ell,i+1}^{(0)},u_\ast \Big] - \b[u_{\ell,i}(t_{\ell,i}),u_\ast]\Big).
\end{align}
Note that by definition of $u_{\ell,i+1}^{(0)}$, the term in the last sum is
zero in $E_{\ell,i+1}$, while in $\Omega\setminus E_{\ell,i+1}$, it
equals $-\b[u_{\ell,i}(t_{\ell,i}),u_\ast] \le 0$. Hence, in any case
the sum on the right-hand side of \eqref{eq:b_sum_split_up} is $\le 0$
and can be omitted in \eqref{ineq:var_ineq_substituted_u_ast_and_summed_up}. Moreover, the
first two terms of \eqref{eq:b_sum_split_up} can be estimated in the
following way. On the set $\Omega\setminus E_{\ell,1}$, they equal $-
\b[u_{\ell,m}(\tau),u_\ast]\le0$, while on the set $E_{\ell,1}$, they can be
rewritten to 
\begin{align*}
    \b[u_o&,u_\ast]- \b[u_{\ell,m}(\tau),u_\ast] \\
    &= \tfrac{q}{q+1} \vert u_o \vert^{q+1} - \tfrac{q}{q+1} \vert u_{\ell,m}(\tau) \vert^{q+1} + \big( \power{(u_{\ell,m}(\tau))}{q} - \power{u_o}{q} \big) \cdot u_\ast.
\end{align*}
Applying the Cauchy-Schwarz inequality and Young's
inequality to the products in the last term on the right-hand side, we
obtain the estimates 
\begin{align*}
	\big| \power{u_o}{q} \cdot u_\ast \big|
	\leq
	|u_o|^q |u_\ast|
	\leq
	\tfrac{q}{q+1} |u_o|^{q+1} + \tfrac{1}{q+1} |u_\ast|^{q+1},
\end{align*}
and
\begin{align*}
    \vert \power{(u_{\ell,m}(\tau))}{q} \cdot u_{\ast} \vert \le \vert u_{\ell,m}(\tau) \vert^{q} \cdot \vert u_{\ast} \vert \le \tfrac{q}{2(q+1)} |u_{\ell,m}(\tau)|^{q+1} + \tfrac{2^{q}}{q+1} |u_{\ast}|^{q+1}.
\end{align*}
Since $\tau\in I_{\ell,m}$, we have $u_{\ell,m}(\tau)=u_\ell(\tau)$.
Joining the preceding estimates with~\eqref{ineq:var_ineq_substituted_u_ast_and_summed_up}, and using the coercivity and $p$-growth assumption~\eqref{eq:integrand}$_{3}$ on $f$, we deduce that
\begin{align}
    &\nu \iint_{\Omega_{\tau}} \vert Du_{\ell} \vert^{p} \,\dx\dt
	+ \tfrac{q}{2(q+1)} \int_\Omega |u_\ell(\tau)|^{q+1} \,\dx
    \nonumber \\ &\leq
	\iint_{\Omega_{\tau}} f(x,u_{\ell}, Du_{\ell}) \,\dx\dt
	+ \tfrac{q}{2(q+1)} \int_\Omega |u_\ell(\tau)|^{q+1} \,\dx 
    \nonumber \\
    &\leq
	\iint_{\Omega_{\tau}} f(x,u_{\ast}, Du_{\ast}) \,\dx\dt
	+ \tfrac{2q}{q+1} \int_\Omega |u_o|^{q+1} \,\dx
    +  \tfrac{2^{q}+1}{q+1} \int_\Omega \vert u_{\ast} \vert^{q+1} \,\dx
    \nonumber \\ &\leq
    \tau L \int_\Omega \vert Du_{\ast} \vert^{p} + \vert u_{\ast} \vert^{p} + G(x) \,\dx
	+ c(q) \int_\Omega |u_o|^{q+1} \,\dx
    +  c(q) \int_\Omega \vert u_{\ast} \vert^{q+1} \,\dx.
    \label{ineq:pre_energy_estimate}
\end{align}
Note that both terms on the left-hand side of~\eqref{ineq:pre_energy_estimate} are non-negative.
On the one hand, omitting the second one, while choosing $\tau=T$, we obtain that
\begin{align}\nonumber
     \iint_{\Omega_{T}} \vert Du_{\ell} \vert^{p} \,\dx\dt
    &\le
    c(\nu,L,T) \int_\Omega \vert Du_{\ast} \vert^{p} + \vert u_{\ast} \vert^{p} + G \,\dx\\ &\phantom{=}
	+ c(q,\nu) \int_\Omega |u_o|^{q+1} \,\dx
    \nonumber 
    + c(q,\nu) \int_\Omega \vert u_{\ast} \vert^{q+1} \,\dx\\
    &=:C_1^p.
\label{ineq:L^p_estimate_for_derivative}
\end{align}
On the other hand, omitting the first term on the left-hand side
of~\eqref{ineq:pre_energy_estimate} and taking the supremum over $\tau \in [0,T)$, we have that
\begin{align}
     \sup_{\tau \in[0,T)} \int_{\Omega} \vert u_{\ell}(\tau) \vert^{q+1}  \,\dx\dt
    &\le  
     c(q,L,T) \int_\Omega \vert Du_{\ast} \vert^{p} + \vert u_{\ast} \vert^{p} + G \,\dx\nonumber \\
    &\phantom{=}
	+ 4 \int_\Omega |u_o|^{q+1} \,\dx + c(q) \int_\Omega \vert u_{\ast} \vert^{q+1} \,\dx\nonumber\\
    &=:C_2^{q+1}.
    \label{ineq:L^q+1_estimate_for_u_l}
\end{align}
Combining \eqref{ineq:L^p_estimate_for_derivative} and \eqref{ineq:L^q+1_estimate_for_u_l}, we conclude that the sequence $(u_{\ell})_{\ell \in \mathds{N}}$ is uniformly bounded in $L^{\infty}(0,T;L^{q+1}(\Omega, \mathds{R}^{N})) \cap L^{p}(0,T;W_{u_{\ast}}^{1,p}(\Omega, \mathds{R}^{N}))$. Consequently, there exists a subsequence, denoted by $\mathfrak{K} \subset \mathds{N}$, and a corresponding limit map
\begin{align}\label{limit_map_general_domains}
    u \in L^{\infty}(0,T;L^{q+1}(\Omega, \mathds{R}^{N})) \cap L^{p}(0,T;W_{u_{\ast}}^{1,p}(\Omega, \mathds{R}^{N}))
\end{align}
such that
\begin{equation}
	\left\{
	\begin{array}{ll}
		u_\ell \wsto u
		&\text{weakly$^\ast$ in $L^\infty(0,T;L^{q+1}(\Omega,\R^N))$ as $\mathfrak{K} \ni \ell \to \infty$,} \\[5pt]
		u_\ell \wto u
		&\text{weakly in $L^p(0,T;W_{u_{\ast}}^{1,p}(\Omega,\R^N))$ as $\mathfrak{K} \ni \ell \to \infty$.}
	\end{array}
	\right.
	\label{weak_convergence_to_limit_map_general_domains}
\end{equation}

\subsection{Boundary values}
In this section, we show that the limit map in~\eqref{limit_map_general_domains} admits the correct boundary values in the sense that $u \equiv u_{\ast}$ a.e.~in $\Omega_{T}\setminus \overline{E}$, and consequently that $u \in V_{q}^{p}(E)$. We proceed as in \cite[Section 5.1.3]{BDSS18}.
By construction, we know that $u_\ell \equiv u_{\ast}$ a.e.~in $\Omega_{T} \setminus E^{(\ell)}$ and $E \subset E^{(\ell)}$, where $E^{(\ell)}$ is the set defined in \eqref{cylinders_superset_general_domain}.
Consider an arbitrary, but fixed point $z_{0}:= (x_{0},t_{0}) \in \Omega_T \setminus \overline{E}$. 
Again, we use the notation $\Lambda_{\delta}(t_{0}):=(t_{0}-\delta, t_{0} +\delta) \cap [0,T)$ and $Q_{\delta}(z_{0}):=B_{\delta}(x_{0}) \times \Lambda_{\delta}(t_{0})$. Now we claim that
\begin{align}\label{claim:cylinder Q_epsilon_general_domain_bound}
    \exists \,\epsilon >0, \ell_{0} \in \mathds{N}: \forall  \, \ell \ge \ell_{0}: Q_{\epsilon}(z_{0}) \subset \Omega_{T} \setminus \overline{E^{(\ell)}}
\end{align}
holds.
Since $\Omega_T \setminus \overline{E}$ is open, there exists an $\epsilon > 0$ such that $Q_{2\epsilon}(z_{0}) \subset \Omega_{T} \setminus \overline{E}$. 
In particular, this implies that $B_{2\epsilon} (x_{0}) \subset \Omega \setminus \overline{E^{t}}$ for all $t \in \Lambda_{2\epsilon}(t_{0})$.
However, this yields
\begin{align*}
    B_{2\epsilon}(x_{0}) \subset \bigcap_{t \in \Lambda_{2\epsilon}(t_{0})} \big( \Omega \setminus \overline{E^{t}} \big) = \Omega \mathbin{\Big\backslash} \bigcup_{t \in \Lambda_{2\epsilon}(t_{0})} \overline{E^{t}}
\end{align*}
and therefore, we have that
\begin{align}\label{ineq:set_bound_for_B_epsilon}
    B_{\epsilon}(x_{0}) \subset  \Omega \mathbin{\Big\backslash} \overline{\bigcup_{t \in \Lambda_{2\epsilon}(t_{0})} \overline{E^{t}}} \subset \Omega \mathbin{\Big\backslash} \overline{\bigcup_{t \in \Lambda_{2\epsilon}(t_{0})} E^{t}}.
\end{align}
Further, assume that $\ell > \tfrac{T}{\epsilon}$, which implies that $h_{\ell} < \epsilon$. Then, there exist $1 \le i_{0} < i_{1} \le \ell$ such that
\begin{align}\label{ineq: time_interval_union_bound}
    \Lambda_{\epsilon}(t_{0}) \subset \bigcup_{i_{0} \le i \le i_{1}} I_{\ell, i} \subset \Lambda_{2 \epsilon} (t_{0}).
\end{align}
Combining \eqref{ineq:set_bound_for_B_epsilon} and \eqref{ineq: time_interval_union_bound} gives us that
\begin{align*}
    B_{\epsilon}(x_{0}) \times I_{\ell, i}
    &\subset
    \Bigg[ \Omega \mathbin{\Big\backslash} \overline{\bigcup_{t \in \Lambda_{2\epsilon}(t_{0})} E^{t}} \Bigg] \times I_{\ell, i}
    \subset
    \Bigg[ \Omega \mathbin{\Big\backslash} \overline{\bigcup_{t \in I_{\ell,i}} E^{t}} \Bigg] \times I_{\ell, i} \\
    &=
    (\Omega \times I_{\ell , i}) \mathbin{\Big\backslash} \Bigg[ \,\overline{\bigcup_{t \in I_{\ell,i}} E^{t}} \times I_{\ell, i} \Bigg]
    =
    (\Omega \times I_{\ell, i}) \setminus \overline{Q_{\ell, i}}
\end{align*}
for all $i_0 \leq i \leq i_1$.
This implies that
\begin{align*}
    Q_{\epsilon}(z_{0}) 
    \subset 
    B_{\epsilon}(x_{0}) \times \bigcup_{i_{0} \le i \le i_{1}} I_{\ell, i} 
    \subset 
    \bigcup_{i_{0} \le i \le i_{1}} \Big[ (\Omega \times I_{\ell, i}) \setminus \overline{Q_{\ell, i}} \Big] 
    \subset 
    \Omega_{T} \setminus \overline{E^{(\ell)}}.
\end{align*}
Therefore, \eqref{claim:cylinder Q_epsilon_general_domain_bound} holds for any $\ell_{0} \in \mathds{N}$ that satisfies $\ell_{0} > \frac{T}{\epsilon}$. 
By \eqref{weak_convergence_to_limit_map_general_domains}$_{2}$ and the
fact that $u_{\ast} \in L^p(Q_{\epsilon}(z_0), \mathds{R}^{N})$, we
deduce that $u_{\ell} - u_{\ast} \wto u - u_{\ast}$ weakly in
$L^{p}(Q_{\epsilon}(z_{0}), \mathds{R}^{N})$ as $\mathfrak{K} \ni \ell \to \infty$.
Thus, by lower semicontinuity, we conclude that
\begin{align*}
    \iint_{Q_{\epsilon}(z_{0})} \vert u -u_{\ast} \vert^{p} \,\dx\dt 
    \le 
    \liminf_{\ell \to \infty} \iint_{Q_{\epsilon}(z_{0})} \vert u_{\ell}-u_{\ast} \vert^{p} \,\dx\dt 
    =0.
\end{align*}
This means that $u \equiv u_{\ast}$ a.e.~in $Q_{\epsilon}(z_{0})$.
Since $z_0\in \Omega_{T} \setminus \overline{E}$ was arbitrary, we find that $u \equiv u_{\ast}$ a.e.~in $\Omega_{T} \setminus \overline{E}$. 
Finally, $u$ admits the correct boundary values by the following Fubini type argument.
By assumption \eqref{eq:boundary_E_zero_measure}, and since $E$ is relatively open in $\Omega_T$, it follows that
\begin{align*}
    0=\L^{n+1}(\partial E) = \L^{n+1} \big( \overline{E} \setminus E \big) = \int_{0}^{T} \L^{n} \big( (\Omega \times \{ t \}) \cap \big( \overline{E} \setminus E \big) \big) \, \dt,
\end{align*}
and hence, $\L^{n} \big( (\Omega \times \{ t \}) \cap \big( \overline{E} \setminus E \big) \big)=0$ for a.e.~$t \in [0,T)$.
In combination with the identity
\begin{align*}
    \big( \Omega \setminus E^{t} \big) \times \{t\} = \big( (\Omega \times \{t\}) \setminus \overline{E} \big) \cup \big( (\Omega \times \{ t \}) \cap \big(\overline{E} \setminus E \big) \big)
\end{align*}
we obtain that $u(t)=u_{\ast}$ a.e.~in $\Omega \setminus E^{t}$ for a.e.~$t \in [0,T)$.
Therefore, $u(t) \in V_{t}$ for the corresponding times $t$, which
means that $u$ admits the correct boundary values in the sense $u \in V_{q}^{p}(E)$.

\subsection{Convergence almost everywhere}
Our aim in this section is to develop the convergence assertion
\begin{align}\label{almost everywhere weakly star convergence with limit}
    \power{u_\ell}{q} \wsto \power{u}{q}
	\quad \text{weakly$^\ast$ in $L^{\infty}\Big(0,T;L^{\frac{q+1}{q}}\big(\Omega,\R^N\big)\Big)$ as $\mathfrak{K} \ni \ell \to \infty$.}
\end{align}
By \eqref{ineq:L^q+1_estimate_for_u_l} we know that $\big( \power{u_\ell}{q} \big)_{\ell \in \mathds{N}}$ is bounded in this function space.
Therefore, for some limit map
\begin{align*}
    w \in L^{\infty}\Big(0,T;L^{\frac{q+1}{q}}\big(\Omega,\R^N\big)\Big),
\end{align*}
and after passing to another subsequence if necessary, we have that
\begin{align}\label{almost everywhere weakly star convergence no limit }
    \power{u_\ell}{q} \wsto w
	\quad \text{weakly$^\ast$ in $L^{\infty} \Big( 0,T;L^{\frac{q+1}{q}}(\Omega,\R^N) \Big)$ as $\mathfrak{K} \ni \ell \to \infty$}.
\end{align}
It remains to show that $w= \power{u}{q}$.
To this end, we derive a variational inequality for the maps $u_\ell$.
We consider a comparison map $v\in V^p_q(E)$ with $\partial_t v \in L^{q+1}(\Omega_{T}, \mathds{R}^{N})$.
For $\tau \in [0,T)$ we choose $k< \ell$ such that $\tau \in I_{\ell,k+1}$.
We denote the restriction of $v$ to the subcylinder $\Omega \times I_{\ell,i}$ by
$v_{\ell, i} \colon \Omega \times I_{\ell, i} \to \mathds{R}^{N}$,
where $i \in \{1,\ldots,k+1\}$.
The maps $v_{\ell,i}$ are admissible comparison maps in
\eqref{ineq:var_ineq_seq_var_sol_u_li}, since
\eqref{cylinders_superset_general_domain} implies $v=u_*$
a.e.~in $(\Omega\setminus E_{\ell,i})\times I_{\ell,i}$.
Recalling~\eqref{initial_condition_u_l,i_generarl_domain}, the definition of the initial values of $u_{\ell,i}$, we obtain the inequality
\begin{align*}
    \int_{\Omega} \b \Big[ u_{\ell,i}^{(0)},v(t_{\ell,i-1}) \Big] \,\dx &= \int_{E_{\ell,i}} \b[u_{\ell,i-1}(t_{\ell, i-1}),v(t_{\ell,i-1})] \,\dx \\
    &\le \int_{\Omega} \b[u_{\ell,i-1}(t_{\ell, i-1}),v(t_{\ell,i-1})] \,\dx.
\end{align*}
Therefore, the variational inequality
\eqref{ineq:var_ineq_seq_var_sol_u_li} implies that
\begin{align}
    \int_{\Omega } &\b[u_{\ell,i}(\widetilde{\tau}),v_{\ell,i}(\widetilde{\tau})] \,\dx
    +\iint_{\Omega \times [t_{\ell,i-1}, \widetilde{\tau})} f(x,u_{\ell,i}, Du_{\ell,i}) \,\dx\dt
    \nonumber \\ &\le
    \iint_{\Omega \times [t_{\ell,i-1}, \widetilde{\tau})} f(x,v_{\ell,i}, Dv_{\ell,i}) \,\dx\dt
    \label{ineq:var_ineq_for_u_l,i_in_general_domains} \\ &\phantom{=}
    + \iint_{\Omega \times [t_{\ell,i-1}, \widetilde{\tau})} \partial_t v_{\ell,i} \cdot \big( \power{v_{\ell,i}}{q}-\power{u_{\ell,i}}{q} \big) \,\dx\dt
    \nonumber \\
    &\phantom{=} 
    + \int_{\Omega } \b[u_{\ell,i-1}(t_{\ell, i-1}),v(t_{\ell,i-1})] \,\dx
    \nonumber
\end{align}
for any $\widetilde{\tau} \in I_{\ell,i}$ and $i \in \{1,\ldots,k+1\}$.
Note that in particular $u_{\ell,0}(t_{\ell,0}) \equiv u_o$ holds, where $t_{\ell,0}=0$.
Summing up \eqref{ineq:var_ineq_for_u_l,i_in_general_domains} over $i \in \{1,\ldots,k\}$ with the choice $\widetilde{\tau}=t_{\ell,i}$ (which yields $[t_{\ell,i-1}, \widetilde{\tau} ) \equiv  I_{\ell,i}$), and adding \eqref{ineq:var_ineq_for_u_l,i_in_general_domains} for $i=k+1$ with
$\widetilde{\tau} = \tau$ we obtain the desired variational inequality for $u_\ell$, i.e.,
\begin{align}\label{ineq:var_ineq_for_u_ell}
    \int_{\Omega } \b[u_{\ell}&(\tau),v(\tau)] \,\dx+\iint_{\Omega_{\tau}} f(x,u_{\ell}, Du_{\ell}) \,\dx\dt \nonumber \\ &\le \iint_{\Omega_{\tau}} f(x,v, Dv) \,\dx\dt + \iint_{\Omega_{\tau}} \partial_t v \cdot \big( \power{v}{q}-\power{u_{\ell}}{q} \big) \,\dx\dt  + \int_{\Omega } \b[u_o,v(0)] \,\dx 
\end{align}
for any $\tau \in [0,T)$ and any $v \in V^p_q(E)$ with $\partial_t v \in L^{q+1}(\Omega_{T}, \mathds{R}^{N})$.

At this stage, we construct suitable comparison maps for the functions $u_\ell$.
For $\lambda>0$, let
\begin{align*}
    w_{\ell, \lambda}:= u_\ast+[u_\ell - u_{\ast}]_{\lambda} \quad\text{ and }\quad w_{\lambda}:= u_\ast+[u - u_{\ast}]_{\lambda},
\end{align*}
where $[u_\ell - u_{\ast}]_{\lambda}$ and $[u - u_{\ast}]_{\lambda}$ denote the mollifications in time of $u_\ell-u_\ast$ and $u-u_\ast$,
respectively, defined according to
\eqref{eq:time_mollification}  with initial values $v_o\equiv0$. 
Since $u_{\ell}, u \in L^{\infty}(0,T;L^{q+1}(\Omega, \mathds{R}^{N}))
\cap L^{p}(0,T;W^{1,p}_{u_\ast}(\Omega, \mathds{R}^{N}))$, the maps
$w_{\ell,\lambda}$ and $w_\lambda$ are contained in the same
function space for any $\lambda >0$. Furthermore, by the definition of
the mollification we have 
\begin{align*}
    \partial_t w_{\ell, \lambda} = \partial_t [u_\ell - u_{\ast}]_{\lambda} = - \tfrac{1}{\lambda} \big([u_\ell - u_{\ast}]_{\lambda}- (u_\ell - u_{\ast}) \big),
\end{align*}
which implies that 
\begin{align}\label{correct_space_for_time_derivative_of_mollification}
\partial_t w_{\ell, \lambda} \in L^{\infty}\big( 0,T;L^{q+1}(\Omega, \mathds{R}^{N}) \big) \cap L^{p}\big( 0,T;W^{1,p}_0(\Omega, \mathds{R}^{N}) \big). 
\end{align}
Further, by Lemma \ref{lem:mollification:estimate_and_continuous_convergence} and \eqref{ineq:L^p_estimate_for_derivative}, we bound the norm of the spatial derivatives by
\begin{align}\label{ineq:space_derivative_estimate}
  \Vert Dw_{\ell, \lambda} \Vert_{L^{p}(\Omega_T, \mathds{R}^{N})}
  &= \Vert Du_\ast + D[u_\ell - u_{\ast}]_{\lambda} \Vert_{L^{p}(\Omega_T, \mathds{R}^{N})} \nonumber \\
   &\le \Vert Du_{\ell} \Vert_{L^{p}(\Omega_T,
     \mathds{R}^{N})}+2T^{\frac1p}\Vert Du_{\ast}\Vert_{L^{p}(\Omega,
     \mathds{R}^{N})}  \nonumber\\
  &\le C_1+2T^{\frac1p}\Vert Du_{\ast}\Vert_{L^{p}(\Omega,
     \mathds{R}^{N})}.
\end{align}
Now, for $z_{0}:=(x_{0}, t_{0}) \in E$ with $t_0>0$ we consider a backward-in-time cylinder
\begin{align*}
    Q_{\rho, s}(z_{0}):= Q_{\rho, s}(x_{0}, t_{0}) = B_{\rho}(x_{0}) \times (t_{0}-s,t_{0}) \Subset E
\end{align*}
with $\rho,s\in(0,1]$. 
Define $\epsilon_{0}:=\dist(Q_{\rho, s}(z_{0}), \partial E)$ and let $\epsilon \in (0, \epsilon_{0})$.
We introduce the following cut-off functions.
First, for the space variable let $\phi_{\epsilon} \in
C^{1}(\mathds{R}^{N}, [0,1])$ be such that  $\phi_{\epsilon} \equiv 1$ in $B_{\rho}(x_{0})$, $\phi_{\epsilon} \equiv 0$ in $\mathds{R}^{N} \setminus B_{\rho+\epsilon}(x_{0})$ and $\vert D\phi_{\epsilon}(x) \vert \le \tfrac{2}{\epsilon}$. 
Further, for the time variable define $\psi_{\epsilon} \in C^{1}((0,T), [0,1])$ such that $\psi_{\epsilon} \equiv 1$ in $(t_{0}-s,t_{0})$, $\psi_{\epsilon} \equiv 0$ in $(0,T) \setminus (t_{0}-s-\epsilon,t_{0}+\epsilon)$ and $\vert \psi_{\epsilon}^{\prime}(t) \vert \le \tfrac{2}{\epsilon}$.
Combining both cut-off functions, we define $\eta_{\epsilon}(x,t):=\phi_{\epsilon}(x)\psi_{\epsilon}(t)$.
Furthermore, we set
\begin{align*}
    v_{\ell, \lambda, \epsilon}:= u_\ast +
  \eta_{\epsilon}(w_{\ell,\lambda}-u_\ast)
  \quad\text{ and }\quad
  v_{\lambda, \epsilon}:=u_\ast + \eta_{\epsilon}(w_{ \lambda}-u_\ast).
\end{align*}
In particular, we have that $v_{\ell, \lambda, \epsilon} \equiv u_{\ast}$ a.e.~in $\Omega_{T} \setminus E$.
Now fix both $\lambda > 0 $ and $\epsilon \in (0,\epsilon_{0})$. 
By \eqref{weak_convergence_to_limit_map_general_domains}$_2$ and Lemma \ref{lem:mollification_weak_convergence}, we find that
\begin{align}\label{weak_limit_of_v_l_lambda_epsilon}
    v_{\ell, \lambda, \epsilon} \wto v_{ \lambda, \epsilon} \text{
  weakly in $L^p\big( 0,T;W_{u_{\ast}}^{1,p}(\Omega,\R^N) \big)$ \quad as $\mathfrak{K}\ni\ell \to \infty$.}
\end{align}
Moreover, for the time derivative we compute
\begin{align}\label{time_derivative_combined_mollification}
    \partial_t v_{\ell, \lambda, \epsilon} &= \phi_{\epsilon}  \psi_{\epsilon}^{\prime}(t) (w_{\ell, \lambda}-u_\ast) + \eta_{\epsilon} \partial_t w_{\ell, \lambda}  \nonumber \\
    &= \phi_{\epsilon}  \psi_{\epsilon}^{\prime}(t)(w_{\ell,\lambda}-u_\ast) - \tfrac{1}{\lambda} \eta_{\epsilon} ([u_{\ell} - u_{\ast}]_{\lambda} - (u_{\ell} - u_{\ast})) \nonumber \\
    &= \phi_{\epsilon} \psi_{\epsilon}^{\prime}(t) (w_{\ell, \lambda}-u_\ast) + \tfrac{1}{\lambda} \eta_{\epsilon} (u_\ell-w_{\ell, \lambda}).
\end{align}
On the one hand, this implies that
\begin{align}\label{correct_space_for_time_derivative_of_mollification_combined}
\partial_t v_{\ell, \lambda, \epsilon} \in L^{\infty} \big( 0,T;L^{q+1}(\Omega, \mathds{R}^{N}) \big) \cap L^{p} \big( 0,T;W^{1,p}(\Omega, \mathds{R}^{N}) \big)
\end{align}
for any $\lambda>0$, $\epsilon \in (0,\epsilon_0)$, and $\ell \in \N$.
On the other hand, using again \eqref{weak_convergence_to_limit_map_general_domains}$_2$ and Lemma~\ref{lem:mollification:estimate_and_continuous_convergence}, we have that $(\partial_t v_{\ell, \lambda, \epsilon})_{\ell \in \N}$ is bounded in the space $L^p\big(\Omega_{T},\R^{N}\big)$ for any fixed $\lambda>0$ and $\epsilon \in (0,\epsilon_0)$.
Thus, combining this with \eqref{weak_limit_of_v_l_lambda_epsilon} and applying Rellich's theorem, we conclude that
\begin{align}\label{strong_limit_of_v_l_lambda_epsilon}
     v_{\ell, \lambda, \epsilon} \to v_{ \lambda, \epsilon} \text{
  strongly in $L^p(\Omega_{T},\R^{Nn})$ \quad as $\mathfrak{K}\ni\ell \to \infty$.}
\end{align}
In this context, note that it is not necessary to pass to another subsequence to obtain the strong convergence in \eqref{strong_limit_of_v_l_lambda_epsilon}, since the limit is already determined by the weak convergence in \eqref{weak_limit_of_v_l_lambda_epsilon}.

Observing that $v_{\ell, \lambda, \epsilon}$ is an admissible
comparison map in the variational inequality
\eqref{ineq:var_ineq_for_u_ell}, omitting both non-negative terms on
the left-hand side and recalling~\eqref{time_derivative_combined_mollification}, we infer
\begin{align*}
    0 &\le
    \iint_{\Omega_{\tau}} f(x,v_{\ell, \lambda, \epsilon}, Dv_{\ell, \lambda, \epsilon}) \,\dx\dt
    + \iint_{\Omega_{\tau}} \phi_{\epsilon} (w_{\ell, \lambda}-u_\ast) \psi_{\epsilon}^{\prime} \cdot \big( \power{v_{\ell, \lambda, \epsilon}}{q}-\power{u_{\ell}}{q} \big) \,\dx\dt
    \nonumber\\ &\phantom{=}
    +\iint_{\Omega_{\tau}} \tfrac{1}{\lambda} \eta_{\epsilon}(u_{\ell} - w_{\ell, \lambda}) \cdot \big( \power{v_{\ell, \lambda, \epsilon}}{q} - \power{u_{\ell}}{q} \big) \,\dx\dt
    + \int_{\Omega } \b[u_o,u_{\ast}] \,\dx
\end{align*}
for any $\tau\in[0,T)$. 
Note that $\b[u_o,u_{\ast}] \le \tfrac{2q}{q+1} \vert u_o \vert^{q+1}
+\tfrac{2}{q+1} \vert u_{\ast} \vert^{q+1} $. By definition, the
support of $\eta_{\epsilon}$ is contained in the closure of $\widehat
Q:=B_{\rho+\eps}(x_o)\times(t_o-s-\eps,t_o+\eps)$. Moreover, we have
$\eta_\eps\equiv1$ on $Q_{\rho,s}:=Q_{\rho,s}(z_o)$, which implies
$v_{\ell,\lambda,\eps}=w_{\ell,\lambda}$ on this cylinder. 
Therefore, splitting up the penultimate integral on the right-hand side of
the preceding inequality and multiplying the inequality by $\lambda$, we obtain that
\begin{align}\label{ineq:Q_rho_s_estimate_with_split_integrals}
    \iint_{Q_{\rho, s}} &(w_{\ell, \lambda} - u_\ell ) \cdot \big( \power{w_{\ell, \lambda}}{q} - \power{u_{\ell}}{q} \big) \,\dx\dt \nonumber \\
    &\le c(q) \lambda \int_{\Omega} \vert u_o \vert^{q+1} \,\dx + c(q) \lambda \int_{\Omega} \vert u_{\ast} \vert^{q+1} \,\dx \nonumber \\
    &\phantom{=}
    + \lambda \iint_{\Omega_{\tau}} f(x,v_{\ell, \lambda, \epsilon}, Dv_{\ell, \lambda, \epsilon}) \,\dx\dt \nonumber \\
    &\phantom{=}
    + \lambda \iint_{\Omega_{\tau}} \phi_{\epsilon} (w_{\ell, \lambda}-u_\ast) \psi_{\epsilon}^{\prime} \cdot \big( \power{v_{\ell, \lambda, \epsilon}}{q}-\power{u_{\ell}}{q} \big) \,\dx\dt \nonumber \\
    &\phantom{=}
    + \iint_{\widehat Q \setminus Q_{\rho, s}} \eta_{\epsilon}(u_{\ell} - w_{\ell, \lambda}) \cdot \big( \power{v_{\ell, \lambda,\eps}}{q} - \power{u_{\ell}}{q} \big) \,\dx\dt \nonumber \\
    &=: c(q) \lambda \int_{\Omega} \vert u_o \vert^{q+1} \,\dx + c(q) \lambda \int_{\Omega} \vert u_{\ast} \vert^{q+1} \,\dx+\lambda \mathrm{I}_1 +\lambda \mathrm{I}_2 + \mathrm{I}_3, 
\end{align}
where the definitions of $\mathrm{I}_1$--$\mathrm{I}_3$ are apparent
from the context. Our goal now is to estimate each of these terms
separately. We start off with the first integral. By means of the $p$-growth condition \eqref{eq:integrand}$_3$, the fact that $0 \le \eta \le 1$, H\"older's and
Poincaré's inequality, and by the bound \eqref{ineq:space_derivative_estimate}, we derive the estimate
\begin{align}
    \vert I_{1} \vert
    &\le
    L \iint_{\Omega_T}  \vert Dv_{\ell, \lambda, \epsilon}\vert^{p} +  \vert v_{\ell, \lambda, \epsilon}\vert^{p} + G \,\dx\dt  \nonumber\\
    &\le
    c(p)L \iint_{\Omega_T}  \vert Dw_{\ell,\lambda}\vert^{p} + |Du_\ast|^p + |u_\ast|^p + G \,\dx\dt
    \nonumber\\&\phantom{=}
    +c(p) L \iint_{\Omega_T} \vert (w_{\ell, \lambda}-u_\ast) \otimes D\phi_{\epsilon} \vert^{p} + \vert w_{\ell,\lambda}-u_\ast \vert^{p} \,\dx\dt
    \nonumber \\ &\le
    c(p)L \iint_{\Omega_T}  \vert Dw_{\ell,\lambda}\vert^{p} + |Du_\ast|^p + |u_\ast|^p + G \,\dx\dt
    \nonumber\\&\phantom{=}
    + c(p) L \Big( 1 + \tfrac{2^{p}}{\epsilon^{p}}  \Big) \iint_{\Omega_T} \vert w_{\ell, \lambda} - u_\ast \vert^{p}\,dx\dt
    \nonumber \\ &\le
    c(p,\Omega)L \Big( 1 + \tfrac{2^{p}}{\epsilon^{p}}  \Big)\iint_{\Omega_T}  \vert Dw_{\ell,\lambda}\vert^{p} + |Du_\ast|^p + |u_\ast|^p + G \,\dx\dt
    \nonumber \\ &\le
    c(p,\Omega)L \Big( 1 + \tfrac{2^{p}}{\epsilon^{p}} \Big)\bigg(C_1^p+  T \int_{\Omega}  |Du_\ast|^p + |u_\ast|^p + G \,\dx \bigg). 
    \label{ineq:Q_rho_s_estimate_1}
\end{align}
To proceed with the second integral, note that by Lemma~\ref{lem:mollification:estimate_and_continuous_convergence}, we have that
\begin{align}
    \Vert w_{\ell, \lambda} - u_\ast\Vert_{L^{m}(\Omega_{T},
  \mathds{R}^{N})}
  &=
   \big \Vert [u_\ell-u_\ast]_\lambda \big\Vert_{L^{m}(\Omega_{T},
  \mathds{R}^{N})}\nonumber\\
  &\le
    \Vert u_{\ell}-u_\ast \Vert_{L^{m}(\Omega_{T}, \mathds{R}^{N})}
    \label{ineq:mollification_estimate_in_L^q+1}
\end{align}
holds for any $m\in \Big[ 1,\frac{p(n+q+1)}{n} \Big]$. 
Further, taking into account that $0\le \phi_{\epsilon} \le 1$, $
\vert \psi_{\epsilon}^{\prime} \vert \le \tfrac{2}{\epsilon}$ and
$0\le \eta_{\epsilon} \le 1$, and using Young's inequality,
Lemma~\ref{lem:mollification:estimate_and_continuous_convergence},
\eqref{ineq:mollification_estimate_in_L^q+1} with $m=q+1$ and~\eqref{ineq:L^q+1_estimate_for_u_l}, we conclude that
\begin{align}\label{ineq:Q_rho_s_estimate_2}
  \vert I_{2} \vert
  &\le \tfrac{2}{\eps}\iint_{\Omega_T}
    \vert w_{\ell, \lambda}-u_\ast \vert \big(\vert v_{\ell,\lambda,\eps}
    \vert^{q}
    + \vert u_\ell\vert^{q} \big) \,\dx\dt \nonumber \\
  &\le
  \tfrac{c(q)}{\eps}\iint_{\Omega_T} \big(\vert w_{\ell,
  \lambda}- u_\ast \vert\big) \big(\vert w_{\ell, \lambda}-u_\ast
  \vert^{q} +\vert u_\ast\vert^q + \vert u_\ell \vert^{q} \big) \,\dx\dt \nonumber \\
    &\le \tfrac{c(q)}{\epsilon} \iint_{\Omega_T} \vert w_{\ell,
      \lambda}-u_\ast \vert^{q+1} + \vert u_\ast\vert^{q+1}+\vert u_\ell \vert^{q+1} \,\dx\dt \nonumber \\
    &\le \tfrac{c(q)}{\epsilon} \iint_{\Omega_T} \vert u_{\ell} \vert^{q+1} + \vert u_{\ast} \vert^{q+1} \,\dx\dt \nonumber \\
    &\le \tfrac{c(q)T}{\epsilon} \bigg(C_2^{q+1}+ \int_{\Omega} \vert u_{\ast} \vert^{q+1} \,\dx\bigg).
\end{align}
Now we devote ourselves to the third integral.
First, we use the fact that $0\le \eta_{\epsilon} \le 1$, Young's inequality and H\"older's inequality with exponents $\tfrac{m}{q+1}$ and $\tfrac{m}{m-(q+1)}$, where $m:= \tfrac{p(n+q+1)}{n}$.
Regarding this step, it is important to point out that $m>q+1$, since we have assumed that $p>\max \Big\{ 1, \tfrac{n(q+1)}{n+q+1} \Big\}$.
Next, we use \eqref{ineq:mollification_estimate_in_L^q+1} and finally, the Gagliardo--Nirenberg inequality and the bounds \eqref{ineq:L^p_estimate_for_derivative} and \eqref{ineq:L^q+1_estimate_for_u_l}.
Indeed, proceeding in this way, we infer
\begin{align}\label{ineq:Q_rho_s_estimate_3}
    \vert I_{3} \vert
    &\le
    \iint_{\widehat Q \setminus Q_{\rho, s}} (\vert u_\ell \vert + \vert w_{\ell, \lambda} \vert) \big( \vert v_{\ell, \lambda,\eps} \vert^{q} + \vert u_\ell \vert^{q} \big) \,\dx\dt \nonumber \\
    &\le
    c(q) \iint_{\widehat Q \setminus Q_{\rho, s}} \vert w_{\ell,\lambda}-u_\ast \vert^{q+1} + \vert u_\ell \vert^{q+1}+\vert u_\ast\vert^{q+1} \,\dx\dt \nonumber \\
    & \le
    c(q) \big\vert \widehat Q \setminus Q_{\rho, s} \big\vert^{\frac{m-(q+1)}{m}}
    \nonumber \\ &\phantom{=} \cdot
      \Big(\|w_{\ell,\lambda}-u_\ast\|_{L^{m}(\widehat Q,\R^N)}^{q+1}
      +\|u_\ell\|_{L^{m}(\widehat Q,\R^N)}^{q+1}
      +\|u_\ast\|_{L^{m}(\widehat Q,\R^N)}^{q+1}\Big)
    \nonumber \\
    & \le
    c(n,q) \eps^{\frac{m-(q+1)}{m}}
      \Big(\|u_\ell\|_{L^{m}(\widehat Q,\R^N)}^{q+1}
      +\|u_\ast\|_{L^{m}(\widehat Q,\R^N)}^{q+1}\Big)
    \nonumber \\
    & \le c(n,p,q) \epsilon^{\frac{m-(q+1)}{m}}\Bigg[  \bigg(
      \sup_{t \in (0,T)} \int_{\Omega} \vert u_\ell \vert^{q+1} \,\dx
      \bigg)^{\frac{p}{n}}
      \iint_{\Omega_{T}} \vert Du_\ell \vert^{p} \,\dx\dt  \nonumber \\
    &\phantom{=}
    \qquad\qquad\qquad\qquad + T \bigg( \int_{\Omega} \vert u_{\ast} \vert^{q+1} \,\dx \bigg)^{\frac{p}{n}} \int_\Omega \vert Du_{\ast} \vert^{p} \,\dx \Bigg]^{\frac{q+1}{m}} \nonumber \\
    &\le
    c(n,p,q) \epsilon^{\frac{m-(q+1)}{m}}
    \bigg(C_1^p C_2^\frac{(q+1)p}{n}
    +\|u_\ast\|_{L^{q+1}(\Omega,\R^N)}^\frac{(q+1)p}{n}\|Du_\ast\|_{L^{p}(\Omega,\R^N)}^p\bigg)^{\frac{q+1}{m}}.
\end{align}
Inserting \eqref{ineq:Q_rho_s_estimate_1}, \eqref{ineq:Q_rho_s_estimate_2} and \eqref{ineq:Q_rho_s_estimate_3} into \eqref{ineq:Q_rho_s_estimate_with_split_integrals} yields
\begin{align*}
    \iint_{Q_{\rho, s}} &(w_{\ell, \lambda} - u_\ell ) \cdot \big( \power{w_{\ell, \lambda}}{q} - \power{u_{\ell}}{q} \big) \,\dx\dt 
    \le C \big(1 + \epsilon^{-p} \big) \lambda
    + C \epsilon^{\frac{m-(q+1)}{m}},
\end{align*}
with a positive exponent $\tfrac{m-(q+1)}{m}>0$ and a constant $C$ depending only on $n,p,q,$ $\nu,L$, $\Omega,T,u_\ast,u_o$ and $G$.
Now, by Lemma~\ref{lem:technical_lemma_2}, H\"older's
inequality, Lemma~\ref{lem:technical_lemma},
\eqref{ineq:mollification_estimate_in_L^q+1}, the preceding inequality
and~\eqref{ineq:L^q+1_estimate_for_u_l}, we obtain that
\begin{align*}
    &\iint_{Q_{\rho, s}(z_o)} \big\vert w_{\ell, \lambda} -u_\ell \big\vert^{q+1} \,\dx\dt \\
    &\le
    c(q) \iint_{Q_{\rho, s}(z_o)} \Big(|w_{\ell, \lambda}|^\frac{q+1}{2} + |u_\ell|^\frac{q+1}{2} \Big) \Big| \power{w_{\ell, \lambda}}{\frac{q+1}{2}} - \power{u_\ell}{\frac{q+1}{2}} \Big| \,\dx\dt \\
    & \le
    c(q) \bigg( \iint_{Q_{\rho, s}(z_o)}  |w_{\ell, \lambda}|^{q+1} + |u_\ell|^{q+1} \,\dx\dt \bigg)^{\frac{1}{2}}
    \bigg( \iint_{Q_{\rho, s}(z_o)}  \Big\vert \power{w_{\ell, \lambda}}{\frac{q+1}{2}} - \power{u_\ell}{\frac{q+1}{2}} \Big\vert^{2} \,\dx\dt \bigg)^{\frac{1}{2}} \\
    &\le
    c(q) \bigg(\iint_{\Omega_T}|u_\ell|^{q+1}+|u_\ast|^{q+1}\,\dx\dt \bigg)^{\frac{1}{2}}
    \bigg( \iint_{Q_{\rho, s}(z_o)} (w_{\ell, \lambda} - u_\ell ) \cdot \big(\power{w_{\ell, \lambda}}{q} - \power{u_{\ell}}{q}\big) \bigg)^{\frac{1}{2}} \\
    &\leq
    c(n,p,q,\nu,L,\Omega,T,u_\ast,u_o,G) \left( \big(1 + \epsilon^{-p} \big) \lambda + \epsilon^{\frac{m-(q+1)}{m}} \right)^\frac{1}{2}.
\end{align*}
We stress that the constant on the right-hand side is independent of $\ell$.
Therefore, for an arbitrary $\delta > 0$, we can choose first
$\epsilon$ and then $\lambda$ small enough to obtain
\begin{align*}
    \Vert w_{\ell, \lambda} - u_\ell \Vert_{L^{\min\{ p, q+1\}}(Q_{\rho, s}(z_o), \mathds{R}^{N})} \le \tfrac{\delta}{2}
\end{align*}
for any $\ell \in \mathds{N}$.  Furthermore, since $[u-u_\ast]_{\lambda} \to u-u_\ast$ in $L^{\min\{ p, q+1\}}(Q_{\rho, s}(z_0), \mathds{R}^{N})$ as $\lambda \downarrow 0$ by Lemma \ref{lem:mollification:estimate_and_continuous_convergence}, by decreasing $\lambda$ we can assume that
\begin{align*}
    \Vert w_{\lambda} - u \Vert_{L^{\min\{ p, q+1\}}(Q_{\rho, s}(z_o), \mathds{R}^{N})} \le \tfrac{\delta}{2}.
\end{align*}
This implies that
\begin{align*}
    \Vert u_\ell &- u \Vert_{L^{\min\{ p, q+1\}}(Q_{\rho, s}(z_o), \mathds{R}^{N})} \\
    &\le 
    \Vert u_\ell - w_{\ell, \lambda} \Vert_{L^{\min\{ p, q+1\}}(Q_{\rho, s}(z_o), \mathds{R}^{N})} +\Vert w_{\ell, \lambda} - w_{ \lambda} \Vert_{L^{\min\{ p, q+1\}}(Q_{\rho, s}(z_o), \mathds{R}^{N})} \\
    &\phantom{=}
    +\Vert w_{\lambda} - u \Vert_{L^{\min\{ p, q+1\}}(Q_{\rho, s}(z_o), \mathds{R}^{N})} \\
    &\le \delta +  \Vert w_{\ell, \lambda} - w_{\lambda} \Vert_{L^{\min\{ p, q+1\}}(Q_{\rho, s}(z_o), \mathds{R}^{N})} \\
    &= \delta +  \Vert v_{\ell, \lambda, \epsilon} - v_{\lambda, \epsilon} \Vert_{L^{\min\{ p, q+1\}}(Q_{\rho, s}(z_o), \mathds{R}^{N})}
\end{align*}
holds for any $\ell \in \mathds{N}$. In the last line, we have used
that $v_{\ell,\lambda,\epsilon}=w_{\ell,\lambda}$ and
$v_{\lambda,\epsilon}=w_\lambda$ in $Q_{\rho,s}(z_o)$ by definition. By \eqref{strong_limit_of_v_l_lambda_epsilon}, it follows that
\begin{align*}
    \lim_{\mathfrak{K}\ni\ell \to \infty} \Vert u_\ell - u \Vert_{L^{\min\{ p, q+1\}}(Q_{\rho, s}(z_o), \mathds{R}^{N})} \le \delta.
\end{align*}
Since $\delta$ was arbitrary, by passing to a (not relabelled) subsequence we find that 
\begin{align*}
    u_\ell \to u \text{ strongly in $L^{\min\{ p, q+1\}}(Q_{\rho, s}(z_o), \mathds{R}^{N})$ as $\ell \to \infty$}.
\end{align*}
Furthermore, since $z_o \in E$ was arbitrary, by a diagonal argument we can again pass to a (not relabeled) subsequence, such that 
\begin{align}\label{almost_everywhere_convergence_general_domains}
    u_\ell \to u \text{ \quad a.e.~in $E$ as $\ell \to \infty$}.
\end{align}
Combining \eqref{almost_everywhere_convergence_general_domains} and \eqref{almost everywhere weakly star convergence no limit } implies \eqref{almost everywhere weakly star convergence with limit}.

\subsection{Variational inequality for the limit map}
In this section, our aim is to show that the limit map from \eqref{limit_map_general_domains} satisfies the variational inequality \eqref{eq:variational_inequality}. 
Let $v \in V^p_q(E)$ with $\partial_t v \in L^{q+1}(\Omega_{T},
\mathds{R}^{N})$ be given, which implies in particular $v \equiv
u_\ast$ a.e.~in $\Omega_T \setminus E$ and $v \in C^{0}([0,T]; L^{q+1}(\Omega, \mathds{R}^{N}))$. Moreover, $v$ is an admissible comparison map in \eqref{ineq:var_ineq_for_u_ell}.
To be able to treat the boundary term at time $\tau$, we integrate \eqref{ineq:var_ineq_for_u_ell} over $\tau \in [t_o, t_o + \delta)$, for $\delta \in (0,T-t_o)$, and divide the resulting inequality by $\delta$.
Since $f$ is non-negative, this leads to
\begin{align}
    \bint_{t_o}^{t_o+\delta} \int_{\Omega } &\b[u_{\ell}(\tau),v(\tau)] \,\dx \,\d\tau
    +\iint_{\Omega_{t_o}} f(x,u_{\ell}, Du_{\ell}) \,\dx\dt
    \nonumber \\ &\le
    \iint_{\Omega_{t_o+\delta}} f(x,v, Dv) \,\dx\dt
    +\bint_{t_o}^{t_o+\delta} \iint_{\Omega_{\tau}} \partial_t v \cdot \big( \power{v}{q}-\power{u_{\ell}}{q} \big) \,\dx\dt \,\d\tau
    \label{general_domains:var_ineq_before_ell_limit}\\ &\phantom{=}
    + \int_{\Omega } \b[u_o,v(0)] \,\dx.
    \nonumber
\end{align}
First, since $f$ satisfies \eqref{eq:integrand}, the second term on the left-hand side is lower semicontinuous with respect to weak convergence in $L^p(0,T;W^{1,p}(\Omega,\R^N))$.
Hence, by \eqref{weak_convergence_to_limit_map_general_domains}$_{2}$ it follows that
\begin{align}\label{general_domains:limit_1}
    \iint_{\Omega_{t_o}} f(x,u, Du) \,\dx\dt  \le \liminf_{\mathfrak{K} \ni \ell \to \infty} \iint_{\Omega_{t_o}} f(x,u_{\ell}, Du_{\ell}) \,\dx\dt. 
\end{align}
Next, \eqref{almost everywhere weakly star convergence with limit} gives us that
\begin{align}
    \lim_{\mathfrak{K} \ni \ell \to \infty} &\bint_{t_o}^{t_o+\delta} \iint_{\Omega_{\tau}} \partial_t v \cdot \big( \power{v}{q}-\power{u_{\ell}}{q} \big) \,\dx\dt \, \d\tau 
    \nonumber \\
    &=
    \bint_{t_o}^{t_o+\delta} \iint_{\Omega_{\tau}} \partial_t v \cdot \big( \power{v}{q}-\power{u}{q} \big) \,\dx\dt\, \d\tau.
    \label{general_domains:limit_2}
\end{align}
Now, combining \eqref{weak_convergence_to_limit_map_general_domains}$_{1}$ with \eqref{almost everywhere weakly star convergence with limit}, and recalling the definition of $\b[\cdot, \cdot]$ yields
\begin{align}\label{general_domains:limit_3}
    \bint_{t_o}^{t_o+\delta} \int_{\Omega } \b[u(\tau),v(\tau)] \,\dx \, \d\tau
    \le
    \liminf_{\mathfrak{K} \ni \ell \to \infty} \bint_{t_o}^{t_o+\delta} \int_{\Omega } \b[u_{\ell}(\tau),v(\tau)] \,\dx \, \d\tau.
\end{align}
Inserting \eqref{general_domains:limit_1}, \eqref{general_domains:limit_2}, and \eqref{general_domains:limit_3} into \eqref{general_domains:var_ineq_before_ell_limit}, we obtain that
\begin{align*}
    &\bint_{t_o}^{t_o+\delta} \int_{\Omega } \b[u(\tau),v(\tau)] \,\dx \, \d\tau
    +\iint_{\Omega_{t_o}} f(x,u, Du) \,\dx\dt
    \nonumber \\ &\le
    \iint_{\Omega_{t_o+\delta}} f(x,v, Dv) \,\dx\dt
    +\bint_{t_o}^{t_o+\delta} \iint_{\Omega_{\tau}} \partial_t v \cdot \big( \power{v}{q}-\power{u}{q} \big) \,\dx\dt \, \d\tau
    + \int_{\Omega } \b[u_o,v(0)] \,\dx
\end{align*}
holds for any $t_o \in [0,T)$. Finally, we pass to the limit $\delta \downarrow 0$ by means of Lebesgue's differentiation theorem. In this context, note that we have in particular $\tau \mapsto \int_\Omega \b[u(\tau),v(\tau)] \,\dx \in L^\infty([0,T])$.
Hence, we conclude that
\begin{align*}
     \int_{\Omega } \b&[u(t_o),v(t_o)] \,\dx 
     +\iint_{\Omega_{t_o}} f(x,u, Du) \,\dx\dt \nonumber \\ 
     &\le  \iint_{\Omega_{t_o}} f(x,v, Dv) \,\dx\dt 
     + \iint_{\Omega_{t_o}} \partial_t v \cdot (\power{v}{q}-\power{u}{q}) \,\dx\dt 
     + \int_{\Omega } \b[u_o,v(0)] \,\dx 
\end{align*}
holds for a.e.~$t_o \in [0,T]$ and any comparison map $v \in V^p_q(E)$ with time derivative $\partial_t v \in L^{q+1}(\Omega_{T}, \mathds{R}^{N})$. Therefore, the limit map $u$ solves the variational inequality in the sense of Definition~\ref{Definition:variational_solution}. This proves Theorem~\ref{thm:existence_in_general_domains}.

\section{Proof of Theorem \ref{thm:time_derivative_in_the_dual_space_general_domains}}
\label{sec:time-derivative-proof}
The proof of Theorem \ref{thm:time_derivative_in_the_dual_space_general_domains} in this section generalizes the arguments from \cite[Section 5.2]{BDSS18}, where the case $q=1$ was treated. The starting point are corresponding estimates for the time derivative from \cite[Theorem~2.4]{SSSS} in the case of nondecreasing domains.

First, for $\ell \in \mathds{N}$ consider the approximate solution $u_{\ell}$ defined by~\eqref{definition_of_approximate_solution(general_domain)}.
    From~\cite[Theorem~2.4]{SSSS} it follows that $\partial_t \power{u_{\ell,i}}{q} \in (V^{p,0}(Q_{\ell,i}))^{\prime}$ for any $i \in \{1,\ldots,\ell\}$. Furthermore, from the proof of~\cite[Theorem~2.4]{SSSS}, we infer the estimate
    \begin{align*}
        \bigg\vert \iint_{Q_{\ell,i}}  \power{u_{\ell,i}}{q} \cdot \partial_t \phi \,\dx\dt \bigg\vert
        \le
        c \iint_{Q_{\ell,i}} \Big[ \vert Du_{\ell,i} \vert^{p-1} + \vert u_{\ell,i} \vert^{p-1} + \vert G \vert^{\frac{1}{p'}}  \Big] [\vert D\phi \vert + \vert \phi \vert ] \,\dx\dt
    \end{align*}
    with a positive constant $c=c(p,L)$, for all $\phi \in V^{p,0}(Q_{\ell,i})$.
    Now, we fix a function $\phi \in C_{0}^{\infty}(E,\R^N)$.
    For $\epsilon \in \big(0, \tfrac{h_{\ell}}{4}\big]$ and $i \in \{1,\ldots,\ell\}$, we consider cut-off functions $\psi_{\ell,i}^{(\epsilon)} \in C^{0,1}([0,T])$ defined by
    \begin{equation*}
    \psi_{\ell,i}^{(\epsilon)}(t) =
    \begin{cases} 
    0, & \text{for } t \in [0, t_{\ell,i-1}], \\
    \frac{t - t_{\ell,i-1}}{\epsilon}, & \text{for } t \in (t_{\ell,i-1}, t_{\ell,i-1} + \epsilon), \\
    1, & \text{for } t \in [t_{\ell,i-1} + \epsilon, t_{\ell,i} - \epsilon], \\
    \frac{t_{\ell,i} - t}{\epsilon}, & \text{for } t \in (t_{\ell,i} - \epsilon, t_{\ell,i}), \\
    0, & \text{for } t \in [t_{\ell,i}, T].
    \end{cases}
    \end{equation*}
    Then we define $\phi_{\ell,i}^{(\epsilon)} :=
    \psi_{\ell,i}^{(\epsilon)}\phi$. Note that $\spt \Big(\phi_{\ell,i}^{(\epsilon)} \Big) \subset\overline{Q_{\ell,i}}$, since we have that $E\cap (\Omega \times I_{\ell, i}) \subset Q_{\ell,i}$.
    Therefore, $\phi_{\ell,i}^{(\epsilon)}$ is an admissible test function in the last inequality.
    Since $0 \leq \psi_{\ell,i}^{(\epsilon)} \leq 1$ is independent of the spatial variables, we find that
    \begin{align*}
        \bigg \vert \iint_{Q_{\ell,i}}&  \power{u_{\ell,i}}{q} \cdot \partial_t \phi_{\ell,i}^{(\epsilon)} \,\dx\dt \bigg \vert \\
        &\le c \iint_{Q_{\ell,i}} \Big[ \vert Du_{\ell,i} \vert^{p-1} + \vert u_{\ell,i} \vert^{p-1} + \vert G \vert^{\frac{1}{p'}}  \Big] \Big[ \Big\vert D\phi_{\ell,i}^{(\epsilon)} \Big\vert + \Big\vert \phi_{\ell,i}^{(\epsilon)} \Big\vert \Big] \,\dx\dt \\
        &\le c \iint_{Q_{\ell,i}} \Big[ \vert Du_{\ell,i} \vert^{p-1} + \vert u_{\ell,i} \vert^{p-1} + \vert G \vert^{\frac{1}{p'}}  \Big] [\vert D\phi \vert + \vert \phi \vert ] \,\dx\dt.
    \end{align*}
    Summing up these estimates over $i \in \{ 1,\ldots,\ell \}$, we obtain that
    \begin{align}\label{time_derivative_estimate_with_remainder}
        \bigg\vert \iint_{\Omega_{T}}&  \power{u_{\ell}}{q} \cdot \partial_t \phi \,\dx\dt \bigg\vert
        = 
        \Bigg\vert \sum_{i=1}^{\ell} \iint_{\Omega \times I_{\ell,i}}  \power{u_{\ell,i}}{q} \cdot \partial_t \phi \,\dx\dt \Bigg\vert \nonumber \\
        &\le
        \sum_{i=1}^{\ell} \bigg\vert  \iint_{\Omega \times I_{\ell,i}}  \power{u_{\ell,i}}{q} \cdot \partial_t \phi_{\ell,i}^{(\epsilon)} \,\dx\dt \bigg\vert + R_{\ell}^{(\epsilon)} \nonumber \\
        &\le
        c \sum_{i=1}^{\ell} \iint_{Q_{\ell,i}} \Big[ \vert Du_{\ell,i} \vert^{p-1} + \vert u_{\ell,i} \vert^{p-1} + \vert G \vert^{\frac{1}{p'}}  \Big] [\vert D\phi \vert + \vert \phi \vert ] \,\dx\dt+ R_{\ell}^{(\epsilon)} \nonumber \\
        &=
        c\iint_{\Omega_{T}} \Big[ \vert Du_{\ell} \vert^{p-1} + \vert u_{\ell} \vert^{p-1} + \vert G \vert^{\frac{1}{p'} }  \Big] [\vert D\phi \vert + \vert \phi \vert ] \,\dx\dt+ R_{\ell}^{(\epsilon)}, 
    \end{align}
    where the remainder term is given by
    \begin{align*}
        R_{\ell}^{(\epsilon)} := \Bigg\vert \sum_{i=1}^{\ell} \iint_{Q_{\ell,i}}  \power{u_{\ell,i}}{q} \cdot \Big( \partial_t \phi - \partial_t \phi_{\ell,i}^{(\epsilon)} \Big) \,\dx\dt \Bigg\vert.
    \end{align*}
    Since $\power{u_{\ell,i}}{q} \in C^{0} \Big(\overline{I_{\ell,i}}, L^{\frac{q+1}{q}}(E_{\ell,i}, \mathds{R}^{N}) \Big)$ and by definition of $\phi_{\ell,i}^{(\epsilon)}$, we get that
    \begin{align*}
        R_{\ell}^{(\epsilon)}
        &\leq
        \Bigg\vert \sum_{i=1}^{\ell} \iint_{Q_{\ell,i}} \Big( 1-\psi_{\ell,i}^{(\epsilon)} \Big) \power{u_{\ell,i}}{q} \cdot\partial_t \phi \,\dx\dt \Bigg\vert
        + \Bigg\vert \sum_{i=1}^{\ell} \iint_{Q_{\ell,i}} \big(\psi_{\ell,i}^{(\epsilon)}\big)^{\prime} \power{u_{\ell,i}}{q} \cdot \phi \,\dx\dt \Bigg\vert \nonumber \\
        &=
        \Bigg\vert \sum_{i=1}^{\ell} \iint_{Q_{\ell,i}} \Big( 1-\psi_{\ell,i}^{(\epsilon)} \Big) \power{u_{\ell,i}}{q} \cdot\partial_t \phi \,\dx\dt \Bigg\vert
        \nonumber \\ &\phantom{=}
        + \Bigg| \sum_{i=1}^{\ell} 
        \bigg[ \tfrac{1}{\epsilon} \iint_{E_{\ell,i} \times (t_{\ell,i-1}, t_{\ell,i-1} + \epsilon)} \power{u_{\ell,i}}{q} \cdot \phi \,\dx\dt
        \nonumber \\ &\phantom{=+ \bigg| \sum_{i=1}^{\ell}\bigg[}
        - \tfrac{1}{\epsilon} \iint_{E_{\ell,i} \times (t_{\ell,i} - \epsilon, t_{\ell,i})} \power{u_{\ell,i}}{q} \cdot \phi \,\dx\dt \bigg] \Bigg| \nonumber \\
        &\to 
        \Bigg\vert \sum_{i=1}^{\ell} \bigg[ \int_{E_{\ell,i}} \power{u_{\ell,i}}{q}(t_{\ell,i-1}) \cdot \phi(t_{\ell,i-1}) \,\dx
        - \int_{E_{\ell,i}} \power{u_{\ell,i}}{q}(t_{\ell,i}) \cdot \phi(t_{\ell,i}) \,\dx \bigg] \Bigg\vert
    \end{align*}
    in the limit $\epsilon \downarrow 0$.
    To analyze the expression on the right-hand side further, note that $\spt(\phi(t_{\ell,i-1})) \subset E_{\ell,i}$, and that the initial values $u_{\ell,i}^{(0)}$ are chosen according to~\eqref{initial_condition_u_l,i_generarl_domain}, which implies that
    \begin{align*}
        \power{u_{\ell,i}}{q}(t_{\ell,i-1}) \cdot \phi(t_{\ell, i-1})
        &=
        \power{\big(u_{\ell,i}^{(0)}\big)}{q} \cdot \phi(t_{\ell, i-1})
        =
        \power{u_{\ell,i-1}}{q}(t_{\ell,i-1})\chi_{E_{\ell,i}} \cdot \phi(t_{\ell, i-1}) \\
        &=
        \power{u_{\ell,i-1}}{q}(t_{\ell,i-1}) \cdot \phi(t_{\ell, i-1})
    \end{align*}
    for $i\in\{2,\ldots,\ell\}$. 
    Hence, we are dealing with a telescoping sum.
    Taking into account that $u_{\ell,1}(0) \cdot \phi(0)=0 = u_{\ell,\ell}(T) \cdot \phi(T)$, since $\phi \in C_{0}^{\infty}(E, \mathds{R}^{N})$, we conclude that
    \begin{align*}
        \lim_{\epsilon \downarrow 0} R_{\ell}^{(\epsilon)} = 0.
    \end{align*}
    Therefore, passing to the limit $\epsilon \downarrow 0$ in \eqref{time_derivative_estimate_with_remainder}, applying H\"older's inequality and Poincaré's inequality, and using~\eqref{ineq:L^p_estimate_for_derivative} yields
    \begin{align*}
        \bigg\vert \iint_{\Omega_{T}} &\power{u_{\ell}}{q} \cdot \partial_t \phi \,\dx\dt \bigg\vert 
        \le
        c\bigg[ \iint_{\Omega_{T}} \vert Du_{\ell} \vert^{p} + \vert u_{\ell} \vert^{p} +\vert G \vert \,\dx\dt \bigg]^{\frac{1}{p'} }
        \Vert \phi \Vert_{V^{p}(E)} \\
        &\le
        c \bigg[ \iint_{\Omega_{T}} \vert Du_{\ell} \vert^{p} +|Du_\ast|^p+|u_\ast|^p + \vert G \vert \,\dx\dt \bigg]^{\frac{1}{p'} } \Vert \phi \Vert_{V^{p}(E)} \\
        &\leq
        c \bigg[ T \int_{\Omega} \vert Du_{\ast} \vert^{p} + \vert u_{\ast} \vert^{p} + |G| \,\dx
        + \int_\Omega |u_o|^{q+1} + |u_\ast|^{q+1} \,\dx \bigg]^{\frac{1}{p'} }
        \Vert \phi \Vert_{V^{p}(E)},
    \end{align*}
    where $c = c(p,q,\nu,L,\diam(\Omega))$.
    Since $\power{u_\ell}{q} \wto \power{u}{q}$ weakly in $L^\frac{q+1}{q}(\Omega_T,\R^N)$ by \eqref{almost everywhere weakly star convergence with limit}, and since $C_{0}^{\infty}(E, \mathds{R}^{N})$ is dense in the function space $V^{p,0}(E)$ (see Remark~\ref{rem:C_0_inf_is_dense_in_V_p,0}) we conclude that $\partial_t \power{u}{q} \in (V^{p,0}(E))^{\prime}$ with the estimate
    \begin{align*}
        &\Vert   \partial_t \power{u}{q} \Vert_{(V^{p,0}(E))^{\prime}}
        \\ &\le
        c \Big[ T \Big( \Vert u_{\ast} \Vert_{W^{1,p}(\Omega,\R^N)}^p + \Vert G \Vert_{L^1(\Omega)} \Big)
        + \Vert u_o \Vert_{L^{q+1}(\Omega, \mathds{R}^{N})}^{q+1} + \Vert u_{\ast} \Vert_{L^{q+1}(\Omega, \mathds{R}^{N})}^{q+1} \Big]^{\frac{1}{p'} }.
    \end{align*}
This concludes the proof of Theorem~\ref{thm:time_derivative_in_the_dual_space_general_domains}.

\section{Proof of Theorem \ref{continuity_in_time_for_general_domains}}
\label{sec:continuity-proof}
\subsection{An integration by parts formula}
\label{sec:integration-by-parts}
Before we are ready to prove continuity with respect to time, we need to derive the following vital tool.
In the following, the norm of $(\mathcal{V}^{p,0}(E))'$ is defined analogous to \eqref{eq:operator-norm}.
\begin{theorem}\label{thm:ineq_integration_by_parts_formula}
Assume that the domain $E$ satisfies \eqref{one sided growth condition}, and \eqref{p-fat}.
For exponents $p>1$ and $q>0$ satisfying $p\ge\frac{(n+1)(q+1)}{n+q+1}$, we consider functions $u,v \in L^p(0,T;W^{1,p}(\Omega,\R^N)) \cap L^{\infty}(0,T;L^{q+1}(\Omega, \mathds{R}^{N}))$ with $v-u \in \mathcal{V}^{p,0}_q(E)$ and time derivatives $\partial_t \power{u}{q} \in (\mathcal{V}^{p,0}(E))'$ and $\partial_t v \in L^{q+1}(\Omega_T,\R^N)$. Then, the formula
\begin{align}\label{ineq:general_integration_by_parts_formula}
        \iint_{E} \partial_t v \cdot \zeta ( \power{v}{q} - \power{u}{q}) \,\dx\dt \le \langle \partial_t \power{u}{q}, \zeta (v-u) \rangle - \iint_{E} \zeta^{\prime}\, \b[u,v] \,\dx\dt
\end{align}
holds for any non-negative function $\zeta \in C_0^{0,1}((0,T))$.
\end{theorem}

\begin{remark}\label{rem:int_by_parts_measure_dens}
By Remark~\ref{rem:V=V} we know that $\mathcal{V}^{p,0}_q(E)=V^{p,0}_q(E)$ if we replace the $p$-fatness condition \eqref{p-fat} by the measure density condition \eqref{ineq:lower_measure_bound}.
Therefore, in this case we obtain the analogous result to Theorem \ref{thm:ineq_integration_by_parts_formula} for the space $V^{p,0}_q(E)$.
\end{remark}

Condition \eqref{one sided growth condition} prevents the domain from shrinking too fast in time.
In contrast, it can increase arbitrarily fast, and even outward jumps are allowed.
Conversely, we can obtain the opposite inequality in the integration by parts formula \eqref{ineq:general_integration_by_parts_formula} under a corresponding condition on the growth of the domain, which is stated in the following corollary.
In particular, if $E$ neither grows nor shrinks too fast (for instance, if $E$ is cylindrical), equality holds in \eqref{ineq:general_integration_by_parts_formula}.
\begin{corollary}
Assume that the assumptions of
Theorem~\ref{thm:ineq_integration_by_parts_formula} hold, except
\eqref{one sided growth condition}, which is replaced by the condition in which the roles of $s$ and $t$ are reversed,
i.e.,
    \begin{align*}
        \power{e}{c}(E^{t}, E^{s}) \le \vert \rho(t) - \rho(s) \vert
        \quad \text{for $0 \le s \le t <T$.}
    \end{align*}
    Then, we have that
    \begin{align}\label{ineq:integration_by_parts_formula_reversed}
    	\iint_{E} \partial_t v \cdot \zeta ( \power{v}{q} - \power{u}{q}) \,\dx\dt
    	\geq
    	\langle \partial_t \power{u}{q}, \zeta (v-u) \rangle
    	- \iint_{E} \zeta^{\prime}\, \b[u,v] \,\dx\dt
    \end{align}
    holds for any non-negative function $\zeta \in C_0^{0,1}((0,T))$.
\end{corollary}
\begin{proof}
    Let
    \begin{align*}
        \widetilde{E}:= \bigcup_{t \in (0,T)} E^{T-t} \times \{ t \}
    \end{align*}
    and $\tilde{\zeta}(t):= \zeta(T-t)$.
    Consider the maps $\tilde{u}(x,t):=u(x,T-t), \tilde{v}(x,t):=v(x,T-t) \in L^p(0,T;W^{1,p}(\Omega,\R^N)) \cap L^{\infty}(0,T;L^{q+1}(\Omega, \mathds{R}^{N}))$ with $\tilde{v}-\tilde{u} \in \mathcal{V}^{p,0}_q \big(\widetilde{E} \big)$, $\partial_t \power{\tilde{u}}{q} \in \big(\mathcal{V}^{p,0} \big( \widetilde{E} \big) \big)'$ and $\partial_t \tilde{v} \in L^{q+1} \big( \widetilde{E},\R^N \big)$. Applying Theorem~\ref{thm:ineq_integration_by_parts_formula} to $\tilde{u}, \tilde{v}$ and $\tilde{\zeta}$ yields
\begin{align*}
	\iint_{E} \partial_t v \cdot \zeta ( \power{v}{q} - \power{u}{q}) \,\dx\dt
	&=
	-\iint_{\widetilde{E}} \partial_t \tilde{v} \cdot \tilde{\zeta} ( \power{\tilde{v}}{q} - \power{\tilde{u}}{q}) \,\dx\dt
	\\ &\geq
	-\langle \partial_t \power{\tilde{u}}{q}, \tilde{\zeta} (\tilde{v}-\tilde{u}) \rangle
    + \iint_{E} \tilde{\zeta}^{\prime}\, \b[\tilde{u},\tilde{v}] \,\dx\dt
    \\ &=
    \langle \partial_t \power{u}{q}, \zeta (v-u) \rangle
    -\iint_{E} \zeta^{\prime}\, \b[u,v] \,\dx\dt,
\end{align*}
which finishes the proof.
\end{proof}

Now, we define the forward in-time Steklov average by
\begin{align}\label{Steklov_average(forward_in_time)}
    [ v ]^{S}_h(x,t)
    :=
    \tfrac{1}{h} \int_t^{t+h} v(x,s) \,\ds.
\end{align}
By construction, we have that
\begin{align*}
    \partial_t [ v ]^{S}_h(t) = \tfrac{1}{h} (v(t+h) - v(t)).
\end{align*}
Before proving Theorem~\ref{thm:ineq_integration_by_parts_formula}, we state an intermediate step.

\begin{theorem}
\label{thm:integration_by_parts_compact_spt}
Suppose that the assumptions of Theorem~\ref{thm:ineq_integration_by_parts_formula} are satisfied, except \eqref{one sided growth condition}, which is replaced by the weaker condition \eqref{ineq:one_sided_growth_with_modulus}.
Further, let $\zeta \in C^{0,1}_0((0,T))$ be a non-negative cut-off function in time and let $\eta$ be a non-negative cut-off function such that $\eta^{q+1} \in C^{0,1}(E)$ and there exists $\sigma>0$ such that $\spt(\eta(t)) \subset \overline{E^{t,\sigma}}$ for a.e.~$t \in [0,T]$, with the inner parallel set $E^{t,\sigma}$ defined in~\eqref{def:parallel-set}.
Then, we have that
\begin{align*}
	\big\langle \partial_t \power{u}{q}, \eta^{q+1} \zeta (v-u) \big\rangle
    &=
    \iint_E \partial_t v \cdot \eta^{q+1} \zeta \big( \power{v}{q} - \power{u}{q} \big) \,\dx\dt
    + \iint_E \partial_t \big( \eta^{q+1} \zeta \big) \b[ u,v ] \,\dx\dt.
\end{align*}
\end{theorem}

\begin{proof}
Let $h >0$ and fix $\zeta \in C^{0,1}_0((0,T))$. Extend the function $u$ to $\Omega \times (T, \infty)$ by zero and consider the Steklov average $\big[\power{u}{q}\big]^{S}_h$.
First, we show that
\begin{align}\label{limit_claim_intermediate_result_proof}
	\lim_{h \downarrow 0} \iint_{E} \partial_t \big[\power{u}{q}\big]^{S}_h \cdot \eta^{q+1}\zeta (v-u) \,\dx\dt
	=
	\big\langle \partial_t \power{u}{q}, \eta^{q+1}\zeta (v-u) \big\rangle.
\end{align}
Indeed, since $\eta^{q+1} \zeta (v-u) \in V^{p,q}_\mathrm{cpt}(E)$ by construction, by Lemma~\ref{intermediate_density_argument} we find a sequence $(\phi_\ell)_{\ell \in \mathds{N}} \subset C_{0}^{\infty}(E, \mathds{R}^{N})$ such that $\phi_\ell \to \eta^{q+1}\zeta (v-u)$ in the norm topology of $V_q^{p,0}(E)$ as $\ell \to \infty$ and $\spt(\phi_\ell (t)) \subset \overline{E^{t,\sigma}}$ for any $t \in [0,T)$ and any $\ell \in \mathds{N}$. 
Since $\spt(\zeta) \Subset (0,T)$, we may further assume that $\spt(\phi_\ell) \subset \Omega \times (\sigma, T - \sigma)$ for any $\ell \in \mathds{N}$.
Now, we consider the difference
\begin{align}\label{eq:difference_integral_and_dual_pair_limit}
    \iint_{E} \partial_t \big[\power{u}{q}\big]^{S}_h &\cdot \eta^{q+1}\zeta (v-u) \,\dx\dt
    - \big\langle \partial_t \power{u}{q}, \eta^{q+1} \zeta (v-u) \big\rangle
    \nonumber \\ &=
    \bigg[ \iint_{E} \partial_t \big[\power{u}{q}\big]^{S}_h \cdot \phi_\ell \,\dx\dt
    - \big\langle \partial_t \power{u}{q}, \phi_\ell \big\rangle \bigg]
   \nonumber \\
   &\phantom{=}
    + \iint_{E} \partial_t \big[\power{u}{q}\big]^{S}_h \cdot ( \eta^{q+1}\zeta (v-u) - \phi_\ell ) \,\dx\dt \nonumber \\
   &\phantom{=}
   + \big\langle \partial_t \power{u}{q}, \phi_\ell - \eta^{q+1}\zeta (v-u) \big\rangle \nonumber \\
   &=: \mathrm{I}_{h, \ell} + \mathrm{II}_{h, \ell} + \mathrm{III}_{\ell},
\end{align}
where the definition of the terms in the last line is apparent from the context.
Now, we analyze the limit $h \downarrow 0$ of each of these terms separately.
First, since $\spt(\phi_\ell) \subset E$, for any fixed $\ell \in \mathds{N}$ we obtain that
\begin{align}\label{eq:weak_convergence_of_steklov_average_of_u_power_q}
    \iint_{E} \partial_t \big[\power{u}{q}\big]^{S}_h \cdot \phi_\ell \,\dx\dt
    &= \iint_{\Omega_{T}} \partial_t \big[\power{u}{q}\big]^{S}_h \cdot \phi_\ell \,\dx\dt
    =
    - \iint_{\Omega_{T}}  \big[\power{u}{q}\big]^{S}_h \cdot \partial_t \phi_\ell \,\dx\dt \nonumber \\
    &\to
    -\iint_{\Omega_{T}}  \power{u}{q} \cdot \partial_t \phi_\ell \,\dx\dt
    =
    - \iint_{E}  \power{u}{q} \cdot \partial_t \phi_\ell \,\dx\dt
    =
    \big\langle \partial_t \power{u}{q} , \phi_\ell \big\rangle
\end{align}
in the limit $h \downarrow 0$.
Hence, we have that
\begin{align}\label{first_limit_intermediate_result_part_int_formula_proof}
    \lim_{h \downarrow 0} \mathrm{I}_{h, \ell} =0
    \quad \text{for any $\ell \in \mathds{N}$.}
\end{align}
Next, we rewrite   
\begin{align*}
    \mathrm{II}_{h, \ell}
    =
    \lim_{m \to \infty} \mathrm{II}_{h, \ell, m}
    :=
    \lim_{m \to \infty}
    \iint_{E} \partial_t \big[\power{u}{q}\big]^{S}_h \cdot ( \phi_{m} - \phi_\ell ) \,\dx\dt.
\end{align*} 
Since we need to pass to the limit $m \to \infty$ in $\mathrm{II}_{h, \ell, m}$ first, we are not able to proceed as before.
Instead, we fix $\ell, m \in \mathds{N}$ and consider $h \in \big(0,  \min\big\{\tfrac\sigma2, \omega^{-1} \big( \frac{\sigma}{2} \big) \big\}\big)$, where $\omega$ is the modulus of continuity given in~\eqref{ineq:one_sided_growth_with_modulus}.
Observing that $\spt(\phi_{k}) \subset \Omega \times (2h,T-2h)$ for any $k \in \mathds{N}$, we compute that
\begin{align*}
	\mathrm{II}_{h, \ell, m}
	&=
	\iint_{\Omega \times (2h,T-2h)} \partial_t \big[\power{u}{q}\big]^{S}_h \cdot ( \phi_{m} - \phi_\ell ) \,\dx\dt \\
    &=
    \tfrac{1}{h} \int_{2h}^{T-2h} \int_{\Omega} [\power{u}{q}(t+h) - \power{u}{q}(t)] \cdot (\phi_{m}(t) - \phi_{\ell}(t)) \,\dx\dt \\
    &=
    \tfrac{1}{h} \int_{2h}^{T-h} \int_{\Omega} \power{u}{q}(t) \cdot [(\phi_{m}-\phi_{\ell})(t-h)-(\phi_{m} - \phi_{\ell})(t)] \,\dx\dt.
\end{align*}
Applying Fubini's theorem yields
\begin{align}\label{II_(h, ell, m)_calculated}
    \mathrm{II}_{h, \ell, m}
    &=
    - \tfrac{1}{h} \int_{2h}^{T-h} \int_{\Omega} \power{u}{q}(t) \cdot \int_{0}^{h} \partial_t (\phi_{m} - \phi_{\ell} )(t-s)\ds \,\dx\dt \nonumber \\
    &=
    - \tfrac{1}{h} \int_{0}^{h} \bigg[ \iint_{\Omega \times (2h,T-h)}
      \power{u}{q}(t) \cdot \partial_t (\phi_{m} - \phi_{\ell})(t-s) \,\dx\dt \bigg] \,\ds \nonumber \\
    &=
    \tfrac{1}{h} \int_{0}^{h} \big\langle \partial_t \power{u}{q},(\phi_{m} - \phi_{\ell})(\cdot-s) \big\rangle \,\ds.
    \end{align}
However, the last line only makes sense if we can guarantee that
$(\phi_{m} - \phi_{\ell})(\cdot-s) \in \mathcal{V}^{p,0}(E)$ for any $s \in [0,h]$.
In fact, we can even show that the function is contained in the
smaller space $V_\mathrm{cpt}^{p,q}(E)$.
For the proof of this claim, let $t \in [0,T)$, $s \in [0,h]$ and $x \in \Omega \setminus E^{t,\frac\sigma2}$.
On the one hand, if $x \in \Omega \setminus E^{t-s}$, we
have that $(\phi_{m} - \phi_{\ell})(x,t-s)=0$ directly.
On the other hand, if $x \in E^{t-s}$, we choose 
a point $y\in\Omega\setminus E^t$ with $|x-y|\le\frac\sigma2$ and use 
condition~\eqref{ineq:one_sided_growth_with_modulus} to estimate
\begin{align*}
    \dist \big( x, \partial E^{t-s} \big)
    &=
      \dist \big( x, \Omega \setminus E^{t-s} \big)\\
  &\leq |x-y|+\dist \big( y, \Omega \setminus E^{t-s} \big)\\
    &\leq
    \tfrac\sigma2+\power{e}{c} \big(E^{t-s}, E^{t} \big)\\
    &\leq
     \tfrac\sigma2+\omega(s)\\
    &\leq
     \tfrac\sigma2+\omega \big( \omega^{-1} \big( \tfrac\sigma2 \big) \big)
    = \sigma.
\end{align*}
This means that $x \in \Omega \setminus E^{t-s,\sigma}$.
Therefore, the choice of $(\phi_\ell)_{\ell\in\N}$ implies that $(\phi_{m} -
\phi_{\ell})(x,t-s)=0$ also holds in the second case.
Combining the two cases, we deduce that $\spt((\phi_{m} -
\phi_{\ell})(\cdot,t-s))\subset \overline{E^{t,\frac{\sigma}{2}}}$ for
every $t\in(0,T]$,
which implies 
$(\phi_{m} - \phi_{\ell})(\cdot-s) \in
V_\mathrm{cpt}^{p,q}(E)\subset\mathcal{V}^{p,0}(E)$ and thus the last line
in~\eqref{II_(h, ell, m)_calculated} is valid.
Recalling that $\mathcal{V}^{p,0}(E)$ is equipped with the norm $\| \cdot \|_{V^p(E)}$, this allows us to estimate
\begin{align*}
    \vert \mathrm{II}_{h, \ell, m} \vert
    &\leq
    \tfrac{1}{h} \int_{0}^{h} \Vert \partial_t \power{u}{q} \Vert_{(\mathcal{V}^{p,0}(E))'}
    \Vert (\phi_{m} - \phi_{\ell})(\cdot-s) \Vert_{V^{p}(E)} \,\ds \\
    &=
    \Vert \partial_t \power{u}{q} \Vert_{(\mathcal{V}^{p,0}(E))'}
    \Vert \phi_{m} - \phi_{\ell} \Vert_{V^{p}(E)}.
\end{align*}
Passing to the limit $m \to \infty$ yields
\begin{align}\label{secondt_limit_intermediate_result_part_int_formula_proof}
    \vert \mathrm{II}_{h, \ell} \vert
    \leq
    \Vert \partial_t \power{u}{q} \Vert_{(\mathcal{V}^{p,0}(E))'}
    \big\Vert \phi_{\ell} -\eta^{q+1}\zeta (v-u) \big\Vert_{V^{p}(E)}
\end{align}
for any $h \in \big(0, \min\big\{\tfrac\sigma2, \omega^{-1}\big( \frac{\sigma}{2} \big) \big\}\big)$ and $\ell \in \mathds{N}$.
Analogously, we estimate $\mathrm{III}_{\ell}$ by
\begin{align}\label{third_limit_intermediate_result_part_int_formula_proof}
    \vert \mathrm{III}_{\ell} \vert
    \leq
    \Vert \partial_t \power{u}{q} \Vert_{(\mathcal{V}^{p,0}(E))'}
    \big\Vert \phi_{\ell} -\eta^{q+1}\zeta (v-u) \big\Vert_{V^{p}(E)}
\end{align}
for any $\ell \in \mathds{N}$.
Inserting \eqref{first_limit_intermediate_result_part_int_formula_proof}, \eqref{secondt_limit_intermediate_result_part_int_formula_proof} and \eqref{third_limit_intermediate_result_part_int_formula_proof} into \eqref{eq:difference_integral_and_dual_pair_limit}, we deduce that
\begin{align*}
    \limsup_{h\downarrow 0}
    \bigg\vert \iint_{E} &\partial_t [\power{u}{q}]^{S}_h \cdot \eta^{q+1} \zeta(v-u) \,\dx\dt
    - \big\langle \partial_t \power{u}{q}, \eta^{q+1}\zeta (v-u) \big\rangle \bigg\vert
    \\ &\leq
    2 \Vert \partial_t \power{u}{q} \Vert_{(\mathcal{V}^{p,0}(E))'}
    \big\Vert \phi_{\ell} -\eta^{q+1}\zeta (v-u) \big\Vert_{V^{p}(E)}
\end{align*}
for any $\ell \in \N$.
Since the second term on the right-hand side vanishes as $\ell \to \infty$, we conclude that \eqref{limit_claim_intermediate_result_proof} holds.

Next, we analyze the integral on the left-hand side of \eqref{limit_claim_intermediate_result_proof} further.
First, since $v-u=0$ a.e.~in $\Omega_T\setminus E$, we may write
\begin{align*}
	\iint_{E} \partial_t \big[\power{u}{q}\big]^{S}_h \cdot \eta^{q+1}\zeta (v-u) \,\dx\dt
	=
	\iint_{\Omega_T} \partial_t \big[\power{u}{q}\big]^{S}_h \cdot \eta^{q+1}\zeta (v-u) \,\dx\dt.
\end{align*}
On the one hand, by the product and chain rules for Sobolev functions and integration by parts we deduce that
\begin{align*}
	\iint_{\Omega_T} &\partial_t [\power{u}{q}]^{S}_h \cdot \eta^{q+1}\zeta v \,\dx\dt
	=	
	-\iint_{\Omega_T} [\power{u}{q}]^{S}_h \cdot \partial_t \big( \eta^{q+1}\zeta v \big) \,\dx\dt
	\\ &=
	\iint_{\Omega_T} \partial_t v \cdot \eta^{q+1}\zeta \big( \power{v}{q} - [\power{u}{q}]^{S}_h \big) \,\dx\dt
	- \iint_{\Omega_T} \eta^{q+1}\zeta (\power{v}{q} \cdot \partial_t v) \,\dx\dt
	\\ & \phantom{=}
	- \iint_{\Omega_T} \partial_t \big( \eta^{q+1} \zeta \big) [\power{u}{q}]^{S}_h \cdot v \,\dx\dt
	\\ &=
	\iint_{\Omega_T}  \partial_t v \cdot \eta^{q+1} \zeta \big( \power{v}{q} - [\power{u}{q}]^{S}_h \big) \,\dx\dt
	- \iint_{\Omega_T} \eta^{q+1} \zeta \partial_t \Big( \tfrac{1}{q+1} |v|^{q+1} \Big) \,\dx\dt
	\\ & \phantom{=}
	- \iint_{\Omega_T} \partial_t \big( \eta^{q+1} \zeta \big) [\power{u}{q}]^{S}_h \cdot v \,\dx\dt
	\\&=
	\iint_{\Omega_T}  \partial_t v \cdot \eta^{q+1} \zeta \big( \power{v}{q} - [\power{u}{q}]^{S}_h \big) \,\dx\dt
	\\ &\phantom{=}
	+ \iint_{\Omega_T} \partial_t \big( \eta^{q+1} \zeta \big) \Big( \tfrac{1}{q+1} |v|^{q+1} - [\power{u}{q}]^{S}_h \cdot v \Big) \,\dx\dt.
\end{align*}
Passing to the limit $h \downarrow 0$, we find that
\begin{align}
	\lim_{h \downarrow 0}
	\iint_{\Omega_T} &\partial_t [\power{u}{q}]^{S}_h \cdot \eta^{q+1} \zeta v \,\dx\dt
	\nonumber \\ &=
	\iint_{\Omega_T} \partial_t v \cdot \eta^{q+1} \zeta \big( \power{v}{q} - \power{u}{q} \big) \,\dx\dt
	\label{eq:aux1-int-by-parts-compact-spt} \\ &\phantom{=}
	+ \iint_{\Omega_T} \partial_t \big( \eta^{q+1} \zeta \big) \Big( \tfrac{1}{q+1} |v|^{q+1} - \power{u}{q} \cdot v \Big) \,\dx\dt.
	\nonumber
\end{align}
On the other hand, we obtain that
\begin{align*}
	&\iint_{\Omega_T}
	\partial_t [\power{u}{q}]^{S}_h \cdot \eta^{q+1}\zeta u \,\dx\dt
	\\ &\phantom{=}
	-\tfrac{q}{q+1}	\iint_{\Omega_T}
	\tfrac{1}{h} \big( \big( \eta^{q+1} \zeta \big)(t-h) - \big( \eta^{q+1} \zeta\big) (t) \big) |u(t)|^{q+1} \,\dx\dt
	\\&=
	\iint_{\Omega_T}
	\tfrac{1}{h} \big( \power{u}{q}(t+h) - \power{u}{q}(t) \big)
	\cdot \big( \eta^{q+1} \zeta \big)(t) u(t) \,\dx\dt
	\\&\phantom{=}
	-\tfrac{q}{q+1}	\iint_{\Omega_T}
	\tfrac{1}{h} \big( \big( \eta^{q+1} \zeta \big)(t-h) - \big( \eta^{q+1} \zeta\big) (t)\big) |u(t)|^{q+1} \,\dx\dt
	\\&=
	\tfrac{q}{q+1} \iint_{\Omega_T}
	\tfrac{1}{h} \big( \eta^{q+1} \zeta \big)(t) |u(t+h)|^{q+1}
	- \tfrac{1}{h} \big( \eta^{q+1} \zeta\big) (t-h) |u(t)|^{q+1} \,\dx\dt
	\\ &\phantom{=}
	- \iint_{\Omega_T} \tfrac{1}{h}
	\Big( \tfrac{1}{q+1} |u(t)|^{q+1} + \tfrac{q}{q+1} |u(t+h)|^{q+1} - \power{u}{q}(t+h) \cdot u(t) \Big) \big( \eta^{q+1} \zeta \big)(t) \,\dx\dt
	\\&=
	\tfrac{q}{q+1} \iint_{\Omega \times (h,T+h)}
	\tfrac{1}{h} \big( \eta^{q+1}\zeta \big)(t-h) |u(t)|^{q+1} \,\dx\dt
	\\&\phantom{=}
	-\tfrac{q}{q+1} \iint_{\Omega_T} 
	\tfrac{1}{h} \big( \eta^{q+1}\zeta \big)(t-h) |u(t)|^{q+1} \,\dx\dt
	\\ &\phantom{=}
	- \iint_{\Omega_T} \tfrac{1}{h}
	\b[u(t+h),u(t)] \big( \eta^{q+1}\zeta \big)(t) \,\dx\dt.
\end{align*}
Since $u(t)=0$ for $t\in(T,T+h)$ and $\zeta(t-h)=0$ for $t\in (0,h)$, the first two terms on the right-hand side of the preceding inequality cancel each other.
Recalling that $\b[u(t),u(t)]=0$, using the definition of the Steklov average and performing an integration by parts in the last term on the right-hand side yields
\begin{align}\label{calculation:appyl_integration_by_parts_toterm_on_right_hand_side}
	\iint_{\Omega_T}
	&\partial_t [\power{u}{q}]^{S}_h \cdot \eta^{q+1} \zeta u \,\dx\dt
	\nonumber \\&\phantom{=}
	-\tfrac{q}{q+1}	\iint_{\Omega_T}
	\tfrac{1}{h} \big( \big( \eta^{q+1} \zeta \big) (t-h) - \big( \eta^{q+1} \zeta \big)(t) \big) |u(t)|^{q+1} \,\dx\dt
	\nonumber \\&=
	- \iint_{\Omega_T}
	\partial_t \bigg( \bint_t^{t+h} \b[u(s),u(t)] \,\ds \bigg) \big( \eta^{q+1} \zeta \big)(t) \,\dx\dt
	\nonumber \\&=
	\iint_{\Omega_T} \partial_t  \big( \eta^{q+1} \zeta\big) (t)
	\bint_t^{t+h} \b[u(s),u(t)] \,\ds \,\dx\dt.
\end{align}
Furthermore, since $u \in L^{q+1}(\Omega_T,\R^N)$, by a standard property of Steklov averages we have that
\begin{align*}
	\bint_t^{t+h} \b[u(s),u(t)] \,\ds
	\to
	\b[u(t),u(t)]=0
	\quad \text{in $L^1(\Omega_T)$ as $h \downarrow 0$.}
\end{align*}
Since we have that $\partial_t \big( \eta^{q+1} \zeta \big) \in L^\infty(0,T)$, the term on the right-hand side of \eqref{calculation:appyl_integration_by_parts_toterm_on_right_hand_side}  vanishes as $h \downarrow 0$ and we obtain that
\begin{align}
	\lim_{h \downarrow 0}
	&\bigg(-\iint_{\Omega_T}
	\partial_t [\power{u}{q}]^{S}_h	\cdot \eta^{q+1} \zeta u \,\dx\dt\bigg)
	=
	\tfrac{q}{q+1} \iint_{\Omega_T} \partial_t \big( \eta^{q+1} \zeta \big) |u|^{q+1} \,\dx\dt.
	\label{eq:aux2-int-by-parts-compact-spt}
\end{align}
Hence, inserting \eqref{eq:aux1-int-by-parts-compact-spt} and \eqref{eq:aux2-int-by-parts-compact-spt} into \eqref{limit_claim_intermediate_result_proof}, recalling the definition of $\b[\cdot,\cdot]$, and using that $v=u$ a.e.~in $\Omega_T \setminus E$ shows that 
\begin{align*}
    \big\langle \partial_t \power{u}{q}, \eta^{q+1} \zeta (v-u) \big\rangle
    =
    \iint_{E} \partial_t v \cdot \eta^{q+1} \zeta ( \power{v}{q} - \power{u}{q}) \,\dx\dt
    + \iint_{E} \partial_t \big( \eta^{q+1} \zeta \big) \b[u,v] \,\dx\dt, 
\end{align*}
which proves the theorem.
\end{proof}

Next, for $\sigma, h >0$, we define a cut-off function based on the backward-in-time Steklov average of $\eta_\sigma$, where $\eta_\sigma$ is given by~\eqref{definition_eta_sigma}.
Indeed, we set
\begin{align}\label{backward_in_time_Steklov_average}
    \eta_{\sigma, h}(x,t)
    :=
    \bigg( \tfrac{1}{h} \int_{t-h}^{t} \eta_{\sigma}(x,s) \,\ds \bigg)^\frac{1}{q+1},
\end{align}
where we extended $\eta_\sigma$ to negative times by zero. 
Similarly as in \cite[Lemmas 5.9]{BDSS18}, we have the following properties of $\eta_{\sigma,h}$.
\begin{lemma}\label{lem:properties_for_eta_sigma_h}
Suppose that $E$ satisfies \eqref{ineq:one_sided_growth_with_modulus} with a modulus of continuity $\omega$, and let $\sigma > 0$ and $h \in \left(0,\omega^{-1} \big(\tfrac{\sigma}{2} \big)\right]$.
Then $\eta_{\sigma,h}$, defined according to \eqref{backward_in_time_Steklov_average}, satisfies
\begin{enumerate}
    \item[(i)] $\eta_{\sigma,h}^{q+1}\in C^{0,1}(E)$;
    \item[(ii)] $\eta_{\sigma,h}(x,t)=0 $ for all $t\in[0,T)$ and $x\in E^{t}\setminus E^{t,\sigma/2}$;
    \item[(iii)] $\partial_t \eta_{\sigma, h}^{q+1} (x,t) \ge 0$ for a.e.~$(x,t) \in E$ with $x \in E^{t,2\sigma}$.
\end{enumerate}
If $E$ fulfills the stronger condition \eqref{one sided growth condition} with $\rho \colon (-1,T)\rightarrow(0,\infty)$ satisfying $\rho \in W^{1,r}(0,T)$ for
$r$ defined according to \eqref{definition_of_r_for_one_sided_growth}, we additionally have that
\begin{enumerate}
    \item[(iv)] $\partial_t\eta_{\sigma,h}^{q+1}(x,t)\geq-\frac{1}{\sigma} \frac{\vert\rho(t)-\rho(t-h)\vert}{h}$ for a.e.~$(x,t)\in E$. 
\end{enumerate}
\end{lemma}
\begin{proof}
    By definition of $\eta_{\sigma,h}$ and Lebesgue's
      differentiation theorem, we have 
\begin{equation}
  \partial_t\eta_{\sigma,h}^{q+1}(x,t)=\frac{\eta_{\sigma}(x,t)-\eta_{\sigma}(x,t-h)}{h}
  \qquad\mbox{for a.e.~$(x,t)\in E$}.
  \label{eq:derivative}
\end{equation}
The properties (i) and (iii) follow from (\ref{eq:derivative}) as
in \cite[Lemmas 5.9]{BDSS18}.
In particular, they do not depend on the evolution of the set $E$.

Next, fix $x\in E^{t}\setminus E^{t,\sigma/2}$
and $0<h\leq\omega^{-1} \big( \frac\sigma2 \big)$. We wish to show that
\begin{align}\label{eq:eta_sigma_equal_to_0}
\eta_{\sigma}(x,s)=0\quad\text{for any }s\in[t-h,t].
\end{align}
In the case $s\in[-h,0)$ and in the case $s\in[0,T),x\not\in E^{s}$,
the claim is clear from the construction of $\eta_{\sigma}$. Therefore,
suppose that $s\in[0,T)$ and $x\in E^{s}$. Let $y\in\partial E^{t}$
be such that $\left|x-y\right|=\dist(x,\Omega\setminus E^{t})$. Then
we estimate
\begin{align*}
\dist \big( x,\Omega\setminus E^{s} \big) & \leq\left|x-y\right|+\dist \big( y,\Omega\setminus E^{s} \big)\\
 & \leq\dist \big( x,\Omega\setminus E^{t} \big)+\boldsymbol{e}^{c} \big( E^{s},E^{t} \big)\\
 & \leq\tfrac{\sigma}{2}+\omega(t-s)\\
 & \leq\tfrac{\sigma}{2}+\omega \big( \omega^{-1} \big( \tfrac{\sigma}{2} \big) \big)
 \leq\sigma.
\end{align*}
Due to the definition of $\eta_{\sigma}$, this implies \eqref{eq:eta_sigma_equal_to_0} in the remaining case. Furthermore, we obtain (ii) by the definition of $\eta_{\sigma,h}$.

Finally, we establish (iv).
To this end, let $(x,t)\in E$. The claim is clear if $t-h<0$, since in this case we have that $\eta_{\sigma}(x,t)-\eta_{\sigma}(x,t-h)=\eta_{\sigma}(x,t)\geq0$.
Similarly, it is clear if $x\not\in E^{t-h}$. Therefore, consider
$t-h\geq0$ and $x\in E^{t-h}$. Let $y\in\partial E^{t}$ be such
that $\left|x-y\right|=\dist( x,\Omega\setminus E^{t})$. Using that
$x\in E^{t-h}$, $y\in\partial E^{t}\subset\Omega\setminus E^{t}$
and the definition of $\boldsymbol{e}^{c}$, we estimate
\begin{align*}
\dist \big( x,\Omega\setminus E^{t-h} \big) & \leq\left|x-y\right|+\dist \big( y,\Omega\setminus E^{t-h} \big)\\
 & \leq\dist \big( x,\Omega\setminus E^{t} \big) + \boldsymbol{e}^{c} \big( E^{t-h},E^{t} \big)\\
 & \leq\dist \big( x,\Omega\setminus E^{t} \big) + \vert\rho(t-h)-\rho(t)\vert,
\end{align*}
where in the last estimate we used (\ref{one sided growth condition}). Now,
since $\tilde{\eta}$ is nondecreasing and $1$-Lipschitz, using the preceding inequality we obtain that
\begin{align*}
    &\eta_{\sigma}(x,t)-\eta_{\sigma}(x,t-h)
    =
    \tilde{\eta}\left(\frac{\dist \big( x,\Omega\setminus E^{t} \big) }{\sigma}\right)
    - \tilde{\eta}\left(\frac{\dist\big( x,\Omega\setminus E^{t-h} \big)}{\sigma}\right)\\
    & \geq
    \tilde{\eta}\left(\frac{\dist\big( x,\Omega\setminus E^{t} \big)}{\sigma}\right)
    - \tilde{\eta}\left(\frac{\dist \big( x,\Omega\setminus E^{t} \big) + \vert \rho(t-h)-\rho(t) \vert}{\sigma}\right)\\
    & \geq
    -\tfrac{1}{\sigma}|\rho(t-h)-\rho(t)|.
\end{align*}
Thus, by (\ref{eq:derivative}) we arrive at (iv).
\end{proof}

We are now able to prove Theorem~\ref{thm:ineq_integration_by_parts_formula}.
\begin{proof}[Proof of Theorem~\ref{thm:ineq_integration_by_parts_formula}]
Since we know that $\spt(\eta_{\sigma,h}) \subset \overline{E^{t,\frac{\sigma}{2}}}$ and $\eta_{\sigma,h}^{q+1} \in C^{0,1}(E)$ by Lemma~\ref{lem:properties_for_eta_sigma_h}, we are allowed to apply Theorem~\ref{thm:integration_by_parts_compact_spt}, which gives us that
\begin{align}
	\big\langle \partial_t \power{u}{q}, &\eta_{\sigma,h}^{q+1} \zeta (v-u) \big\rangle
	\nonumber \\ &=
    \iint_E \eta_{\sigma,h}^{q+1} \partial_t v \cdot \zeta \big( \power{v}{q} - \power{u}{q} \big) \,\dx\dt
    + \iint_E \eta_{\sigma,h}^{q+1} \zeta^{\prime} \b[ u,v ] \,\dx\dt
    \label{ineq:integration_by_parts_formula_applied_to_eta_sigma_h}
    \\ &\phantom{=}
    + \iint_E \partial_t \eta_{\sigma,h}^{q+1} \zeta \b[u,v] \,\dx\dt
    \nonumber
\end{align}
holds for any non-negative cut-off function in time $\zeta \in C_{0}^{0,1}((0,T))$, $\sigma > 0$ and $h \in \big(0, \omega^{-1} \big( \tfrac{\sigma}{2} \big) \big)$, where $\omega$ denotes the modulus of continuity in \eqref{ineq:one_sided_growth_with_modulus}.
Since we have that $\partial_t \eta_{\sigma, h}^{q+1}(t) \ge 0$ a.e.~in $E^{t,2\sigma}$ by Lemma~\ref{lem:properties_for_eta_sigma_h} (iii), and by using Lemma~\ref{lem:properties_for_eta_sigma_h} (iv) for the last term in \eqref{ineq:integration_by_parts_formula_applied_to_eta_sigma_h}, we deduce that
\begin{align}
    \iint_E \partial_t \eta_{\sigma,h}^{q+1} \zeta \b[u,v] \,\dx\dt
    &\ge
    \int_{0}^{T} \int_{E^t \setminus E^{t,2\sigma}} \partial_t \eta_{\sigma,h}^{q+1} \zeta \b[u,v] \,\dx\dt \nonumber \\
    &\geq -\frac{1}{\sigma}\int_{0}^{T}\int_{E^{t}\setminus
      E^{t,2\sigma}}\frac{\left|\rho(t)-\rho(t-h)\right|}{h}\zeta\b[u,v] \,\dx\dt\nonumber\\
    &\underset{h\downarrow0}{\longrightarrow} -\frac{1}{\sigma}\int_{0}^{T}|\rho'(t)|\int_{E^{t}\setminus
      E^{t,2\sigma}}\b[u,v] \,\dx\dt.
    \label{ineq.estimate_for_last_term_in_integration_by_parts_formula}
\end{align}
In the last step, we used the convergence
$\frac{\left|\rho(t)-\rho(t-h)\right|}{h}\to \rho^\prime(t)$ in
$L^1(0,T)$ as $h\downarrow0$ together with the fact $\b[u,v]\in
L^\infty(0,T;L^1(\Omega))$, which follows from $u,v\in L^{\infty}(0,T;L^{q+1}(\Omega, \mathds{R}^{N}))$.  
Moreover, since we have $u-v \in \mathcal{V}_{q}^{p,0}(E)$, $\zeta \in C^{0,1}_0((0,T))$, $\eta_{\sigma,h} \to \eta_\sigma$ and $D\big(\eta_{\sigma,h}^{q+1}\big) \to D\big( \eta_{\sigma}^{ q+1} \big)$ a.e.~in $\Omega_T$ as $h \downarrow 0$ by Lebesgue's differentiation theorem, and we know that $0 \leq \eta_{\sigma,h} \leq 1$, as well as
\begin{align*}
    \big| D \eta_{\sigma,h}^{q+1}(x,t) \big|
    =
    \bigg| \bint_{t-h}^t D\eta_\sigma(x,s) \,\ds \bigg|
    \leq
    \|D\eta_\sigma\|_{L^\infty(\Omega_T)}
    \leq
    \frac{1}{\sigma}
\end{align*}
a.e.~in $E$ for all $h>0$, by the dominated convergence theorem we conclude that
\begin{align}\label{limit_for_eta_zeta_u}
    \left\{
    \begin{array}{ll}
        \eta_{\sigma, h}^{q+1} u \to \eta_{\sigma} u & \text{in $L^{q+1}(\Omega_T, \mathds{R}^{N})$ as $h \downarrow 0$}, \\[5pt]
        \eta_{\sigma, h}^{q+1} v \to \eta_{\sigma} v & \text{in $L^{q+1}(\Omega_T, \mathds{R}^{N})$ as $h \downarrow 0$}, \\[5pt]
        \eta_{\sigma, h}^{q+1} \zeta (v-u) \to \eta_{\sigma} \zeta (v-u) & \text{in $\mathcal{V}^{p,0}_q(E)$ as $h \downarrow 0$.}
    \end{array}
    \right.
\end{align}
Hence, plugging \eqref{ineq.estimate_for_last_term_in_integration_by_parts_formula} into \eqref{ineq:integration_by_parts_formula_applied_to_eta_sigma_h}, and using \eqref{limit_for_eta_zeta_u} to pass to the limit $h \downarrow 0$, we obtain that
\begin{align}
    \big\langle \partial_t \power{u}{q}, \eta_{\sigma} \zeta (v-u) \big\rangle
    &\geq
   \iint_E \eta_{\sigma} 
    \big( \partial_t v \cdot \zeta \big( \power{v}{q} - \power{u}{q} \big) + \zeta^{\prime} \b[ u,v ] \big) \,\dx\dt
    \label{integration_by_parts_formula_before_taking-the_sigma_limit}
   \\&\phantom{=}
   -\frac{1}{\sigma} \Vert \zeta \Vert_{L^{\infty}}   \int_{0}^{T} \int_{E^t \setminus E^{t,2\sigma}}  \vert  \rho^{\prime}(t) \vert \b[u,v] \,\dx\dt .
   \nonumber
\end{align}
We claim that the last term vanishes in the limit
$\sigma\downarrow0$. For the proof of this claim, we distinguish
between three cases. We begin with the case $p \geq q+1$, in which $r=p'$.
From Lemma \ref{lem:technical_lemma} the estimate
\begin{align}\label{ineq:new_estimate_for_b_term}
    \b[u,v] \le c (|u|+|v|)^{q}|u-v|
\end{align}
follows immediately.
Furthermore, using H\"older's inequality with exponents $p'$ and $p$, taking into account that $\dist \big( x, \partial E^t \big) \le 2
\sigma$ on $E^t \setminus E^{t, 2\sigma}$, and applying Lemma \ref{lem:p-Hardy} yields
\begin{align*}
  \frac{1}{\sigma} &\int_{0}^{T} \int_{E^t \setminus E^{t,2\sigma}} \vert \rho^{\prime}(t) \vert \b[u,v] \,\dx\dt\\
   &\le
   c \int_{0}^{T}  \int_{E^t \setminus E^{t,2\sigma}}  \vert \rho^{\prime}(t) \vert(|u|+|v|)^{q} \frac{|u-v|}{\sigma} \,\dx \, \dt 
   \\ &\leq
   c \Bigg( \int_0^T |\varrho'(t)|^{p'}
    \int_{E^t \setminus E^{t,2\sigma}} (|u|+|v|)^{qp'} \,\dx \dt \Bigg)^\frac{1}{p'}
    \\ &\phantom{=}\quad \cdot
   \Bigg( \int_0^T \int_{E^t} \bigg| \frac{u-v}{\dist(x,\partial E^t)} \bigg|^p \,\dx \, \dt \Bigg)^\frac{1}{p}
   \\ &\leq
   c \Bigg( \int_0^T |\varrho'(t)|^{p'}
    \int_{E^t \setminus E^{t,2\sigma}} (|u|+|v|)^{qp'} \,\dx \dt \Bigg)^\frac{1}{p'}
    \bigg( \iint_{E} | D(u-v)|^p \,\dx \dt \bigg)^\frac{1}{p}.
\end{align*}
Note that the second term on the right-hand side of the preceding inequality is finite, since $u,v \in \mathcal{V}^p(E)$.
Now, we set $F_\sigma \colon (0,T) \to [0,\infty]$,
\begin{align*}
    F_\sigma(t) :=
    |\varrho'(t)|^{p'}
    \int_{E^t \setminus E^{t,2\sigma}} (|u|+|v|)^{qp'} \,\dx .
\end{align*}
Since $\varrho' \in L^{p'}(0,T)$, $u,v \in L^\infty(0,T;L^{q+1}(\Omega,\R^N))$, and $qp'\le q+1$ in the case $p\ge q+1$, we have that $F_\sigma \in L^1(0,T)$ for any $\sigma>0$ with
\begin{align*}
    \sup_{\sigma>0} F_\sigma(t) :=
    |\varrho'(t)|^{p'}
    \int_{\Omega\times\{t\}} (|u|+|v|)^{qp'} \,\dx
    \in L^1(0,T).
\end{align*}
Further, since $x \mapsto (\left|u(x,t)\right|+\left|v(x,t)\right|)^{qp'}\in L^{1} (\Omega)$ for almost every $t \in (0,T)$ and $\left|E^{t}\setminus E^{t,2\sigma}\right| \rightarrow 0$ for every $t \in (0,T)$, we have that $F_\sigma(t)\rightarrow0$ for almost every $t\in(0,T)$.
Hence, by the dominated convergence theorem the first term on the right-hand side in the antepenultimate display formula vanishes as $\sigma \downarrow 0$.

Next, we consider the case $\frac{(n+1)(q+1)}{n+q+1} < p< q+1$, in which we have that
\begin{equation*}
    r=\frac{p(n+q+1)-n(q+1)}{p(n+q+1)-(n+1)(q+1)}.
\end{equation*}
Note that this implies $r'<p<q+1$.
Using \eqref{ineq:new_estimate_for_b_term} and H\"older's inequality
with exponents $r,\frac{pr'}{p-r'}$ and $p$, we estimate
\begin{align*}
  &\frac{1}{\sigma}\int_0^T \int_{E^t\setminus
    E^{t,2\sigma}}|\rho'(t)|\b[u,v] \,\dx \dt\\
  &\qquad\le c
    \int_0^T \int_{E^t\setminus
    E^{t,2\sigma}}|\rho'(t)|\big(|u|+|v|\big)^{\frac{q+1}{r}}\big(|u|+|v|\big)^{\frac{q+1-r'}{r'}}\frac{|u-v|}{\sigma}
    \,\dx \dt\\
  &\qquad\le c
  \bigg(\int_0^T|\rho'(t)|^r \int_{E^t}\big(|u|+|v|\big)^{q+1} \,\dx \,\dt\bigg)^{\frac{1}{r}}\\
   &\qquad\qquad\cdot \bigg(\int_0^T\int_{E^t\setminus
     E^{t,2\sigma}}\big(|u|+|v|\big)^{\frac{p(n+q+1)}{n}} \, \dx \, \dt\bigg)^{\frac{p-r'}{pr'}}\\
  &\qquad\qquad\cdot\Bigg(\int_0^T\int_{E^t}\bigg| \frac{u-v}{\dist(x,\partial E^t)}\bigg|^{p} \,\dx \,\dt\Bigg)^{\frac{1}{p}}.
\end{align*}
In the last step, we used the fact $r'=\frac{p(n+q+1)-n(q+1)}{q+1}$,
which implies that
\begin{equation*}
  \frac{(q+1-r')p}{p-r'}= \frac{p(n+q+1)}{n}.
\end{equation*}
Now, the first integral on the right-hand side of the preceding inequality is finite, since $\rho'\in L^r(0,T)$ and $u,v\in L^\infty(0,T;L^{q+1}(\Omega,\R^N))$. The last integral is finite by
Hardy's inequality, cf.~Lemma~\ref{lem:p-Hardy}. 
Because the Gagliardo--Nirenberg inequality implies $u,v\in L^{\frac{p(n+q+1)}{n}}(\Omega_T)$, the second integral is also finite, and moreover, we have that
\begin{equation*}
  \int_0^T\int_{E^t\setminus
     E^{t,2\sigma}}\big(|u|+|v|\big)^{\frac{p(n+q+1)}{n}}\dx\dt\to0
 \end{equation*}
in the limit $\sigma\downarrow0$.
Thus, the right-hand side of the antepenultimate display formula vanishes in the limit $\sigma\downarrow0$.

Finally, in the borderline case $p= \frac{(n+1)(q+1)}{n+q+1}$ we have that
\begin{align*}
    qp' =
    \frac{qp}{\frac{(n+1)(q+1)}{n+q+1}-1}
    =
    \frac{p(n+q+1)}{n}.
\end{align*}
Therefore, setting $r=\infty$ and $r'=1$, the preceding computations extend to this case, if we interpret $\frac{1}{\infty}$ as zero and replace $\bigg(\int_0^T|\rho'(t)|^r \, \dt\bigg)^{\frac{1}{r}}$ by $\|\varrho\|_{L^\infty(0,T)}$.

Hence, joining the three preceding cases, in any case we have that 
\begin{equation}\label{boundary-term-vanish}
    \lim_{\sigma \downarrow 0} \frac{1}{\sigma} \int_{0}^{T} \int_{E^t \setminus E^{t,2\sigma}} \vert \rho^{\prime}(t) \vert \b[u,v] \,\dx\dt
  = 0.
\end{equation}
Moreover, by Lemma~\ref{weak_convergence_curoff_density_part} and the dominated convergence theorem, we conclude that
\begin{align}\label{sigma_convergence}
    \left\{
    \begin{array}{ll}
        \eta_{\sigma} \zeta (v-u) \wto \zeta (v-u)  & \text{weakly in $\mathcal{V}_{q}^{p,0}(E)$ as $\sigma \downarrow 0$}, \\[5pt]
        \eta_{\sigma}u \to  u & \text{in $L^{q+1}(E,\R^N)$ as $\sigma \downarrow 0$} \\[5pt]
        \eta_{\sigma}v \to  v & \text{in $L^{q+1}(E,\R^N)$ as $\sigma \downarrow 0$}. 
    \end{array}
    \right.
\end{align}
Therefore, passing to the limit $\sigma \downarrow 0$ in
\eqref{integration_by_parts_formula_before_taking-the_sigma_limit} by
means of \eqref{boundary-term-vanish} an~\eqref{sigma_convergence}, we infer
\begin{align*}
    \langle \partial_t \power{u}{q}, \zeta (v-u) \rangle
    \geq
	\iint_E \partial_t v \cdot \zeta \big( \power{v}{q} - \power{u}{q} \big) \,\dx\dt
	+ \iint_E \zeta^{\prime} \b[ u,v ] \,\dx\dt.
\end{align*}
This concludes the proof of the theorem.
\end{proof}

\subsection{Left-sided continuity in time}
\label{sec:left-sided-continuity}
We begin this section with an auxiliary lemma on the continuity in time in cylindrical domains.
\begin{lemma}\label{lem:time-continuity-cylindrical}
Consider exponents $q>0$, $p> \max \Big\{1, \frac{n(q+1)}{n+q+1} \Big\}$, and a function $u\in L^p(0,T;W^{1,p}_0(\Omega,\R^N))\cap L^\infty(0,T;L^{q+1}(\Omega,\R^N))$ such that the power $\power{u}{q}$ possesses a weak time derivative $\partial_t\power{u}{q}\in L^{p'}(0,T;W^{-1,p'}(\Omega,\R^N))$.
Then we have that $u\in C^0([0,T];L^{q+1}(\Omega,\R^N))$.
\end{lemma}

\begin{proof}
  For $h \in (0,T]$, we consider the time mollification $[u]_h$ according to \eqref{eq:time_mollification} with zero initial values.
  Since $[u]_h\in L^p(0,T;W^{1,p}_0(\Omega,\R^N))\cap L^\infty(0,T;L^{q+1}(\Omega,\R^N))$ with $\partial_t[u]_h\in L^{q+1}(\Omega_T,\R^N)$, by using the integration by parts formula~\eqref{ineq:general_integration_by_parts_formula} (which holds with equality in the cylindrical domain $\Omega_T$) with $v=[u]_h$, we obtain that
  \begin{equation*}
    -\iint_{\Omega_T} \zeta^{\prime}\, \b\big[u,[u]_h \big] \,\dx\dt
    =
    \iint_{\Omega_T}\partial_t[u]_h\cdot\zeta
    \big( \power{[u]_h}{q}-\power{u}{q} \big) \,\dx\dt
    -
    \big\langle \partial_t\power{u}{q},\zeta([u]_h-u) \big\rangle
  \end{equation*}
  for any cut-off function $\zeta\in C^{0,1}_0((0,T))$.
  Now, for arbitrary $\tau\in \left[\frac T2,T\right]$ and $\eps\in \left(0,\frac T4\right)$ we define
  \begin{equation*}
    \zeta(t):=
    \left\{
    \begin{array}{cl}
      \frac{4t}{T},&\mbox{for }0\le t< \frac T4,\\[0.6ex]
      1,&\mbox{for }\frac T4\le t<\tau-\eps,\\[0.6ex]
      \frac{\tau-t}{\eps},&\mbox{for }\tau-\eps\le t< \tau,\\[0.6ex]
      0,&\mbox{for }\tau\le t\le T.
    \end{array}
    \right.
  \end{equation*}
  With this choice of cut-off function, the preceding formula becomes
  \begin{align*}
    \tfrac1\eps\int_{\tau-\eps}^{\tau}\int_{\Omega} \b\big[u,[u]_h\big]\,\dx\dt
    &=
    \iint_{\Omega_T}\partial_t[u]_h\,\cdot \zeta \big(\power{[u]_h}{q}-\power{u}{q}\big) \,\dx\dt
    -
    \big\langle \partial_t\power{u}{q},\zeta([u]_h-u)\big\rangle\\
    &\quad +\tfrac{4}{T}\int_0^{\frac{T}{4}}\int_{\Omega} \b[u,[u]_h]\,\dx\dt.
  \end{align*}
  Due to \eqref{eq:ODE_mollification}, the first term on the right-hand side is nonpositive.
  Letting $\eps\downarrow0$, by Lebesgue's differentiation theorem we therefore obtain the bound
  \begin{align*}
    \int_{\Omega\times\{\tau\}}\b\big[u,[u]_h\big] \,\dx
    &\le
    \big\|\partial_t\power{u}{q}\big\|_{L^{p'}(0,T;W^{-1,p'}(\Omega,\R^N))} \big\|\zeta \big( [u]_h-u \big) \big\|_{L^p(0,T;W^{1,p}_0(\Omega,\R^N))}\\
    &\quad+
      \tfrac{4}{T}\int_0^{\frac{T}{4}}\int_{\Omega} \b\big[u,[u]_h\big] \,\dx\dt
  \end{align*}
  for a.e.~$\tau\in\left[\frac T2,T\right]$. The terms on the right-hand side vanish
  in the limit $h\downarrow0$ since $[u]_h\to u$ in
  $L^p(0,T;W^{1,p}_0(\Omega,\R^N))$ and in $L^{q+1}(\Omega_T,\R^N)$. Using,
  moreover, Lemma~\ref{lem:technical_lemma} in order to bound the
  left-hand side from below, we obtain
  \begin{align*}
   \esssup_{\tau\in\left[\frac T2,T\right]}\,
    \int_{\Omega\times\{\tau\}}\Big|\power{u}{\frac{q+1}{2}}-\power{[u]_h}{\frac{q+1}{2}}\Big|^2\dx
    \to0
    \qquad\mbox{as $h\downarrow0$}.
  \end{align*}
  This implies in particular that
  $\power{[u]_h}{\frac{q+1}{2}} \in C^0\left(\left[\frac T2,T\right],L^2(\Omega,\R^N)\right)$ converges to $\power{u}{\frac{q+1}{2}}$ in the norm of this space as $h\downarrow0$.
  Therefore,
  after choosing a suitable representative for $u$,
  we have that $\power{u}{\frac{q+1}{2}}\in C^0\left(\left[\frac T2,T\right],L^2(\Omega,\R^N)\right)$.
  Now, for $t,\tau\in \left[\frac T2,T\right]$, Lemma~\ref{lem:technical_lemma_2} and the
  Cauchy--Schwarz inequality imply that
  \begin{align*}
    &\int_\Omega|u(t)-u(\tau)|^{q+1}\dx\\
     &\qquad\le
     c(q)\int_\Omega\big(|u(t)|+|u(\tau)|\big)^{\frac{q+1}{2}}
     \Big|\power{u(t)}{\frac{q+1}{2}}-\power{u(\tau)}{\frac{q+1}{2}}\Big|\dx \\
    &\qquad\le
    c(q)\bigg(\int_\Omega\big(|u(t)|+|u(\tau)|\big)^{q+1}\dx\bigg)^{\frac12}
    \bigg(\int_\Omega
    \Big|\power{u(t)}{\frac{q+1}{2}}-\power{u(\tau)}{\frac{q+1}{2}}\Big|^2\dx\bigg)^{\frac12}
    \to0
  \end{align*}
  in the limit $t\to \tau$.
  Consequently, we have $u\in C^0\left(\left[\frac
  T2,T\right],L^{q+1}(\Omega,\R^N)\right)$.
  The analogous argument for $u(x,-t)$ implies $u\in C^0\left(\left[0,\frac T2\right],L^{q+1}(\Omega,\R^N)\right)$ as well, which completes the proof of the lemma. 
\end{proof}

Next, we consider a representative of $u \in L^{\infty}(0,T;L^{q+1}(\Omega, \mathds{R}^{N}))$ such that
\begin{align}\label{limit_of_representative}
	\left\{
	\begin{array}{ll}
	\displaystyle\power{u}{q}(x,\tau) = \liminf_{h \downarrow 0} \tfrac{1}{h} \int_{\tau - h}^{\tau} \power{u}{q}(x,s) \,\ds
	&\text{ for a.e.~$x \in E^{\tau}$}, \\[10pt]
	u(x,\tau) = u_\ast(x)
	&\text{ for a.e.~$x \in \Omega \setminus E^\tau$}
	\end{array}
	\right.
\end{align}
is satisfied for any $\tau \in (0,T)$.
In this situation, Lemma \ref{lem:time-continuity-cylindrical} enables us to prove the following result.
\begin{lemma}\label{lem:integral_limit_for_representative_of_u}
    Let $p \geq \frac{(n+1)(q+1)}{n+q+1}$, and assume that the domain $E$ satisfies \eqref{ineq:lower_measure_bound}, and \eqref{one sided growth condition} with $r$ given by \eqref{definition_of_r_for_one_sided_growth}.
    Consider $u \in V_{q}^{p}(E) \cap L^{\infty}(0,T;L^{q+1}(\Omega, \mathds{R}^{N}))$ with time derivative $\partial_t \power{u}{q} \in (V^{p,0}(E))'$.   
Then any representative of $u$ that satisfies \eqref{limit_of_representative}, also fulfills 
    \begin{align}\label{statement:convergence_of_q+1_integral_of_u}
        \lim_{h \downarrow 0} \tfrac{1}{h} \iint_{\Omega \times (\tau_o -h, \tau_o)} \b[u(t),u_\ast]\,\dx\dt = \int_{\Omega} \b[u(\tau_o),u_\ast] \,\dx
    \end{align}
    for any $\tau_o \in (0,T)$. 
  \end{lemma}

  \begin{proof}
     Since the function $x\mapsto|x|^{\frac{q+1}{q}}$ is convex,
assumption \eqref{limit_of_representative} and Jensen's
inequality imply
\begin{align*}
  \b[u(\tau_o),u_\ast]\le\liminf_{h\downarrow0}\tfrac1h\int_{\tau_o-h}^{\tau_o}\b[u(t),u_\ast]\dt
\end{align*}
a.e.~in $\Omega$. Note that this inequality holds trivially in
$\Omega\setminus E^{\tau_o}$ by
\eqref{limit_of_representative}. Combining the preceding result with Fatou's lemma, we deduce that
\begin{align}\label{liminf_integal_estimate_ofu(tau_o)}
  \int_\Omega \b[u(\tau_o),u_\ast]\,\dx
  \le
  \liminf_{h \downarrow 0} \tfrac{1}{h} \iint_{\Omega \times (\tau_o - h, \tau_o)} \b[u(t),u_\ast] \,\dx\dt.
\end{align}
Hence, it remains to develop the opposite inequality.
To this end, fix $\tau_o \in (0,T)$.
Further, if $q \geq 1$, we let $\widetilde{\eta} \in
C^{0,1}(\mathds{R})$ be the function in \eqref{definition_eta_sigma} and if $q \in (0,1)$, we set
\begin{align*}
	\widetilde{\eta}(s) :=
	\left\{
	\begin{array}{ll}
	0 &\text{for } s \in [0,1], \\[5pt]
	2^\frac{1-q}{q} (s-1)^\frac{1}{q}
	&\text{for } s \in \left(1, \frac{3}{2} \right], \\[5pt]
	1 - 2^\frac{1-q}{q} (2-s)^\frac{1}{q}
	&\text{for } s \in \left( \frac{3}{2}, 2 \right], \\[5pt]
	1 & \text{for } s \in (2,\infty).
	\end{array}
	\right.
\end{align*}
For $h \in (0, \tau_o)$, we define $h_r:=h^{1-\frac 1r}$ and $\eta_h \colon \R^n \to [0,1]$ by letting
\begin{align}
    \eta_{h}(x)
    :=
    \widetilde{\eta} \bigg( \frac{\dist(x, \Omega \setminus E^{\tau_o})}{h_r} \bigg).
    \label{def:definition_of_eta_h}
\end{align}
Then, we have that $\eta_h, \eta_h^q \in C^{0,1}(\Omega)$.
Using that $\eta_h \equiv 1$ in $E^{\tau_o,2h_r}$, we split the integral on the right-hand side of \eqref{liminf_integal_estimate_ofu(tau_o)} into
    \begin{align}\label{split_integral_in_proof_into_two}
        \tfrac{1}{h} \iint_{\Omega \times (\tau_o - h, \tau_o)}
      \b[u(t),u_\ast] \,\dx\dt
      &= \tfrac{1}{h} \int_{\tau_o - h}^{\tau_o} \int_{E^{t} \setminus E^{\tau_o, 2h_r}} \big( 1- \eta_{h}^{q+1} \big) \b[u(t),u_\ast] \,\dx\dt \nonumber \\
        &\phantom{=}
        + \tfrac{1}{h} \int_{\tau_o - h}^{\tau_o} \int_{\Omega} \eta_{h}^{q+1} \b[u(t),u_\ast] \,\dx \dt \nonumber \\
        &=: \mathrm{I}_{h} + \mathrm{II}_{h}. 
    \end{align}
In the following, we will deal with the terms $\mathrm{I}_h$ and $\mathrm{II}_h$, whose definitions are clear from the context, separately.
Let us first consider $\mathrm{I}_h$.
For any $x \in E^{t} \setminus E^{\tau_o, 2h_r}$ with $t \in (\tau_o-h, \tau_o)$, there exists $y \in \Omega \setminus E^{\tau_o}$ such that $\vert x-y \vert = \dist(x, \Omega \setminus E^{\tau_o}) \le 2h_r$.
Thus, by~\eqref{one sided growth condition} and the Sobolev embedding
$W^{1,r}((-1,T))\subset C^{0,1-\frac1r}((-1,T))$  we obtain that   
    \begin{align*}
        \dist(x, \Omega \setminus E^{t}) & \le \vert x-y \vert + \dist(y, \Omega \setminus E^{t} ) \\
        &\le \vert x-y \vert + \power{e}{c}(E^{t}, E^{\tau_o}) \\
        &\le 2h_r + \vert \rho(\tau_o) - \rho(t) \vert \\
        &\le 2h_r + \Vert \rho \Vert_{W^{1,r}((-1,T))}\vert \tau_o - t \vert^{1-\frac{1}{r}} \\
        &\le  2h_r + \Vert \rho
          \Vert_{W^{1,r}((-1,T))}h^{1-\frac{1}{r}}\\
        &= c_o h_r,
    \end{align*}
    with the constant $c_o:=2+\Vert \rho
    \Vert_{W^{1,r}((-1,T))}$.
This means that $E^{t} \setminus E^{\tau_o, 2h_r} \subset E^{t} \setminus E^{t,c_oh_r}$, which enables us to estimate
    \begin{align*}
      \mathrm{I}_{h}
      &\le
        \tfrac{1}{h} \int_{\tau_o - h}^{\tau_o}
        \int_{E^{t} \setminus E^{t,c_oh_r} } \b[u(t),u_\ast] \,\dx\dt\\
      &\le
        \tfrac{c(q)}{h} \int_{\tau_o - h}^{\tau_o}
        \int_{E^{t} \setminus E^{t,c_oh_r} } \big( |u(t)|^q+|u_\ast|^q \big) |u(t)-u_\ast|  \,\dx\dt\\
      &\le
        c(q)\Big( \|u\|_{L^\infty(0,T;L^{q+1}(\Omega,\R^N))}^q
        + \|u_\ast\|_{L^{q+1}(\Omega,\R^N)}^q \Big)
        \\ &\phantom{=} \cdot
        \bigg(\tfrac{1}{h} \int_{\tau_o - h}^{\tau_o} \int_{E^{t} \setminus E^{t,c_oh_r} } |u(t)-u_\ast|^{q+1} \,\dx\dt\bigg)^{\frac{1}{q+1}}.
    \end{align*}
    In the last steps, we used first Lemma~\ref{lem:technical_lemma}
    and then H\"older's inequality. 
In order to estimate the last integral further, we distinguish between two cases.
First, let $p \ge q+1$.
In this case, we have that $r=p'$.
Applying H\"older's inequality, the $p$-Hardy inequality in Lemma \ref{lem:p-Hardy}, and again H\"older's inequality yields
\begin{align*}
    \tfrac{1}{h} \int_{\tau_o - h}^{\tau_o} &\int_{E^{t} \setminus E^{t,c_oh_r}}  \vert u(t) -u_\ast\vert^{q+1} \,\dx\dt \nonumber \\
    &\le
    \big(c_oh_r\big)^{q+1} \tfrac{1}{h} \int_{\tau_o - h}^{\tau_o} \int_{E^t} \bigg( \frac{\vert u(x,t)-u_\ast(x) \vert}{\dist(x, \partial E^{t} )} \bigg)^{q+1} \,\dx\dt \nonumber \\
    &\le
    c_o^{q+1} |\Omega|^{\frac{p-(q+1)}{p}} h^{\frac{q+1-p}{p}}
    \int_{\tau_o -h}^{\tau_o} \bigg( \int_{E^{t}} \bigg( \frac{\vert u(x,t)-u_\ast(x) \vert}{\dist(x, \partial E^{t})} \bigg)^{p} \,\dx \bigg)^{\frac{q+1}{p}} \,\dt \nonumber \\
    &\le
    c\, 
    c_o^{q+1} |\Omega|^{\frac{p-(q+1)}{p}} h^{\frac{q+1-p}{p}}
    \int_{\tau_o - h}^{\tau_o} \bigg( \int_{\Omega} \vert Du-Du_\ast \vert^{p} \,\dx \bigg)^{\frac{q+1}{p}} \dt
    \nonumber \\ &\leq
    c\,
    c_o^{q+1} |\Omega|^{\frac{p-(q+1)}{p}}
    \bigg( \int_{\tau_o - h}^{\tau_o} \int_{\Omega} \vert Du-Du_\ast \vert^{p} \,\dx\dt \bigg)^\frac{q+1}{p}.
\end{align*}
Hence, the right-hand side of the preceding inequality converges to zero as $h \downarrow 0$.
Next, we deal with the cases $\frac{(n+1)(q+1)}{n+q+1} < p < q+1$, where $r=\tfrac{p(n+q+1)-n(q+1)}{p(n+q+1)-(n+1)(q+1)}$, and $p= \frac{(n+1)(q+1)}{n+q+1}$, where $r=\infty$.
By H\"older's inequality with exponents $\frac{p}{r'}$ and $\frac{p}{p-r'}$ (note that $r'<p$ in this parameter range), and Lemma \ref{lem:p-Hardy} we obtain that
\begin{align*}
    \tfrac{1}{h} &\int_{\tau_o-h}^{\tau_o} \int_{E^{t} \setminus E^{t,c_oh_r}}  \vert u-u_\ast \vert^{q+1} \,\dx\dt \\
    &\le
    c\int_{\tau_o-h}^{\tau_o} \int_{E^{t} \setminus E^{t,c_oh_r}} \bigg|\frac{u-u_\ast}{\dist(x,\partial E^t)}\bigg|^{r'} |u-u_\ast|^{q+1-r'} \,\dx\dt \\
    &\le
    c \Bigg( \int_{\tau_o-h}^{\tau_o} \int_{E^t} \bigg|\frac{u-u_\ast}{\dist(x,\partial E^t)}\bigg|^{p} \,\dx\dt\Bigg)^{\frac{r'}{p}}
    \\ &\phantom{=} \cdot
    \bigg(\int_{\tau_o-h}^{\tau_o} \int_{E^t \setminus E^{t,c_oh_r}}|u-u_\ast|^{\frac{p(q+1-r')}{p-r'}} \,\dx\dt\bigg)^{\frac{p-r'}{p}}\\
    &\le
    c\bigg(\int_{\tau_o-h}^{\tau_o} \int_{E^t} |Du - Du_\ast|^p \,\dx\dt\bigg)^{\frac{r'}{p}}
    \\ &\phantom{=} \cdot
    \bigg(\int_{\tau_o-h}^{\tau_o}\int_{E^t \setminus E^{t,c_oh_r}}|u-u_\ast|^{\frac{p(q+1-r')}{p-r'}} \,\dx\dt\bigg)^{\frac{p-r'}{p}},
\end{align*}
which tends to zero in the limit $h\downarrow0$, since $\frac{p(q+1-r')}{p-r'} = p\frac{n+q+1}{n}$ and thus the second integral on the right-hand side of the preceding inequality is finite by the Gagliardo--Nirenberg inequality.
Combining both cases yields
\begin{align}\label{limit_of_first_split_integral}
    \lim_{h \to 0} \mathrm{I}_{h} =0.
\end{align}
Now, we proceed with the second integral $\mathrm{II}_{h}$.
Let $\epsilon_o \in (0,h)$ such that $E^{\tau_o,h_r} \times (\tau_o-2\epsilon_o, \tau_o+2\epsilon_o) \subset E$.
For $\epsilon \in (0,\epsilon_o)$ consider $\zeta \in C_{0}^{0,1}((0,T))$ given by
    \begin{equation*}
    \zeta(t) = \begin{cases} 
    \tfrac{1}{h}(t-\tau_o+h), & \text{for } t \in [\tau_o -h, \tau_o-\epsilon], \\[5pt]
    \Big(\tfrac{1}{\epsilon} - \tfrac{1}{h}\Big)(\tau_o-t), & \text{for } t \in (\tau_o-\epsilon, \tau_o], \\[5pt]
    0, & \text{else}.
    \end{cases}
    \end{equation*}
    First, recall Definition~\eqref{def:definition_of_eta_h}.
    Since $\eta_h^q\in C^{0,1}(\Omega)$ by
      definition, for the time derivative we have $\partial_t \power{(\eta_h u)}{q} = \eta_h^q
      \partial_t \power{u}{q} \in (V^{p,0}(E))'$. Therefore, by Remark \ref{rem:int_by_parts_measure_dens} we may
      apply Theorem~\ref{thm:ineq_integration_by_parts_formula} with
      $u$ replaced by $\eta_h u \in V^{p}_q(E)$ and $v \equiv \eta_h u_\ast$ and deduce the inequality
    \begin{align*}
        \big\langle \partial_t \power{u}{q}, &\zeta \eta_{h}^{q+1} (u - u_\ast) \big\rangle = \big\langle \partial_t  \power{(\eta_{h} u )}{q}, \zeta \eta_{h} (u - u_\ast)  \big\rangle \\
        &\leq -\iint_{E} \zeta^{\prime} \b[\eta_h u, \eta_h u_\ast] \,\dx\dt \\
         &=-\tfrac{1}{h} \int_{\tau_o - h}^{\tau_o} \int_{\Omega} \eta_{h}^{q+1} \b[u,u_\ast] \,\dx\dt + \tfrac{1}{\epsilon} \int_{\tau_o -\epsilon}^{\tau_o} \int_{\Omega} \eta_{h}^{q+1} \b[u,u_\ast] \,\dx\dt.
    \end{align*}
Together with the fact that $0 \leq \zeta \leq 1$, this enables us to establish the estimate
    \begin{align}\label{estimate_for_II_h}
        \mathrm{II}_{h}
        &\le
        \tfrac{1}{\epsilon} \int_{\tau_o - \epsilon}^{\tau_o} \int_{\Omega} \eta_{h}^{q+1} \b[u,u_\ast] \,\dx \dt 
        \\ &\phantom{=}
        +  \big\Vert \partial_t \power{u}{q} \Vert_{(V^{p,0}(E))'} \big\Vert \chi_{\Omega \times (\tau_o-h, \tau_o)} \eta_{h}^{q+1} (u-u_\ast) \big\Vert_{V^{p}(E)}.
        \nonumber
    \end{align}
    Note that in contrast to $\mathrm{II}_h$, the domain of integration on the right-hand side of the preceding inequality is independent of $h$.
    Further, since $E^{\tau_o,h_r} \times (\tau_o-2\epsilon_o, \tau_o+2\epsilon_o) \subset E$ and $u \in V^{p}(E)$, $\partial_t \power{u}{q} \in (V^{p,0}(E))'$, as well as $\spt(\eta_h) \subset \overline{E^{\tau_o,h_r}}$, we obtain that
    \begin{equation*}
	\left\{
	\begin{array}{ll}
		\eta_h u \in L^{p}(\tau_o - \epsilon, \tau_o; W_{0}^{1,p}(E^{\tau_o,h_r},\mathds{R}^{N})) \cap L^{q+1}(\Omega_{T}, \mathds{R}^{N}), \\[5pt]
		\partial_t \power{(\eta_h u)}{q} \in L^{p'}(\tau_o - \epsilon, \tau_o; W^{-1,p'}(E^{\tau_o,h_r},\mathds{R}^{N})).
	\end{array}
	\right.
\end{equation*}
Thus, Lemma~\ref{lem:time-continuity-cylindrical}, applied to $\eta_hu$ on the cylinder $E^{\tau_o,h_r}\times(\tau_o-\eps,\tau_o)$ implies that 
\begin{align}
    \eta_h u \in C^{0}([\tau_o - \epsilon, \tau_o];L^{q+1}(E^{\tau_o,h_r}, \mathds{R}^{N}))
    \label{claim_inclusion}
\end{align}
holds.
In turn, this yields
\begin{align*}
    \lim_{\epsilon \downarrow 0} \tfrac{1}{\epsilon} \int_{\tau_o -
  \epsilon}^{\tau_o} \int_{\Omega} \eta_{h}^{q+1} \b[u,u_\ast]  \,\dx
  \dt
  =
  \int_{\Omega} \eta_{h}^{q+1} \b[u(\tau_o),u_\ast]  \,\dx. 
\end{align*}
Hence, passing to the limit $\epsilon \downarrow 0$ in \eqref{estimate_for_II_h} leads us to
\begin{align}\label{estimate_for_II_h_after_taking_epsilon_limit}
    \mathrm{II}_{h}
    &\leq
    \int_{\Omega} \eta_{h}^{q+1} \b[u(\tau_o),u_\ast] \,\dx
    \\ &\phantom{=}
    + \Vert \partial_t \power{u}{q} \Vert_{(V^{p,0}(E))'} 
    \big\Vert \chi_{\Omega \times (\tau_o-h, \tau_o)} \eta_{h}^{q+1} (u-u_\ast) \big\Vert_{V^{p}(E)}.
    \nonumber   
\end{align}
At this stage, we would like to pass to the limit $h \downarrow 0$.
To this end, observe that by \eqref{def:definition_of_eta_h}, the facts $E^{t} \setminus E^{\tau_o, 2h_r} \subset  E^{t} \setminus E^{t,c_o h_r}$ and $\dist(x, \partial E^{t} ) \le c_o h_r$ for $x \in E^{t} \setminus E^{t,c_o h_r}$, and the $p$-Hardy inequality in Lemma \ref{lem:p-Hardy} it follows that
\begin{align*}
    &\int_{\tau_o-h}^{\tau_o} \int_{E^t} \big\vert D \big( \eta_{h}^{q+1} (u-u_\ast) \big) \big\vert^{p} \,\dx \dt
    \\ &\ \ =
    \int_{\tau_o-h}^{\tau_o} \int_{E^t} \left\vert D
    \left( \widetilde{\eta}^{q+1} \Big( \tfrac{\dist(x, \Omega \setminus E^{\tau_o})}{h_r} \Big) (u-u_\ast) \right) \right\vert^{p} \,\dx \dt
    \\ &\ \ \leq
    \frac{c}{h_r^p} \int_{\tau_o -h}^{\tau_o} \int_{E^{t} \setminus E^{\tau_o,2h_r}} \vert u-u_\ast \vert^{p} \,\dx \dt
    +c \int_{\tau_o -h}^{\tau_o} \int_{\Omega} \vert Du-Du_\ast
         \vert^{p} \,\dx\dt
  \\ &\ \ \leq
       c \int_{\tau_o -h}^{\tau_o} \int_{E^{t} \setminus
       E^{t,c_o h_r}} \bigg(\frac{\vert u-u_\ast \vert}{\dist(x,\partial
       E^t)}\bigg)^p \,\dx \dt
    +c \int_{\tau_o -h}^{\tau_o} \int_{\Omega} \vert Du-Du_\ast
         \vert^{p} \,\dx\dt
    \\ &\ \leq
    c \int_{\tau_o -h}^{\tau_o} \int_{\Omega} \vert Du-Du_\ast \vert^{p} \,\dx\dt.
\end{align*}
Hence, we conclude that
\begin{equation*}
    \big\Vert \chi_{\Omega \times (\tau_o-h, \tau_o)} \eta_{h}^{q+1} (u-u_\ast) \big\Vert_{V^{p}(E)} \to 0
\qquad\mbox{ as $h \downarrow 0$.}
\end{equation*}
Using this in \eqref{estimate_for_II_h_after_taking_epsilon_limit} gives us that
\begin{align}\label{limsup_of_second_integral}
    \limsup_{h \downarrow 0} \mathrm{II}_{h} \le \int_{\Omega} \b[ u(\tau_o),u_\ast] \,\dx.
\end{align}
Finally, by combining \eqref{split_integral_in_proof_into_two}, \eqref{limit_of_first_split_integral} and \eqref{limsup_of_second_integral}, we obtain that
\begin{align*}
    \limsup_{h \downarrow 0} \tfrac{1}{h} \int_{\tau_o-h}^{\tau_o}
  \int_{E^t} \b[u(t),u_\ast] \,\dx\dt
  \le
  \int_{\Omega} \b[ u(\tau_o),u_\ast] \,\dx.
\end{align*}
Together with \eqref{liminf_integal_estimate_ofu(tau_o)}, this gives us the claim of the lemma.
\end{proof}

At this stage, we are finally ready to prove the left-sided continuity in time of variational solutions that possess a weak time derivative.
In particular, this applies to the solution constructed in Theorem \ref{thm:existence_in_general_domains} under the assumptions of Theorem \ref{continuity_in_time_for_general_domains}.
\begin{lemma}\label{lem:left-sided_continuity}
Let $p \geq \frac{(n+1)(q+1)}{n+q+1}$,
let $E$ be a noncylindrical domain fulfilling \eqref{ineq:lower_measure_bound}, and \eqref{one sided growth condition} with $r$ given by \eqref{definition_of_r_for_one_sided_growth}, and let $\tau_o \in (0,T)$ be arbitrary. Then for any representative of $u \in V^p_q(E) \cap L^{\infty}(0,T, L^{q+1}(\Omega, \mathds{R}^{N}))$ with $\partial_t \power{u}{q} \in (V^{p,0}(E))'$ satisfying \eqref{limit_of_representative} we have that
    \begin{align}\label{left_continuity_limit}
        \lim_{\tau \uparrow \tau_o} \Vert u(\tau) - u(\tau_o) \Vert_{L^{q+1}} = 0.
    \end{align}
\end{lemma}
\begin{proof}
Fix arbitrary $\tau \in (0, \tau_o)$ and $h>0$, let $\eta_{h}$ again be defined according to \eqref{def:definition_of_eta_h}, and estimate
    \begin{align}
        \Vert u(\tau) - u(\tau_o) \Vert_{L^{q+1}(\Omega, \mathds{R}^{N})} 
        &\le
        \Vert \eta_h u(\tau) - \eta_h u(\tau_o) \Vert_{L^{q+1}(\Omega, \mathds{R}^{N})} \nonumber \\
        &\phantom{=}
        + \Vert  (1- \eta_h) (u(\tau) - u_\ast) \Vert_{L^{q+1}(\Omega, \mathds{R}^{N})} \nonumber \\
        &\phantom{=}
        + \Vert  (1- \eta_h) (u(\tau_o) - u_\ast) \Vert_{L^{q+1}(\Omega, \mathds{R}^{N})} \nonumber \\
        &=: \mathrm{I}_{h}(\tau) + \mathrm{II}_{h}(\tau) + \mathrm{III}_{h}.
        \label{ineq:triangle_inequality_for_difference_u(tau)_and_u(tau_o)}
    \end{align}
    First, since the domain $E$ is relatively open, for any $h>0$ there exists $\epsilon > 0$, such that $E^{\tau_o,h} \times [\tau_o-\epsilon, \tau_o] \subset E$.
    Thus, by the same argument as for \eqref{claim_inclusion} in the proof of Lemma \ref{lem:integral_limit_for_representative_of_u}, we obtain that
    \begin{align}\label{eq:limit_I_h}
        \lim_{\tau \uparrow \tau_o} \mathrm{I}_{h} = 0.
    \end{align}
    Next, for $0 < \lambda < \min\{ \tau, \tau_o-\tau\}$, consider the cut-off function $\xi_{\lambda} \in C_{0}^{0,1}((0,T))$ such that
    \begin{equation*}
    \xi_{\lambda}(t) =
    \left\{
    \begin{array}{ll} 
    \tfrac{1}{\lambda}(t-\tau+\lambda), & \text{for } t \in [\tau-\lambda, \tau), \\[5pt]
    1, & \text{for } t \in [\tau, \tau_o - \lambda), \\[5pt]
    \tfrac{1}{\lambda}(\tau_o - t), & \text{for } t \in [\tau_o-\lambda, \tau_o), \\[5pt]
    0, & \text{else}.
    \end{array}
    \right.
    \end{equation*}
    Since $(1-\eta_h)^q \in C^{0,1}(E)$ and thus $\partial_t\power{[(1-\eta_h)u]}{q} = (1-\eta_h)^q \partial_t\power{u}{q} \in (V^{p,0}(E))'$, taking Remark \ref{rem:int_by_parts_measure_dens} into account, we are allowed to use the integration by parts formula from Theorem~\ref{thm:ineq_integration_by_parts_formula} with $(1-\eta_h)u$ in
    place of $u$, $v\equiv(1-\eta_h)u_\ast$, and $\xi_\lambda$ as the cut-off function in time, with the result
    \begin{align}
      \tfrac1\lambda \int_{\tau-\lambda}^\tau \int_\Omega
      &(1-\eta_h)^{q+1}\b[u,u_\ast]\,\dx\dt
      \nonumber \\ &\le
      \tfrac1\lambda\int_{\tau_o-\lambda}^{\tau_o} \int_\Omega
      (1-\eta_h)^{q+1}\b[u,u_\ast]\,\dx\dt
      \label{Int_by_parts} \\ &\phantom{=}+
      \big\langle \partial_t \power{[(1-\eta_h)u]}{q}, \xi_\lambda(1-\eta_h)(u_\ast-u) \big\rangle.
      \nonumber
    \end{align}
    The last term can be estimated by
    \begin{align*}
      \big\langle &\partial_t \power{[(1-\eta_h)u]}{q}, \xi_\lambda(1-\eta_h)(u_\ast-u) \big\rangle
      \\ &=
      \big\langle \partial_t \power{u}{q}, \xi_\lambda(1-\eta_h)^{q+1}(u_\ast-u) \big\rangle\\
      &\le
      \big\| \partial_t\power{u}{q} \big\|_{(V^{p,0}(E))'}
      \big\| \chi_{\Omega\times(\tau-\lambda,\tau_o)}(1-\eta_h)^{q+1}(u_\ast-u) \big\|_{V^{p,0}(E)}.
    \end{align*}
    Moreover, we use
    Lemma~\ref{lem:integral_limit_for_representative_of_u} with $u$ replaced by $(1-\eta_h)u$ and boundary values $(1-\eta_h)u_\ast \in W^{1,p}(\Omega,\R^N) \cap L^{q+1}(\Omega,\R^N)$ in the definition of the space $V^p_q(E)$ to pass to the limit
    $\lambda\downarrow0$ in~\eqref{Int_by_parts} and arrive at
    \begin{align*}
      \int_\Omega &(1-\eta_h)^{q+1}\b[u(\tau),u_\ast]\,\dx\\
      &\le
         \int_\Omega(1-\eta_h)^{q+1}\b[u(\tau_o),u_\ast]\,\dx\\
       &\phantom{=}+
        \big\| \partial_t\power{u}{q} \big\|_{(V^{p,0}(E))'}
        \big\|\chi_{\Omega\times(\tau,\tau_o)}(1-\eta_h)^{q+1}(u_\ast-u) \big\|_{V^{p,0}(E)}. 
    \end{align*}
    Next, we let $\tau\uparrow\tau_o$ and observe that the last term
    vanishes in the limit. Therefore, we obtain that
    \begin{align*}
      \limsup_{\tau\uparrow\tau_o}\int_\Omega(1-\eta_h)^{q+1}\b[u(\tau),u_\ast]\,\dx
      \le
      \int_\Omega(1-\eta_h)^{q+1}\b[u(\tau_o),u_\ast]\,\dx.
    \end{align*}
    Now, we recall the definition of $\mathrm{II}_h(\tau)$ and use in turn Lemma~\ref{lem:technical_lemma_2}, H\"older's inequality, and Lemma~\ref{lem:technical_lemma}.
    In this way, we deduce that
    \begin{align*}
      \big[\mathrm{II}_h(\tau)\big]^{q+1}
      &\le
        c\int_\Omega (1-\eta_h)^{q+1} \Big( |u(\tau)|^{\frac{q+1}{2}} + |u_\ast|^{\frac{q+1}{2}} \Big) \Big| \power{u(\tau)}{\frac{q+1}{2}} - \power{u_\ast}{\frac{q+1}{2}} \Big|\,\dx\\
      &\le
        c\bigg(\int_\Omega (1-\eta_h)^{q+1} \big(|u(\tau)|^{q+1}+|u_\ast|^{q+1}\big)\,\dx\bigg)^{\frac12}
        \\ &\phantom{=} \cdot
        \bigg(\int_\Omega (1-\eta_h)^{q+1} \Big| \power{u(\tau)}{\frac{q+1}{2}} - \power{u_\ast}{\frac{q+1}{2}} \Big|^2 \,\dx\bigg)^{\frac12}\\
      &\le
        c\Big(\|u\|_{L^\infty(0,T;L^{q+1}(\Omega,\R^N))} + \|u_\ast\|_{L^{q+1}(\Omega,\R^N)}\Big)^{\frac{q+1}{2}}
        \\ &\phantom{=} \cdot
        \bigg(\int_\Omega(1-\eta_h)^{q+1}\b[u(\tau),u_\ast]\,\dx\bigg)^{\frac12}.
    \end{align*}
    Combining the two preceding formulas, we obtain that
    \begin{align*}
      \limsup_{\tau\uparrow\tau_o} \mathrm{II}_h(\tau)
      &\le
      c \Big(\|u\|_{L^\infty(0,T;L^{q+1}(\Omega,\R^N))} + \|u_\ast\|_{L^{q+1}(\Omega,\R^N)}\Big)^{\frac{1}{2}}
      \\ &\phantom{=} \cdot
      \bigg(\int_\Omega(1-\eta_h)^{q+1}\b[u(\tau_o),u_\ast]\,\dx\bigg)^{\frac1{2(q+1)}}.
    \end{align*}
    Using this estimate together with \eqref{eq:limit_I_h}
    in \eqref{ineq:triangle_inequality_for_difference_u(tau)_and_u(tau_o)} yields
\begin{align*}
    \limsup_{\tau \uparrow \tau_o}{} &\Vert u(\tau) - u(\tau_o) \Vert_{L^{q+1}(\Omega, \mathds{R}^{N})}
    \\ &\le
    c \Big(\|u\|_{L^\infty(0,T;L^{q+1}(\Omega,\R^N))} + \|u_\ast\|_{L^{q+1}(\Omega,\R^N)}\Big)^{\frac{1}{2}}
    \\ &\phantom{=} \cdot
    \bigg(\int_\Omega(1-\eta_h)^{q+1}\b[u(\tau_o),u_\ast]\,\dx\bigg)^{\frac1{2(q+1)}}
    \\ &\phantom{=} +
    \Vert  (1- \eta_h) (u(\tau_o) - u_\ast) \Vert_{L^{q+1}(\Omega, \mathds{R}^{N})}.    
\end{align*}
Note that the right-hand side of the preceding inequality vanishes as $h \downarrow 0$ by the dominated convergence theorem, since $(1-\eta_h)\to0$ a.e.~in $E^{\tau_o}$, and $u(\tau_o)=u_\ast$ a.e.~in $\Omega\setminus E^{\tau_o}$, as well as $u(\tau_o) \in L^{q+1}(\Omega,\R^N)$ for any $\tau_o \in (0,T)$ by \eqref{limit_of_representative}.
This proves the claim of the lemma.
\end{proof}

\subsection{Localization in time}
\label{sec:localization}
We develop a localized version of the variational inequality on the noncylindrical subdomains $E \cap (\Omega \times (\tau_o, \tau))$.
\begin{lemma}
    Let $p \geq \frac{(n+1)(q+1)}{n+q+1}$, and assume that $f$ is a variational integrand satisfying \eqref{eq:integrand}, and that $E$ a noncylindrical domain satisfying \eqref{ineq:lower_measure_bound}, and \eqref{one sided growth condition} with $r$ given by \eqref{definition_of_r_for_one_sided_growth}.
    Then, any solution to the variational inequality \eqref{eq:variational_inequality} with $\partial_t \power{u}{q} \in (V^{p,0}(E))'$ also fulfills the localized variational inequality 
    \begin{align}
        \iint_{E\cap (\Omega \times (\tau_o,\tau))} &f(x,u,Du) \,\dx\dt
        + \int_{E^\tau}	\b[u(\tau),v(\tau)] \,\dx
	\nonumber \\ & \leq
	\iint_{E\cap (\Omega \times (\tau_o,\tau))} f(x,v,Dv) + \partial_t v \cdot \big( \power{v}{q} - \power{u}{q} \big) \,\dx\dt
    \label{ineq:localized_variational_inequality} \\ &\phantom{=}
	+\int_{E^{\tau_o}} \b[u(\tau_o),v(\tau_o)] \,\dx
    \nonumber
    \end{align}
    for any $0 \le \tau_o < \tau <T$ and any comparison map $v \in V_{q}^{p}(E \cap (\Omega \times
     (\tau_o-\epsilon_o,\tau)))$ satisfying $\partial_t v \in
     L^{q+1}(\Omega \times (\tau_o-\epsilon_o,\tau), \R^N)$ for
     some $\epsilon_o>0$.
\end{lemma}

\begin{proof}
Let $\tau_o > 0$ and $\tau \in (\tau_o, T)$.
Since $u-u_\ast \in V^{p,0}_q(E \cap \Omega_\tau)$, by Corollary \ref{cor:C_0_inf_is_dense_in_V_q_p,0} there exists a sequence $(\phi_\ell)_{\ell \in \N}$ of smooth functions $\phi_\ell \in C^\infty_0(E \cap \Omega_\tau,\R^N)$ such that $\phi_\ell \to u-u_\ast$ strongly in $V^{p,0}_q(E \cap \Omega_\tau)$ as $\ell \to \infty$.
Now, we fix $\ell \in \N$ and let $\epsilon_\ell\in \big( 0,\frac13\tau_o \big)$ be so small that $\spt(\phi_\ell) \subset \Omega \times (\epsilon_\ell,\tau-\epsilon_\ell)$.
For $0<\epsilon< \min \big\{ \eps_\ell, \frac12\epsilon_o \big\}$, we introduce the cut-off functions in time
\begin{equation*}
	\zeta_{\epsilon}(t) =
	\left\{
	\begin{array}{cl}
		\frac{t}{\epsilon}, & \text{for } t \in [0, \epsilon), \\[5pt]
		1, & \text{for } t \in [\epsilon, \tau_o - \epsilon], \\[5pt]
		\frac{\tau_o - t}{\epsilon}
		& \text{for } t \in (\tau_o - \epsilon, \tau_o), \\[5pt]
		0, &\text{for } t \in [\tau_o, \tau],
	\end{array}
	\right.
\end{equation*}
and
\begin{equation*}
	\xi_{\epsilon}(t) =
	\left\{
	\begin{array}{cl}
		0, & \text{for } t \in [0, \tau_o-2\eps], \\[5pt]
		\frac{t-\tau_o + 2\epsilon}{\epsilon}, & \text{for } t \in (\tau_o - 2\epsilon, \tau_o - \epsilon), \\[5pt]
		1, & \text{for } t \in [\tau_o - \epsilon, \tau],
	\end{array}
	\right.
\end{equation*}
and consider a map $v \in V_{q}^{p}(E \cap (\Omega \times
(\tau_o-2\epsilon_o,\tau)))$ with $\partial_t v \in L^{q+1}(\Omega
\times (\tau_o-2\epsilon_o,\tau), \R^N)$.
Then, setting
\begin{align*}
	w_{\epsilon,\ell}
	:=
	(1-\xi_\epsilon) (u_\ast + \phi_\ell) + \xi_\epsilon v
	=
	u_\ast + (1-\xi_\epsilon) \phi_\ell + \xi_\epsilon (v-u_\ast),
\end{align*}
we find that $w_{\epsilon,\ell} \in V^p_q(E)$ satisfies $\partial_t w_{\epsilon,\ell} \in L^{q+1}(\Omega_\tau,\R^N)$ and $w_{\epsilon,\ell}(0) = u_\ast \in L^{q+1}(\Omega,\R^N)$.
Therefore, $w_{\epsilon,\ell}$ is an admissible comparison map in the
variational inequality satisfied by $u$, which holds for any $\tau \in
(0,T)$ provided that we choose suitable representatives of $u$ and
$v$, as a consequence of Lemma~\ref{lem:left-sided_continuity}.
For the term involving the time derivative, by the integration by parts formula in Theorem \ref{thm:ineq_integration_by_parts_formula}, which is applicable due to Remark~\ref{rem:int_by_parts_measure_dens}, and since $w_{\epsilon,\ell} = u_\ast$ in $(0,\epsilon)$ is independent of time, $\zeta_\epsilon \equiv 1$ in $[\epsilon,\tau_o - \epsilon]$, and $w_{\epsilon,\ell} = v$ in $(\tau_o - \epsilon,\tau)$, we conclude that
\begin{align*}
	&\iint_{\Omega_\tau} \partial_t w_{\epsilon,\ell} \cdot \big( \power{w_{\epsilon,\ell}}{q} - \power{u}{q} \big) \,\dx\dt
	\\ &=
	\iint_{\Omega_\tau} \partial_t w_{\epsilon,\ell} \cdot \zeta_\epsilon \big( \power{w_{\epsilon,\ell}}{q} - \power{u}{q} \big) \,\dx\dt
	+
	\iint_{\Omega \times (\tau_o - \epsilon,\tau)} (1-\zeta_\epsilon) \partial_t v \cdot \big( \power{v}{q} - \power{u}{q} \big) \,\dx\dt
	\\ &\leq
	\big\langle \partial_t \power{u}{q}, \zeta_\epsilon (w_{\epsilon,\ell} - u) \big\rangle
	- \iint_{\Omega_\tau} \zeta^{\prime}_\epsilon \b[u(t),w_{\epsilon,\ell}(t)] \,\dx\dt
	\\ &\phantom{=} +
	\iint_{\Omega \times (\tau_o-\epsilon,\tau)} (1-\zeta_\epsilon) \partial_t v \cdot \big( \power{v}{q} - \power{u}{q} \big) \,\dx\dt
	\\ &=
	\big\langle \partial_t \power{u}{q}, \zeta_\epsilon (w_{\epsilon,\ell} - u) \big\rangle
	- \bint_0^\epsilon \int_\Omega \b[u(t),u_\ast] \,\dx\dt
	+ \bint_{\tau_o - \epsilon}^{\tau_o} \int_\Omega \b[u(t),v(t)] \,\dx\dt
	\\ &\phantom{=} +
	\iint_{\Omega \times (\tau_o - \epsilon,\tau)} (1-\zeta_\epsilon) \partial_t v \cdot \big( \power{v}{q} - \power{u}{q} \big) \,\dx\dt.
\end{align*}
Therefore, altogether we infer
\begin{align}
	\iint_{\Omega_\tau} &f(x,u,Du) \,\dx\dt + \int_\Omega \b[u(\tau),v(\tau)] \,\dx
	\nonumber \\ &\leq
	\iint_{\Omega_\tau} f(x,w_{\epsilon,\ell},Dw_{\epsilon,\ell}) \,\dx\dt
	+ \iint_{\Omega_\tau} \partial_t w_{\epsilon,\ell} \cdot \big( \power{w_{\epsilon,\ell}}{q} - \power{u}{q} \big) \,\dx\dt
	\nonumber \\ &\phantom{=}
	+ \int_\Omega \b[u_o,u_\ast] \,\dx.
	\nonumber \\ &\leq
	\iint_{\Omega_\tau} f(x,w_{\epsilon,\ell},Dw_{\epsilon,\ell}) \,\dx\dt
	+ \iint_{\Omega \times (\tau_o-\epsilon,\tau)} (1-\zeta_\epsilon) \partial_t v \cdot \big( \power{v}{q} - \power{u}{q} \big) \,\dx\dt
	\label{eq:aux-localization} \\ &\phantom{=}
	+ \big\langle \partial_t \power{u}{q}, \zeta_\epsilon (w_{\epsilon,\ell} - u) \big\rangle
	+ \int_\Omega \b[u_o,u_\ast] \,\dx
	- \bint_0^\epsilon \int_\Omega \b[u(t),u_\ast] \,\dx\dt
	\nonumber \\ &\phantom{=}
	+ \bint_{\tau_o - \epsilon}^{\tau_o} \int_\Omega \b[u(t),v(t)] \,\dx\dt.
	\nonumber
\end{align}
By the dominated convergence theorem, for the first three terms on the right-hand side of \eqref{eq:aux-localization} we obtain that
\begin{align*}
	\lim_{\epsilon \downarrow 0}
	\bigg(
	\iint_{\Omega_\tau} &f(x,w_{\epsilon,\ell},Dw_{\epsilon,\ell}) \,\dx\dt
	+ \iint_{\Omega \times (\tau_o-\epsilon,\tau)} (1-\zeta_\epsilon) \partial_t v \cdot \big( \power{v}{q} - \power{u}{q} \big) \,\dx\dt
	\\ &\phantom{=}
	+\big\langle \partial_t \power{u}{q}, \zeta_\epsilon (w_{\epsilon,\ell} - u) \big\rangle
	\bigg)
	\\ &=
	\iint_{\Omega_{\tau_o}} f(x,\phi_\ell + u_\ast,D(\phi_\ell + u_\ast)) \,\dx\dt
	+ \iint_{\Omega \times (\tau_o,\tau)} f(x,v,Dv) \,\dx\dt
	\\ &\phantom{=}
	+ \iint_{\Omega \times (\tau_o,\tau)} \partial_t v \cdot \big( \power{v}{q} - \power{u}{q} \big) \,\dx\dt
	+\big\langle \partial_t \power{u}{q}, \chi_{\Omega_{\tau_o}} (\phi_\ell + u_\ast - u) \big\rangle.
\end{align*}
Next, since $u$ attains its initial values $u_o$ in the $L^{q+1}$-sense, see Remark \ref{rem:initial_condition_convergence}, we find that
\begin{align*}
	\lim_{\epsilon \downarrow 0}
	\bigg(
	\int_\Omega \b[u_o,u_\ast] \,\dx
	- \bint_0^\epsilon \int_\Omega \b[u(t),u_\ast] \,\dx\dt
	\bigg)
	= 0.
\end{align*}
Further, since $u$ is left-sided continuous in time by Lemma \ref{lem:left-sided_continuity}, we have that
\begin{align*}
	\lim_{\epsilon \downarrow 0}
	\bint_{\tau_o - \epsilon}^{\tau_o} \int_\Omega \b[u(t),v(t)] \,\dx\dt
	=
	\int_\Omega \b[u(\tau_o), v(\tau_o)] \,\dx
\end{align*}
for any $\tau_o \in (0,T)$.
Using the preceding computations in \eqref{eq:aux-localization} leads us to
\begin{align*}
	\iint_{\Omega_\tau} &f(x,u,Du) \,\dx\dt + \int_\Omega \b[u(\tau),v(\tau)] \,\dx
	\\ &\leq
	\iint_{\Omega_{\tau_o}} f(x,\phi_\ell + u_\ast,D(\phi_\ell + u_\ast)) \,\dx\dt
	+ \iint_{\Omega \times (\tau_o,\tau)} f(x,v,Dv) \,\dx\dt
	\\ &\phantom{=}
	+ \iint_{\Omega \times (\tau_o,\tau)} \partial_t v \cdot \big( \power{v}{q} - \power{u}{q} \big) \,\dx\dt
	+\big\langle \partial_t \power{u}{q}, \chi_{\Omega_{\tau_o}} (\phi_\ell + u_\ast - u) \big\rangle
	\\ &\phantom{=}
	+\int_\Omega \b[u(\tau_o), v(\tau_o)] \,\dx
\end{align*}
for any~$\tau_o \in (0,\tau)$.
Finally, since $\phi_\ell + u_\ast \to u$ in $V^p_q(E)$ and $f$ satisfies the local Lipschitz condition \eqref{ineq:Lipschitz_condition}, we conclude the claim of the lemma.
\end{proof}

\subsection{Right-sided continuity in time}
\label{sec:right-sided-continuity}
Using the localized variational inequality, we are now able to prove the continuity in time from the right-hand side. 
Combining Lemma \ref{lem:left-sided_continuity} with Lemma \ref{lem:right-sided_continuity} proves Theorem \ref{continuity_in_time_for_general_domains}.

\begin{lemma}\label{lem:right-sided_continuity}
Let $p \geq \frac{(n+1)(q+1)}{n+q+1}$, let the variational integrand $f\colon \Omega \times \mathds{R}^{N} \times \mathds{R}^{Nn} \to \R$ satisfy \eqref{eq:integrand}, and assume that the noncylindrical domain $E$ fulfills~\eqref{ineq:lower_measure_bound}, and \eqref{one sided growth condition} with $r$ given by \eqref{definition_of_r_for_one_sided_growth}.
Then any variational solution $u \in V_{q}^{p}(E)$ in the sense of Definition \ref{Definition:variational_solution} with $\partial_t \power{u}{q} \in (V^{p,0}(E))'$ satisfies
\begin{align}\label{right_continuity_limit}
	\lim_{\tau \downarrow \tau_o} \Vert u(\tau) - u(\tau_o) \Vert_{L^{q+1}(\Omega, \mathds{R}^{N})} = 0
\end{align}
for all $\tau_o \in (0,T)$.
\end{lemma}
\begin{proof}
Let $0 < \tau_o <T$ and $\epsilon>0$.
By the relative openness of $E$ there exists $\delta >0$ depending on $\epsilon$ such that
    \begin{align}\label{time_slice_set_inclusion}
        E^{\tau_o,\epsilon} \times (\tau_o - \delta, \tau_o+2\delta) \subset E.
    \end{align}
Further, we introduce a standard mollifier $\phi \in C_{0}^{\infty}(B_{1}, \mathds{R}_{\ge 0})$ and set $\phi_{\epsilon}:= \epsilon^{-n} \phi(\tfrac{x}{\epsilon})$.
For the sake of notation we write $u_{\tau_o}:= u(\tau_o)$.
We choose a representative of $u$
satisfying~\eqref{limit_of_representative}, which implies in
particular $u_{\tau_o} = u_\ast$ a.e.~in $\Omega \setminus E^{\tau_o}$.
Then, we define 
    \begin{align*}
        u_{\tau_o}^{(\epsilon)} := u_{\ast} + ((u_{\tau_o} - u_{\ast})\chi_{E^{\tau_o, 2 \epsilon}}) \ast \phi_{\epsilon}.
    \end{align*}
    Due to \eqref{time_slice_set_inclusion}, we obtain for all $t \in (\tau_o - \delta, \tau_o +2\delta)$ that $\spt \big( u_{\tau_o}^{(\epsilon)} - u_{\ast} \big) \subset E^t$. 
    Furthermore, $u_{\tau_o}^{(\epsilon)} - u_\ast$ is an element of the space $C_{0}^{\infty}(E^{\tau_o,\epsilon}, \mathds{R}^{N})$ and $u_{\tau_o}^{(\epsilon)}$ converges to $u(\tau_o)$ in the space $L^{q+1}(\Omega, \mathds{R}^{N})$ as $\epsilon \downarrow 0$.
Let $\zeta \in C_{0}^{0,1}((\tau_o - \delta, \tau_o +2\delta))$ be a
non-negative cut-off function with $\zeta \equiv 1$ on $[\tau_o, \tau_o
+ \delta]$. Hence, we are allowed to use
\begin{align*}
w(x,t):= u_\ast(x)+\zeta(t) \big(u_{\tau_o}^{(\epsilon)}(x)-u_\ast(x)\big)
\end{align*}
as a comparison map in the localized version \eqref{ineq:localized_variational_inequality} of the variational inequality. Therefore, by omitting the non-negative term on the left-hand side of the localized variational inequality, we deduce that
    \begin{align}\label{ineq:first_inequality_developed_in_initial_condition_proof}
        \int_{\Omega} \b \Big[u(\tau),u_{\tau_o}^{(\epsilon)} \Big] \,\dx 
	&\leq
	\iint_{\Omega \times (\tau_o, \tau)} f \Big(x,u_{\tau_o}^{(\epsilon)},Du_{\tau_o}^{(\epsilon)} \Big) \,\dx\dt
	+ \int_{\Omega} \b \Big[ u_{\tau_o},u_{\tau_o}^{(\epsilon)} \Big] \,\dx  \nonumber\\
    &=
    (\tau - \tau_o) \int_{\Omega} f \Big(x,u_{\tau_o}^{(\epsilon)},Du_{\tau_o}^{(\epsilon)} \Big) \,\dx
    + \int_{\Omega} \b \Big[u_{\tau_o},u_{\tau_o}^{(\epsilon)} \Big] \,\dx
    \end{align}
    for every~$\tau \in (\tau_o,\tau_o+\delta)$.
By the definition of $\b[\cdot, \cdot]$ we obtain the inequality
    \begin{align}\label{identity:inequality_for_b[]}
        \b[u(\tau),&u_{\tau_o}]
        \nonumber \\
        &\leq
        \b \Big[u(\tau),u_{\tau_o}^{(\epsilon)} \Big]
        + \tfrac{1}{q+1} \Big( |u_{\tau_o}|^{q+1} - \big|u_{\tau_o}^{(\epsilon)} \big|^{q+1} \Big)
        + |u(\tau)|^q \big| u_{\tau_o}^{(\epsilon)}  - u_{\tau_o} \big|.
    \end{align}
Now, combining \eqref{identity:inequality_for_b[]} and \eqref{ineq:first_inequality_developed_in_initial_condition_proof}, and applying H\"older's inequality gives us that
    \begin{align*}
        \int_{\Omega} &\b[u(\tau),u_{\tau_o}] \,\dx
        \\ &\leq
        \int_{\Omega} \b \Big[u(\tau),u_{\tau_o}^{(\epsilon)} \Big] \,\dx 
        + \tfrac{1}{q+1} \int_{\Omega} \Big( |u_{\tau_o}|^{q+1} - \big|u_{\tau_o}^{(\epsilon)} \big|^{q+1} \Big) \,\dx \\
        & \phantom{=}
        + \bigg( \int_{\Omega} |u(\tau)|^{q+1} \,\dx \bigg)^{\frac{q}{q+1}} 
        \bigg( \int_{\Omega} \big\vert u_{\tau_o}^{(\epsilon)} - u_{\tau_o} \big\vert^{q+1} \,\dx \bigg)^{\frac{1}{q+1}}
         \\&\le
        (\tau - \tau_o) \int_{\Omega} f \Big(x,u_{\tau_o}^{(\epsilon)},Du_{\tau_o}^{(\epsilon)} \Big) \,\dx
        + \int_{\Omega} \b \Big[u_{\tau_o},u_{\tau_o}^{(\epsilon)} \Big] \,\dx \\
        &\phantom{=}
        + \tfrac{1}{q+1} \int_{\Omega} \Big( |u_{\tau_o}|^{q+1} - \big|u_{\tau_o}^{(\epsilon)} \big|^{q+1} \Big) \,\dx \\
        &\phantom{=}
        +\Vert u \Vert_{L^{\infty}(0,T;L^{q+1}(\Omega, \mathds{R}^{N}))}^q
        \bigg( \int_{\Omega} \big\vert u_{\tau_o}^{(\epsilon)} - u_{\tau_o} \big\vert^{q+1} \,\dx \bigg)^{\frac{1}{q+1}}.
    \end{align*}
Letting $\tau \downarrow \tau_o$, we deduce that
    \begin{align*}
         \limsup_{\tau \downarrow \tau_o} &\int_{\Omega} \b[u(\tau) , u_{\tau_o} ] \,\dx \d\tau \\
        &\le
	  \int_{\Omega} \b\Big[u_{\tau_o} , u_{\tau_o}^{(\epsilon)}\Big] \,\dx 
        + \tfrac{1}{q+1} \int_{\Omega} \Big( |u_{\tau_o}|^{q+1} - \big|u_{\tau_o}^{(\epsilon)} \big|^{q+1} \Big) \,\dx \\
        &\phantom{=}
         +\Vert u \Vert_{L^{\infty}(0,T;L^{q+1}(\Omega, \mathds{R}^{N}))} ^q 
         \bigg( \int_{\Omega} \big\vert u_{\tau_o}^{(\epsilon)} - u_{\tau_o} \big\vert^{q+1} \,\dx \bigg)^{\frac{1}{q+1}}.
    \end{align*}
    Since $u_{\tau_o}^{(\epsilon)} \to u_{\tau_o}$ in $L^{q+1}(\Omega,\mathds{R}^{N})$ as $\epsilon \downarrow 0$, we obtain that 
    \begin{align}\label{limit_in_initial_condition_proof}
           \limsup_{\tau \downarrow \tau_o} \int_{\Omega}	\b[u(\tau) , u_{\tau_o}] \,\dx \d\tau =0.
     \end{align}
Combining Lemma~\ref{lem:technical_lemma_2} and Lemma~\ref{lem:technical_lemma} gives us the estimate
    \begin{align*}
        \vert u(\tau) - u_{\tau_o} \vert^{q+1}
        &\le 
        c(q) \Big(\vert u(\tau) \vert^{\frac{q+1}{2}} + \vert u_{\tau_o} \vert^{\frac{q+1}{2}} \Big)
        \Big\vert \power{u(\tau)}{\frac{q+1}{2}} - \power{u_{\tau_o}}{\frac{q+1}{2}} \Big\vert \nonumber \\
        &\le 
        c(q) \Big(\vert u(\tau) \vert^{\frac{q+1}{2}} + \vert u_{\tau_o} \vert^{\frac{q+1}{2}} \Big)
        \b[u(\tau), u_{\tau_o}]^{\frac{1}{2}}.
    \end{align*}
    Finally, integrating the preceding inequality over $\Omega$ and applying H\"older's inequality, we conclude that
        \begin{align*}
             \Vert u(\tau) &-u_{\tau_o} \Vert_{L^{q+1}(\Omega, \mathds{R}^{N})}^{q+1}  
             \\ &\le  
             c(q)  \bigg( \int_{\Omega}	\vert u(\tau) \vert^{q+1} + \vert u_{\tau_o} \vert^{q+1} \,\dx\bigg)^{\frac{1}{2}}
            \bigg( \int_{\Omega} \b[u(\tau), u_{\tau_o}] \,\dx \bigg)^{\frac{1}{2}} .
        \end{align*}
    The claim follows by taking the limit $\tau \downarrow \tau_o$ and using \eqref{limit_in_initial_condition_proof}.
\end{proof}

\bigskip
\textbf{Acknowledgements.}
This research was funded in whole or in part by the Austrian Science Fund (FWF) 10.55776/J4853 and 10.55776/P36295. Jarkko Siltakoski was supported by the Magnus Ehrnrooth and Emil Aaltonen foundations.
The second and third authors are grateful to the Department of
Mathematics at the Paris Lodron University of Salzburg, where this
research was initiated, for their
hospitality during their stay.
For the purpose of open access, the authors have applied a CC BY public copyright license to any Author Accepted Manuscript (AAM) version arising from this submission.

\end{document}